\documentclass[11pt,a4paper]{amsart}
\usepackage{amsmath}
\usepackage{amsfonts}
\usepackage{amssymb}
\usepackage{graphicx}
\usepackage[utf8]{inputenc}
\usepackage{pdflscape}
\usepackage{float}
\usepackage{mathtools}
\usepackage[colorlinks=true, linkcolor=red, citecolor=blue]{hyperref}
\usepackage{epsfig}
\usepackage{tikz}
\usepackage{mathrsfs}
\usepackage{comment}

\numberwithin{equation}{section}
\makeatletter
\@namedef{subjclassname@2010}{\textup{2020} Mathematics Subject Classification}
\makeatother

\newtheorem{theorem}{Theorem}[section]
\newtheorem{definition}{Definition}[section]
\newtheorem{lemma}{Lemma}[section]
\newtheorem{proposition}{Proposition}[section]
\newtheorem{corollary}{Corollary}[section]
\newtheorem{remark}{Remark}[section]

\begin{document}
\title[KdV equation: global well-posedness on star-graphs]{A Nonhomogeneous Boundary-Value Problem For The Nonlinear KdV Equation on Star Graphs}
\author[Capistrano--Filho]{Roberto de A. Capistrano--Filho}
\author[Parada]{Hugo Parada$^*$}
\author[da Silva]{Jandeilson Santos da Silva}

\address{Departamento de Matem\'atica, Universidade Federal de Pernambuco, S/N Cidade Universit\'aria, 50740-545, Recife (PE), Brazil}

\email{\url{roberto.capistranofilho@ufpe.br}}
\email{\url{jandeilson.santos@ufpe.br}}

\address{Universit\'e de Lorraine, CNRS, INRIA, IECL, F-54000 Nancy, France.}
\email{\url{hugo.parada@inria.fr}}

\thanks{$^*$Corresponding author: \url{hugo.parada@inria.fr}}
\thanks{Capistrano–Filho was partially supported by CAPES/COFECUB grant number 88887.879175/2023-00, CNPq grant numbers 421573/2023-6 and 301744/2025-4, and Propg (UFPE). Parada is supported by Agence Nationale de la Recherche under the QuBiCCS project ANR-24-CE40-3008. Da Silva acknowledges support from CNPq.}	
	
\numberwithin{equation}{section}
	
\begin{abstract}
This paper investigates an initial-boundary value problem for the Korteweg--de Vries (KdV) equation on a finite equilateral star graph. We use the notion of $s$-compatibility, which extends the classical compatibility conditions to the coupled boundary conditions arising at the central vertex, inspired by the work of Bona, Sun, and Zhang \cite{BSZ 2003}. By combining analytical estimates, fixed-point arguments, and nonlinear interpolation on suitable augmented spaces, we establish global well-posedness for the nonlinear problem in $H^s$, $0\leq s\leq3$. At the critical indices $s=1/2,3/2,5/2$, the admissible data are defined through constrained interpolation; away from these indices, they are characterized by the corresponding pointwise compatibility conditions. 
The nonlinear compatibility condition at the central vertex is handled by introducing an auxiliary variable, which allows the interpolation argument to be carried out on a pair of linear spaces.
\end{abstract}
	
\date{\today}
\maketitle
	
\tableofcontents
	
	
\section{Introduction}	
In mathematics and physics, a quantum graph is a network composed of vertices connected by edges, where each edge carries a differential or pseudo-differential equation. When every edge is endowed with a natural metric, the structure is referred to as a metric graph. A common illustration is a power network: the transmission lines act as edges connecting transformer stations (vertices). In such a case, the governing differential equations may describe the voltage along each line, with boundary conditions at the vertices ensuring that the total current flowing into and out of each vertex sums to zero.

Beyond solving the differential equations in specific applications, the study of quantum graphs encompasses fundamental questions such as well-posedness, controllability, and identifiability. For example, controllability concerns determining the inputs required to steer the system toward a desired configuration, such as ensuring stable power distribution across a grid. Identifiability, on the other hand, focuses on optimal measurement placement—for instance, choosing pressure gauges in a water network to detect and locate potential leaks.

\subsection{Dispersive systems on graph structure}
In recent years, the study of nonlinear dispersive equations on metric graphs has attracted growing interest from mathematicians, physicists, chemists, and engineers (see \cite{BK, BlaExn08, BurCas01, K, Mug15} for detailed discussions and references). A fundamental setting for such investigations is the star graph $\mathcal{G}$, a metric graph consisting of $N$ half-lines $(0,\infty)$ meeting at a single common vertex $\nu = 0$. On each edge, one typically defines a nonlinear dispersive equation—most prominently, the nonlinear Schrödinger equation (see the works of Adami \textit{et al.} \cite{AdaNoj14, AdaNoj15} and Angulo and Goloshchapova \cite{AngGol17a, AngGol17b}).
The introduction of nonlinearities into dispersive models on such networks creates a rich framework for exploring soliton propagation, energy transfer, and nonlinear wave interactions. A central analytical difficulty arises at the vertex, where the graph may exhibit bifurcation or multi-bifurcation phenomena, especially in more intricate network topologies.

Other nonlinear dispersive systems on graphs have also been the subject of extensive study. For example, in the context of well-posedness theory, Cavalcante \cite{Cav} established local well-posedness for the Cauchy problem associated with the Korteweg–de Vries (KdV) equation on a metric star graph composed of one negative and two positive half-lines joined at a common vertex $\nu = 0$ (the so-called $\mathcal{Y}$-junction).

Another important example is the Benjamin–Bona–Mahony (BBM) equation: Bona and Cascaval \cite{bona} proved local well-posedness in the Sobolev space $H^1$, while Mugnolo and Rault \cite{Mugnolobbm} showed the existence of traveling-wave solutions on graphs. Using a different approach, Ammari and Crépeau \cite{AmCr1} derived results concerning well-posedness and stabilization of the BBM equation in star-shaped networks with bounded edges. Moreover, Mugnolo \textit{et al.} \cite{MugNoSe} provided a full characterization of boundary conditions under which the Airy-type evolution equation
$$u_t = \alpha u_{xxx} + \beta u_x, \quad \alpha \in \mathbb{R} \setminus \{0\}, \ \beta \in \mathbb{R},$$
generates contraction semigroups on star graphs. We also refer to the recent work on Airy dynamics on graphs with looping edges \cite{PavaMunoz24}, which provides a complementary perspective on this setting.
\medskip

Significant progress has also been achieved in the fields of control theory and inverse problems on networks. Ignat \textit{et al.} \cite{Ignat2011} examined inverse problems for the heat and Schrödinger equations on tree-like structures. In contrast, Baudouin and Yamamoto \cite{Baudouin} proposed a unified and simplified framework for coefficient identification problems. Furthermore, Ammari and Crépeau \cite{Ammari and Crepeau 2018}, Cerpa \textit{et al.} \cite{Cerpa, Cerpa1}, and Parada \textit{et al.} \cite{Parada2022a, Parada2022b} established stabilization and boundary controllability results for the KdV equation on star-shaped graphs, contributing substantially to the control-theoretic understanding of dispersive models on networks. Recently, in \cite{CapistranoParadaDaSilva25}, we studied the analytical structure of the entire functions associated with the KdV operator, addressing control problems involving mixed Dirichlet–Neumann boundary controls.

\subsection{Problem formulation and functional framework}
Let $\mathcal T$ be a star graph consisting of $N\geq2$ edges
$e_j$, each identified with an interval $I_j=(0,l_j)$, $l_j>0$,
and joined at a common central vertex.

Throughout the paper, we assume that the star graph is equilateral:
\[
l_1=\cdots=l_N=L.
\]
We retain the notation $l_j$ when convenient, with the understanding
that each edge is identified with $(0,L)$ and all endpoint
evaluations at $l_j$ are taken at $L$.
We study the local and global well-posedness of a system of $N$
nonlinear KdV equations posed on $\mathcal T$ and coupled through
conditions at the central vertex, namely,
\begin{align}\label{kdV}
\begin{cases}
\partial_t u_j+\partial_x u_j+\partial_x^3 u_j+u_j\partial_x u_j=0,&t \in (0,T),\ x \in (0,l_j),\ j=1,...,N\\
u_j(t,0)=u_1(t,0),&t\in (0,T),\ j=1,...,N\\
\displaystyle\sum_{j=1}^N\partial_x^2 u_j(t,0)=-\alpha u_1(t,0)-\frac{N}{3}u_1^2(t,0)+g_0(t),&t \in (0,T)\\				
u_j(t,l_j)=p_j(t),\ \ \ \ \ \partial_x u_j(t,l_j)=g_j(t),&t\in (0,T),\ j=1,...,N\\
u_j(0,x)=u_j^0(x),&x \in (0,l_j),\ j=1,...,N
\end{cases}
\end{align}
where $u=(u_1,\ldots,u_N)$ denotes the state of the system,
$g_0$, $g_j$, and $p_j$ are the prescribed boundary inputs,
and $\alpha>N/2$. The coupling conditions at the central node, including the nonlinear term
$-\frac{N}{3}u_1^2(t,0)$, are derived from the energy balance of the system
and discussed in detail in \cite{Parada2022b}. We represent this situation in Figure \ref{fig:network}. 
	\begin{figure}[H]
\centering
\begin{tikzpicture}[
    edge/.style={
        draw=black,
        line width=0.9pt,
        line cap=round
    },
    central vertex/.style={
        circle,
        draw=black,
        fill=white,
        inner sep=0pt,
        minimum size=2.6mm
    },
    external vertex/.style={
        circle,
        draw=black,
        fill=black,
        inner sep=0pt,
        minimum size=2.2mm
    },
    every node/.style={font=\small}
]

\def\EdgeLength{2.65}
\coordinate (v0) at (0,0);

\foreach \angle/\vertex/\edgetext in {
    70/v1/{e_1},
    40/v2/{e_2},
    15/v3/{e_3},
   -40/vNm/{e_{N-1}},
   -70/vN/{e_N}
}{
    \coordinate (\vertex) at (\angle:\EdgeLength);
    \draw[edge]
        (v0) -- (\vertex)
        node[pos=0.58,above,sloped] {$\edgetext$};
    \node[external vertex] at (\vertex) {};
}

\node at (1.75,-0.20) {$\vdots$};

\node[central vertex] at (v0) {};
\node[anchor=east] at (-0.18,0) {$x=0$};

\begin{scope}[shift={(3.65,-0.35)}]
    \node[central vertex] at (0,0) {};
    \node[anchor=west] at (0.28,0)
        {central controlled vertex};

    \node[external vertex] at (0,-0.55) {};
    \node[anchor=west] at (0.28,-0.55)
        {external controlled vertices};
\end{scope}

\end{tikzpicture}
\caption{Equilateral star graph with $N$ edges, each identified with
the interval $(0,L)$.}
\label{fig:network}
\end{figure}


We use the following notation throughout the paper:
	 \begin{itemize}
	 \item[i.] The vector $u$ is given by $u=(u_1,...,u_N) \in \mathbb{L}^2(\mathcal{T})=\displaystyle\prod_{j=1}^NL^2(0,l_j)$
	and initial data is $u^0=(u^0_1,...,u^0_N)$. 
	 \item[ii.] The inner product in $\mathbb{L}^2(\mathcal{T})$ will be given by  $$\left(u,z\right)_{\mathbb{L}^2(\mathcal{T})}=\sum_{j=1}^N\int_0^{l_j}u_jz_jdx,\ \ u,z\in \mathbb{L}^2(\mathcal{T}).$$
	Moreover,  $$\mathbb{H}^s(\mathcal{T})=\displaystyle\prod_{j=1}^NH^s(0,l_j), \ \ s \in \mathbb{R}.$$
	 \end{itemize}
	\begin{itemize}
		\item[iii.] Define also the following spaces:
		$$\mathbb{H}_0^k(\mathcal{T})=\displaystyle\prod_{j=1}^{N}H_0^k(0,l_j),\ k\in \mathbb{N}$$
		and
		$$H_r^k(0,l_j)=\left\{v \in H^k(0,l_j),\ \partial_x^iv(l_j)=0,\ 0\leq i\leq k-1\right\},\ k\in\mathbb{N},$$ 
		where $v^{(m)}=\frac{d^m v}{dx^m}$ and the index $r$ is related to the null right boundary conditions. In addition, 
		$$\mathbb{H}^k_r(\mathcal{T})=\prod_{j=1}^{N}H_r^k(0,l_j)\quad \text{and} \quad \|u\|_{\mathbb{H}^k_r(\mathcal{T})}^2=\sum_{j=1}^{N}\|u_j\|_{H^k(0,l_j)}^2, \quad \ k \in \mathbb{N}.$$
		\item[iv.] Consider the following characterization: 
		$$H_r^{-1}(0,l_j)=\left(H_r^1(0,l_j)\right)'$$ 
		as the dual space of $H_r^1(0,l_j)$ with respect to the pivot space $L^2(0,l_j)$ and $\mathbb{H}_r^{-1}$ denotes the cartesian product of $H_r^{-1}(0,l_j)$. 
\item[v.] For $k\in \mathbb{N}$, the derivative $\partial_x^ku$ means the row vector
\begin{align*}
\partial_x^k u:=(\partial_x^ku_1,...,\partial_x^ku_N).
\end{align*}
The derivative $\partial_t^ku$ is understood in the same way.
\item[vi.] Under the above notation, we have
\begin{align*}
	C^k\left([0,T];\mathbb{H}^s(\mathcal{T})\right)=\prod_{j=1}^NC^k\left([0,T];H^s(0,l_j)\right),\ k=0,1,2,... \text{ \ and \ }s \in \mathbb{R}
\end{align*}
as well as
\begin{align*}
	L^p\left(0,T;\mathbb{H}^s(\mathcal{T})\right)=\prod_{j=1}^NL^p\left(0,T;H^s(0,l_j)\right),\ p\geq 1 \text{ \ and \ }s \in \mathbb{R}.
\end{align*}
\item[vii.] We will consider the spaces
$$\mathbb{H}^k_e(\mathcal{T})=\left\{u=(u_1,...,u_N)\in \mathbb{H}^k_r(\mathcal{T});\ u_j(0)=u_1(0),\ j=1,...,N\right\}\ k \in \mathbb{N},$$
$$\mathbb{H}^3_c(\mathcal{T})=\left\{u=(u_1,...,u_N)\in \mathbb{H}^3(\mathcal{T});\ u_j(0)=u_1(0),\ j=1,...,N\right\},$$
and
$$\mathbb{B}_{s,T}=C([0,T],\mathbb{H}^s(\mathcal{T}))\cap L^2(0,T,\mathbb{H}^{s+1}(\mathcal{T})),$$
with the norm 
$$\displaystyle\|u\|_{\mathbb{B}_{s,T}}:=\|u\|_{C([0,T],\mathbb{H}^s(\mathcal{T}))}+\|u\|_{L^2(0,T,\mathbb{H}^{s+1}(\mathcal{T}))}=\max_{t\in [0,T]}\|u\|_{\mathbb{H}^s(\mathcal{T})}+\left(\int_0^T\|u(t,\cdot)\|_{\mathbb{H}^{s+1}}^2dt\right)^\frac{1}{2}.$$
When $s=0$ we write $\mathbb{B}_T$ instead of $\mathbb{B}_{0,T}$.
\item[viii.] We write $g=(g_1,\ldots,g_N)$ and $p=(p_1,\ldots,p_N)$.
\item[ix.] For $0\leq s\leq3$, we define
\[
\mathbb B^*_{s,T}
:=\mathbb B_{s,T}\cap H^{s/3}(0,T;\mathbb H^1(\mathcal T))
\]
and
\[
\mathcal X_{s,T}
:=\left\{v\in\mathbb B^*_{s,T}:
\partial_x^kv_j\in
L_x^\infty(0,l_j;H^{(s+1-k)/3}(0,T)),
\quad k=0,1,2,\ j=1,\ldots,N\right\}.
\]
This is a Banach space with the norm
\[
\|v\|_{\mathcal X_{s,T}}
:=\|v\|_{\mathbb B^*_{s,T}}
+\sum_{j=1}^N\sum_{k=0}^2
\|\partial_x^kv_j\|_{L_x^\infty(0,l_j;H^{(s+1-k)/3}(0,T))}.
\]
\item[x.] For $0\leq s\leq 3$ and $\lambda\geq 0$, define the ambient data space
\begin{align*}
\mathcal D_{s,T}^{\lambda}:={}&
\mathbb H^s(\mathcal T)
\times H^{\frac{s+\lambda-1}{3}}(0,T)
\times \left[H^{\frac{s}{3}}(0,T)\right]^N
\times \left[H^{\frac{s+\lambda+1}{3}}(0,T)\right]^N.
\end{align*}
For $d=(u^0,g_0,g,p)\in\mathcal D_{s,T}^{\lambda}$, set
\[
\|d\|_{s,\lambda}:={}
\|u^0\|_{\mathbb H^s(\mathcal T)}
+\|g_0\|_{H^{\frac{s+\lambda-1}{3}}(0,T)}
+\|g\|_{\left[H^{\frac{s}{3}}(0,T)\right]^N}+\|p\|_{\left[H^{\frac{s+\lambda+1}{3}}(0,T)\right]^N}.
\]
\end{itemize}

\subsection{Main results and structure of the article}
Our first main result concerns local well-posedness for \eqref{kdV}.
\begin{theorem}\label{main}
Let $T>0$, $L>0$, and $s\in[0,3]\setminus\left\{\frac12,\frac32,\frac52\right\}$ be given, and let
\begin{align*}
(u^0,g_0,g,p)\in{}&
\mathbb H^s(\mathcal T)
\times H^{\frac{s-1}{3}}(0,T)
\times \left[H^{\frac{s}{3}}(0,T)\right]^N
\times \left[H^{\frac{s+1}{3}}(0,T)\right]^N
\end{align*}
be data satisfying the $s$-compatibility conditions introduced in
Section~\ref{sec2}. Then there exists $T^*\in(0,T]$ such that the
IBVP \eqref{kdV} admits a unique solution $u\in\mathcal X_{s,T^*}$. Moreover, the existence time and the constants in the corresponding
local estimates may be chosen uniformly when the compatible data
range over a bounded subset of the above data space. On each such bounded subset, the data-to-solution map is continuous
into $\mathcal X_{s,T^*}$, with $T^*$ chosen uniformly on that subset.
\end{theorem}

The proof combines endpoint linear estimates, a fixed-point argument,
and interpolation on augmented spaces that account for the nonlinear
compatibility conditions.

We next establish global well-posedness in $H^s$, under the standing
assumption that the star graph is equilateral. For $0\leq s\leq3$, define the admissible physical data set by
\begin{equation}\label{eq:physical-interpolated-data}
\mathfrak X^\lambda_{s,T}
:=
\mathcal P\left(\widehat{\mathbb X}^\lambda_{s,T}\right)
\subset\mathcal D^\lambda_{s,T},
\end{equation}
constructed in Appendix~\ref{app:compatibility-interpolation} through
the augmented linear spaces
$\widehat{\mathbb X}^\lambda_{s,T}$ and the nonlinear
parametrization $\mathcal P$. For $d\in\mathfrak X^\lambda_{s,T}$, let
\begin{equation}\label{eq:chart-size}
\langle d\rangle_{s,\lambda}
:=
\inf\left\{
\|U\|_{\widehat{\mathbb X}^\lambda_{s,T}}:
U\in\widehat{\mathbb X}^\lambda_{s,T},
\ \mathcal P(U)=d
\right\}.
\end{equation}
This quantity is the size induced by the nonlinear parametrization 
. As shown in Appendix~\ref{app:compatibility-interpolation}, if
\[s\notin\left\{\frac12,\frac32,\frac52\right\},\]
then $\mathfrak X^\lambda_{s,T}$ agrees with the set
$\mathbb X^\lambda_{s,T}$ of $s$-compatible data introduced in
Section~\ref{sec2}. Moreover,
\[\mathfrak X^\lambda_{0,T}=\mathcal D^\lambda_{0,T},\qquad\text{and}\qquad\mathfrak X^\lambda_{3,T}=\mathbb X^\lambda_{3,T}.\]
At the three critical indices, the constrained interpolation definition of $\mathfrak X^\lambda_{s,T}$ is retained.\medskip

\noindent Our second main result can now be stated as follows.

\begin{theorem}\label{global_(0,3)}
Assume that the star graph is equilateral. Let $T>0$, $\lambda\geq1$, and $s\in[0,3]$. For every $d=(u^0,g_0,g,p)\in\mathfrak X^\lambda_{s,T}$, problem \eqref{kdV} admits a unique global solution $u\in\mathcal X_{s,T}$. If $\mathcal S_{s,T}$ denotes the corresponding data-to-solution map,
then the lifted map
\[
\widehat{\mathcal S}_{s,T}:
\widehat{\mathbb X}^\lambda_{s,T}
\longrightarrow
\mathcal X_{s,T},
\qquad
\widehat{\mathcal S}_{s,T}(U)
=
\mathcal S_{s,T}\bigl(\mathcal P(U)\bigr),
\]
is continuous. Moreover, there exists a positive nondecreasing function
\[
F_{T,s}:[0,\infty)\longrightarrow(0,\infty),
\]
depending only on the fixed parameters
$T,L,N,\alpha,\lambda$, and $s$, such that
\[
\|u\|_{\mathcal X_{s,T}}
\leq
F_{T,s}\bigl(\langle d\rangle_{s,\lambda}\bigr)
\langle d\rangle_{s,\lambda}.
\]
In particular, if $s\notin\left\{\frac12,\frac32,\frac52\right\}$, then
$\mathfrak X^\lambda_{s,T}=\mathbb X^\lambda_{s,T}$, and there exists
a positive nondecreasing function $\widetilde F_{T,s}$ such that
\[
\|u\|_{\mathcal X_{s,T}}
\leq
\widetilde F_{T,s}\bigl(\|d\|_{s,\lambda}\bigr)
\|d\|_{s,\lambda},
\]
where $\|d\|_{s,\lambda}$ is the ambient data norm defined above.
\end{theorem}

Together, Theorems~\ref{main} and~\ref{global_(0,3)} provide a
well-posedness framework for the KdV equation with boundary data on
finite equilateral star graphs. In this geometry, they extend the
single-edge KdV theory developed in \cite{BSZ 2003} to a coupled
system consisting of $N$ edges.

\vspace{0.2cm}

The paper is organized as follows. In Section~\ref{sec2}, we introduce
the notion of $s$-compatibility adapted to finite equilateral star
graphs. Section~\ref{sec3} develops the well-posedness theory for the
linear system associated with \eqref{kdV}. In Section~\ref{sec4}, we
establish the local well-posedness of \eqref{kdV} and prove
Theorem~\ref{main}. The global well-posedness result in the
$H^s$ framework, which constitutes the second main result of the
paper, is proved in Section~\ref{sec6}. Finally,
Section~\ref{sec7} presents some concluding remarks and open problems.
Appendix~\ref{app:compatibility-interpolation} contains the functional
interpolation framework used in the linear and nonlinear arguments. 

\section{\texorpdfstring{$s$-compatibility conditions}{s-compatibility conditions}\label{sec2}}
 When investigating the existence of solutions to \eqref{kdV}, particular attention must be paid to the functional spaces chosen for $u^0$, $p_j$, $g_0$, and $g_j$. These spaces may allow for the existence of traces for these functions, which must be consistent with the initial and boundary conditions of \eqref{kdV}, as we now make precise. 

Suppose that $u\in \prod_{j=1}^NC^\infty\left([0,T]\times [0,l_j]\right)$ is a solution of \eqref{kdV}. Define,
\begin{align*}
\phi_{0,j}(x)=u_j^0(x),\ \ j=1,...,N
\end{align*}
and observe that, according to the boundary conditions in \eqref{kdV}, we have
\begin{align*}
	\begin{cases}
		\phi_{0,j}(0)=\phi_{0,1}(0),\\
		\displaystyle\sum_{j=1}^N\phi_{0,j}''(0)=-\alpha\phi_{0,1}(0)-\frac{N}{3}\phi_{0,1}^2(0)+g_0(0),\\
		\phi_{0,j}(l_j)=p_j(0),\\
		\phi_{0,j}'(l_j)=g_j(0).
	\end{cases}
\end{align*}
Now, define
\begin{align*}
\phi_{1,j}(x)=-(\phi_{0,j}'+\phi_{0,j}'''+\phi_{0,j}\phi_{0,j}')(x),\quad j=1,...,N,
\end{align*}
and note that
\begin{align*}
	\begin{cases}
		\phi_{1,j}(0)=\phi_{1,1}(0),\\
		\displaystyle\sum_{j=1}^N\phi_{1,j}''(0)=-\alpha\phi_{1,1}(0)-\frac{N}{3}\left(\phi_{0,1}(0)\phi_{1,1}(0)+\phi_{1,1}(0)\phi_{0,1}(0)\right)+g_0^{(1)}(0),\\
		\phi_{1,j}(l_j)=p_j^{(1)}(0),\\
		\phi_{1,j}'(l_j)=g_j^{(1)}(0).
	\end{cases}
\end{align*}
Analogously, defining
\begin{align*}
\phi_{2,j}(x)=-(\phi_{1,j}'+\phi_{1,j}'''+\phi_{1,j}\phi_{0,j}'+\phi_{0,j}\phi_{1,j}')(x),\quad j=1,...,N
\end{align*}
we get
\begin{align*}
	\begin{cases}
		\phi_{2,j}(0)=\phi_{2,1}(0),\\
		\displaystyle\sum_{j=1}^N\phi_{2,j}''(0)
 =-\alpha\phi_{2,1}(0)
 -\frac{N}{3}\Bigl(
      \phi_{0,1}(0)\phi_{2,1}(0)
      +2\phi_{1,1}^2(0)
      +\phi_{2,1}(0)\phi_{0,1}(0)
    \Bigr)
 +g_0^{(2)}(0).\\
		\phi_{2,j}(l_j)=p_j^{(2)}(0),\\
		\phi_{2,j}'(l_j)=g_j^{(2)}(0).
	\end{cases}
\end{align*}
In general, we define for $j=1,...,N$,
\begin{align*}
\phi_{k,j}=	
\begin{cases}
u_j^0,& k=0,\\
-\left(\phi_{k-1,j}'''+\phi_{k-1,j}'+\dfrac{1}{2}\displaystyle\sum_{i=0}^{k-1}\frac{(k-1)!}{i!(k-1-i)!}\left(\phi_{i,j}\phi_{k-1-i,j}\right)'\right),& k=1,2,...
\end{cases}
\end{align*}
and, by induction, one can see that, for any $k \in \mathbb{N}\cup\{0\}$,
\begin{align}\label{eq 1.3}
	\begin{cases}
		\phi_{k,j}(0)=\phi_{k,1}(0),& j=1,...N,\\
		\displaystyle\sum_{j=1}^N\phi_{k,j}''(0)=-\alpha\phi_{k,1}(0)-\frac{N}{3}\sum_{l=0}^k\binom{k}{l}\phi_{l,1}(0)\phi_{k-l,1}(0)+g_0^{(k)}(0),&\\
		\phi_{k,j}(l_j)=p_j^{(k)}(0),& j=1,...N,\\
		\phi_{k,j}'(l_j)=g_j^{(k)}(0),& j=1,...N.
	\end{cases}
\end{align}
Here, we study the existence of solutions for \eqref{kdV} with data  
\begin{align*}
(u^0, g_0, g, p) \in \mathbb{H}^s(\mathcal{T}) 
\times H^{\frac{s-1}{3}}(0,T) 
\times \left[ H^{\frac{s}{3}}(0,T) \right]^{N} 
\times \left[ H^{\frac{s+1}{3}}(0,T) \right]^{N}.
\end{align*}
For $0\leq s\leq3$, the recursive terms are considered only for
$0\leq k\leq\lfloor s/3\rfloor$. Endpoint compatibility conditions
are imposed only when all the corresponding traces are well-defined.
For smooth data, the formal recursion and its trace identities
extend to all orders.

Motivated by the above, we must also require that these data satisfy the equalities in \eqref{eq 1.3} whenever the functions and their traces are well-defined. Traces are understood as continuous extensions of the corresponding endpoint evaluations from smooth functions to the relevant Sobolev spaces. This notion of traces follows from the Sobolev compact embedding: for \( a < b \), \( s \geq 0 \), \( k \in \mathbb{N} \cup \{0\} \), and \( \sigma \in (0,1) \) satisfying  
\[
k + \sigma < s - \frac{1}{2},
\]
we have
\begin{align*}
H^s(a,b) \hookrightarrow C^{k,\sigma}([a,b]).
\end{align*}
With this in hand, we determine which equations in \eqref{eq 1.3} should be considered in the compatibility analysis, depending on the value of \( s \).  

We begin by considering \( s \in [0,3] \).  In this case, \( s - 3k \leq 0 < \frac{1}{2} \) for every \( k \geq 1 \), so \( H^{s - 3k}(0, l_j) \) is not embedded in any space of continuous functions.   Consequently, there are no traces for \( \phi_{k,j} \) (and their derivatives) when \( k \geq 1 \).   Similarly, there are no traces for \( p_j^{(k)} \), \( g_0^{(k)} \), and \( g_j^{(k)} \) when \( k \geq 1 \).  It thus remains to study the traces of \( \phi_{0,j} \), \( \phi_{0,j}' \), \( \phi_{0,j}'' \), \( p_j \), \( g_0 \), and \( g_j \).

\subsection{\texorpdfstring{Traces for $\phi_{0,j},\ \phi_{0,j}'$ and $\phi_{0,j}'',\ j=1,...,N$, $s\in [0,3]$}{Traces of the initial data and their derivatives}}We now analyze the different cases according to the value of \( s \).  By the same reasoning used for \( k \geq 1 \), there are no traces of  \( \phi_{0,j} \), \( \phi_{0,j}' \), and \( \phi_{0,j}'' \) when \( 0 \leq s \leq \tfrac{1}{2} \). Hence, it is sufficient to consider the interval \( \tfrac{1}{2} < s \leq 3 \).
\begin{itemize}
	\item[i.] \textit{Case \( \tfrac{1}{2} < s \leq \tfrac{3}{2} \).}  
	In this range, \( 0 < s - \tfrac{1}{2} \), hence
	\begin{align*}
		H^s(0, l_j) \hookrightarrow C([0, l_j]).
	\end{align*}
Consequently, traces of \( \phi_{0,j} \) exist.  On the other hand, since \( s \leq \tfrac{3}{2} \) implies \( s - \tfrac{1}{2} \leq 1 \),  the embedding \( H^s(0, l_j) \hookrightarrow C^1([0, l_j]) \) does not hold.  Therefore, traces of \( \phi_{0,j}' \) and \( \phi_{0,j}'' \) are not defined.
	\item[ii.] \textit{Case \( \tfrac{3}{2} < s \leq \tfrac{5}{2} \).}  
	Here, \( 1 < s - \tfrac{1}{2} \), which yields
	\begin{align*}
		H^s(0, l_j) \hookrightarrow C^1([0, l_j]),
	\end{align*}
	and hence traces of \( \phi_{0,j} \) and \( \phi_{0,j}' \) exist.   However, since \( s \leq \tfrac{5}{2} \) implies \( s - \tfrac{1}{2} \leq 2 \), the embedding \( H^s(0, l_j) \hookrightarrow C^2([0, l_j]) \) is not valid, and thus traces of \( \phi_{0,j}'' \) do not exist.
	\item[iii.] \textit{Case \( \tfrac{5}{2} < s \leq 3 \).}  
	In this case, since \( \tfrac{5}{2} < s \) implies \( 2 < s - \tfrac{1}{2} \), we have
	\begin{align*}
		H^s(0, l_j) \hookrightarrow C^2([0, l_j]).
	\end{align*}
	Therefore, traces of \( \phi_{0,j} \), \( \phi_{0,j}' \), and \( \phi_{0,j}'' \) are well-defined.
\end{itemize}
\subsection{\texorpdfstring{Traces for $p_j,\ g_0$ and $g_j$, $j=1,...,N$, $s\in [0,3]$}{Traces of the boundary data}}For \( s \leq \tfrac{1}{2} \), the spaces \( H^{\frac{s-1}{3}}(0,T) \), \( H^{\frac{s}{3}}(0,T) \), and \( H^{\frac{s+1}{3}}(0,T) \) are not embedded into spaces of continuous functions. Hence, there are no traces for \( p_j \), \( g_0 \), or \( g_j \). For \( \tfrac{1}{2} < s \leq 3 \), we proceed by analyzing the following cases:
\begin{itemize}
	\item[i.] \textit{Case \( \tfrac{1}{2} < s \leq \tfrac{3}{2} \).}  
	Since \( \tfrac{1}{2} < s \) implies \( 0 < \tfrac{s+1}{3} - \tfrac{1}{2} \), and \( s \leq \tfrac{3}{2} \) gives \( \tfrac{s}{3} - \tfrac{1}{2} \leq 0 \), we obtain the embedding
	\begin{align*}
		H^{\frac{s+1}{3}}(0,T) \hookrightarrow C([0,T]),
	\end{align*}
while neither \( H^{\frac{s}{3}}(0,T) \hookrightarrow C([0,T]) \) nor \( H^{\frac{s-1}{3}}(0,T) \hookrightarrow C([0,T]) \) holds. Therefore, traces exist for \( p_j \), but not for \( g_0 \) or \( g_j \).
	\item[ii.] \textit{Case \( \tfrac{3}{2} < s \leq \tfrac{5}{2} \).}  
In this range, we have the embeddings
	\begin{align*}
		H^{\frac{s+1}{3}}(0,T) \hookrightarrow C([0,T]) 
		\quad \text{and} \quad 
		H^{\frac{s}{3}}(0,T) \hookrightarrow C([0,T]),
	\end{align*}
but not \( H^{\frac{s-1}{3}}(0,T) \hookrightarrow C([0,T]) \). Hence, traces exist for \( p_j \) and \( g_j \), but not for \( g_0 \).
	\item[iii.] \textit{Case \( \tfrac{5}{2} < s \leq 3 \).}  
In this case, all the embeddings hold:
	\begin{align*}
		H^{\frac{s-1}{3}}(0,T), \quad H^{\frac{s}{3}}(0,T), \quad H^{\frac{s+1}{3}}(0,T)
		\hookrightarrow C([0,T]).
	\end{align*}
Therefore, traces exist for \( g_0 \), \( g_j \), and \( p_j \).
\end{itemize}
In summary, we have the following relationships between values of $s\in [0,3]$ and the existence of traces:
\begin{center}
\begin{tabular}{|c|c|}
\hline
There exist traces of&When\\[0.3mm]
\hline
$\phi_{0,j},\ p_j$&$1/2<s\leq 3$\\[0.3mm]
\hline
$\phi_{0,j}',\ g_j$&$3/2<s\leq 3$\\[0.3mm]
\hline
$\phi_{0,j}'',\ g_0$&$5/2<s\leq 3$\\[0.3mm]
\hline
\end{tabular}
\end{center}
Based on this, we give the following definition.
\begin{definition}[$s$-compatibility for $0\leq s\leq 3$]
Given $T>0$ and $s \in [0,3]$ we say that
$$(u^0,g_0,g,p)\in \mathbb{H}^s(\mathcal{T})\times H^{\frac{s-1}{3}}(0,T) \times\left[H^\frac{s}{3}(0,T)\right]^{N}\times \left[H^\frac{s+1}{3}(0,T)\right]^N$$
is $s$-compatible when the following conditions are satisfied for $j=1,...,N$:
\begin{align*}
	\phi_{0,j}(0)=\phi_{0,1}(0)\text{ \ \ and \ \ }\phi_{0,j}(l_j)=p_j(0)&\text{ \ \ if \ \ }\frac{1}{2}<s\leq \frac{3}{2},\\
	\phi_{0,j}(0)=\phi_{0,1}(0),\ \ \phi_{0,j}(l_j)=p_j(0)\text{ \ \ and \ \ }\phi_{0,j}'(l_j)=g_j(0) \ &\text{ \ \ if \ \ }\frac{3}{2}<s\leq \frac{5}{2}
\end{align*}
or
\begin{align*}
	\hspace{4 cm}\left\{
	\begin{array}{ll}
		\phi_{0,j}(0)=\phi_{0,1}(0),\\
		\displaystyle\sum_{j=1}^N\phi_{0,j}''(0)=-\alpha\phi_{0,1}(0)-\frac{N}{3}\phi_{0,1}^2(0)+g_0(0),\\
		\phi_{0,j}(l_j)=p_j(0),\\
		\phi_{0,j}'(l_j)=g_j(0).
	\end{array}
	\right.
	&\text{ \ \ if \ \ }\frac{5}{2}<s\leq 3.
\end{align*}
No compatibility condition is imposed when $0\leq s\leq 1/2$.
\end{definition}
For $0\leq s\leq3$ and $\lambda\geq0$, we denote by
$\mathbb X^\lambda_{s,T}$ the set of data
$d=(u^0,g_0,g,p)\in\mathcal D^\lambda_{s,T}$ satisfying
the pointwise $s$-compatibility conditions in the preceding
definition. Thus,
\[
\mathbb X^\lambda_{s,T}
:=
\left\{
d\in\mathcal D^\lambda_{s,T}:
d \text{ satisfies the preceding }
s\text{-compatibility conditions}
\right\}.
\]
The parameter $\lambda$ specifies the additional temporal
regularity of the boundary data and does not alter these
compatibility conditions. In particular,
\[
\mathbb X^\lambda_{0,T}=\mathcal D^\lambda_{0,T}.
\]

\begin{remark}
At $s=1/2,3/2,5/2$, the preceding pointwise definition must be
distinguished from the admissible data sets defined by constrained
interpolation in Appendix~\ref{app:compatibility-interpolation}.
Moreover, for $s>5/2$, the quadratic junction condition makes the
set of compatible physical data nonlinear. The interpolation
argument is therefore carried out in augmented linear spaces. Figure~\ref{fig:compatibility-thresholds} summarizes the thresholds for the pointwise compatibility conditions.
\begin{figure}[htbp]
\centering
\begin{tikzpicture}[
    x=2.65cm,
    y=1cm,
    critical/.style={
        circle,
        draw=black,
        fill=white,
        line width=0.8pt,
        inner sep=0pt,
        minimum size=2.8mm
    },
    endpoint/.style={
        circle,
        draw=black,
        fill=black,
        inner sep=0pt,
        minimum size=2.4mm
    }
]

\draw[->,line width=0.8pt]
    (-0.05,0) -- (3.18,0) node[right] {$s$};

\draw[line width=1.5pt] (0,0)   -- (0.5,0);
\draw[line width=1.5pt] (0.5,0) -- (1.5,0);
\draw[line width=1.5pt] (1.5,0) -- (2.5,0);
\draw[line width=1.5pt] (2.5,0) -- (3,0);

\node[endpoint] at (0,0) {};
\node[endpoint] at (3,0) {};

\foreach \x in {0.5,1.5,2.5}
    \node[critical] at (\x,0) {};

\node[anchor=north] at (0,-0.13) {$0$};
\node[anchor=north] at (0.5,-0.13) {$\frac12$};
\node[anchor=north] at (1.5,-0.13) {$\frac32$};
\node[anchor=north] at (2.5,-0.13) {$\frac52$};
\node[anchor=north] at (3,-0.13) {$3$};

\node[
    align=center,
    font=\scriptsize,
    text width=2.3cm
] at (0.25,1.25)
{no compatibility conditions};
\draw[thin] (0.25,0.12) -- (0.25,0.82);

\node[
    align=center,
    font=\scriptsize,
    text width=3.3cm
] at (1,0.62)
{\,\,\,\, vertex continuity and\\Dirichlet compatibility};
\draw[thin] (1,0.12) -- (1,0.34);

\node[
    align=center,
    font=\scriptsize,
    text width=2.8cm
] at (2,1.25)
{also Neumann\\compatibility};
\draw[thin] (2,0.12) -- (2,0.82);

\node[
    align=center,
    font=\scriptsize,
    text width=3cm
] at (2.75,0.62)
{also the nonlinear\\junction compatibility};
\draw[thin] (2.75,0.12) -- (2.75,0.34);

\end{tikzpicture}
\caption{Thresholds for the pointwise compatibility conditions in the
range $s\in[0,3]$. At each critical index, the corresponding new
trace is at its limiting regularity.}
\label{fig:compatibility-thresholds}
\end{figure}
\end{remark}

For the linear problem, the compatibility conditions are obtained
by omitting the nonlinear terms from the preceding identities.

\section{The linear problem}\label{sec3}
We study the following linear nonhomogeneous initial-boundary
value problem:
\begin{align}\label{LkdV}
	\left\{
	\begin{array}{ll}
		\partial_t u_j+\partial_x u_j+\partial_x^3 u_j=f_j,&t \in (0,T),\ x \in (0,l_j),\ j=1,...,N,\\
		u_j(t,0)=u_1(t,0),&t\in (0,T),\ j=1,...,N,\\
		\displaystyle\sum_{j=1}^N\partial_x^2 u_j(t,0)=-\alpha u_1(t,0)+g_0(t),&t \in (0,T),\\				
		u_j(t,l_j)=p_j(t),\ \ \ \ \ \partial_x u_j(t,l_j)=g_j(t),&t\in (0,T),\ j=1,...,N,\\
		u_j(0,x)=u_j^0(x),&x \in (0,l_j),\ j=1,...,N,
	\end{array}
	\right.
\end{align}
with data
\[
(u^0,g_0,g,p)\in
\mathbb H^s(\mathcal T)
\times H^{(s-1)/3}(0,T)
\times[H^{s/3}(0,T)]^N
\times[H^{(s+1)/3}(0,T)]^N
\]
satisfying the linear $s$-compatibility conditions.
We first treat the endpoint regularities $s=0$ and $s=3$,
with forcing terms in $L^1(0,T;\mathbb L^2(\mathcal T))$ and  $W^{1,1}(0,T;\mathbb L^2(\mathcal T))
\cap L^2(0,T;\mathbb H^1(\mathcal T))$, respectively. The forcing spaces for intermediate regularities
are defined below by real interpolation, before Corollary~\ref{W.P._s in [0,3]}.

\subsection{The case \texorpdfstring{$s=0$}{s=0}}
We begin by considering the homogeneous boundary value problem
\begin{align}\label{LkdV-1}
	\left\{
	\begin{array}{ll}
		\partial_t u_j+\partial_x u_j+\partial_x^3 u_j=f_j,&t \in (0,T),\ x \in (0,l_j),\ j=1,...,N,\\
		u_j(t,0)=u_1(t,0),&t\in (0,T),\ j=1,...,N,\\
		\displaystyle\sum_{j=1}^N\partial_x^2 u_j(t,0)=-\alpha u_1(t,0),&t \in (0,T),\\				
		u_j(t,l_j)=\partial_x u_j(t,l_j)=0,&t\in (0,T),\ j=1,...,N,\\
		u_j(0,x)=u_j^0(x),&x \in (0,l_j),\ j=1,...,N.
	\end{array}
	\right.
\end{align}
Defining the operator $A:D(A)\subset \mathbb{L}^2(\mathcal{T})\rightarrow \mathbb{L}^2(\mathcal{T})$ by
\begin{align*}
D(A)&=\left\{u \in \mathbb{H}^3(\mathcal{T})\cap\mathbb{H}_e^2(\mathcal{T});\ \displaystyle\sum_{j=1}^N\partial_x^2 u_j(0)=-\alpha u_1(0) \right\}
\text{ \ and \ }
Au=-\partial_x u-\partial_x^3u.
\end{align*}
The adjoint operator $A^*:D(A^*)\subset \mathbb{L}^2(\mathcal{T})\rightarrow \mathbb{L}^2(\mathcal{T})$ is given by
\begin{align*}
	D(A^*) &=\left\{v \in \mathbb{H}^3(\mathcal{T})\cap\mathbb{H}_e^1(\mathcal{T}); \begin{array}{l}
		 \displaystyle\sum_{j=1}^N\partial_x^2v_j(0)
     =(\alpha-N)v_1(0),\\
		\partial_x v_j(0)=0,\ j=1,...,N
	\end{array}
	\right\}
	\text{ \ and \ }
	A^*v=\partial_x v+\partial_x^3v.
\end{align*}
The first result was shown in \cite{Ammari and Crepeau 2018} and is a direct consequence of the semigroup theory.  
\begin{proposition}\label{prop 2.1}
Let $T>0$ be given. For every $f\in C^1\left([0,T],\mathbb{L}^2(\mathcal{T})\right)$ and $u^0 \in D(A)$, the problem \eqref{LkdV-1} has a unique classical solution $u \in C^1\left([0,T],\mathbb{L}^2(\mathcal{T})\right)\cap C\left([0,T],D(A)\right)$.
\end{proposition}

Now, let us go back to the problem \eqref{LkdV}. To deal with classical solutions for the nonhomogeneous problem, we recall the space $\mathbb{H}^3_c(\mathcal{T})$, consisting of functions in $\mathbb{H}^3(\mathcal{T})$ that are continuous across the junction. At this stage it is convenient to introduce the space \[C_l^2([0,T])=\left\{\varphi \in C^2([0,T]);\ \varphi(0)=0\right\}.\]
\begin{proposition}\label{prop 2.2}
For every $T>0$, $u^0 \in D(A)$, $g_0,g_j,p_j \in C_l^2([0,T])$ and $f \in C^1\left([0,T];\mathbb{L}^2(\mathcal{T})\right)$ there exists a unique classical solution $u \in C^1\left([0,T],\mathbb{L}^2(\mathcal{T})\right)\cap C\left([0,T],\mathbb{H}^3_c(\mathcal{T})\right)$ to the problem \eqref{LkdV}. 
\end{proposition}
\begin{proof}
Suppose $u^0 \in D(A)$, $g_0,g_j,p_j \in \left\{\varphi \in C^2([0,T]);\ \varphi(0)=0\right\}$ and $f \in C^1\left([0,T];\mathbb{L}^2(\mathcal{T})\right)$. Consider $\phi_j, \psi_j, \theta_j\in C^\infty([0,l_j])$ satisfying
\begin{align*}
\begin{cases}
\phi_j(0)=\phi_1(0),&j=2,...,N,\\
\phi_j(l_j)=0,&j=1,...,N,\\
\phi_{j}^{\prime}(l_j)=0,&j=1,...,N,\\
\displaystyle\sum_{j=1}^N\phi_{j}^{\prime\prime}(0)=-\alpha\phi_1(0)+1,&
\end{cases}\\
\begin{cases}
	\psi_j(0)=0,&j=1,...,N,\\
	\psi_j(l_j)=1,&j=1,...,N,\\
	\psi_{j}^{\prime}(l_j)=0,&j=1,...,N,\\
	\psi_{j}^{\prime\prime}(0)=0,&j=1,...,N,
\end{cases}\qquad 
\begin{cases}
	\theta_j(0)=0,&j=1,...,N,\\
	\theta_j(l_j)=0,&j=1,...,N,\\
 	\theta_{j}^{\prime}(l_j)=1,&j=1,...,N,\\
	\theta_{j}^{\prime\prime}(0)=0,&j=1,...,N.
\end{cases}\\
\end{align*}
For $t \in [0,T]$, define $\tilde{f}(t):=\left(\tilde{f}_1(t,\cdot),...,\tilde{f}_N(t,\cdot)\right)$ where 
\begin{align*}
\tilde{f}_j:=f_j-\phi_jg_{0t}-(\phi_{j}^{\prime}+\phi_{j}^{\prime\prime\prime})g_0-\psi_jp_{jt}-(\psi_{j}^{\prime}+\psi_{j}^{\prime\prime\prime})p_j-\theta_jg_{jt}-(\theta_{j}^{\prime}+\theta_{j}^{\prime\prime\prime})g_j
\end{align*}
and note that $\tilde{f} \in C^1\left([0,T],\mathbb{L}^2(\mathcal{T})\right)$. Let $v\in  C^1\left([0,T],\mathbb{L}^2(\mathcal{T})\right)\cap C\left([0,T],D(A)\right)$ be the unique classical solution to the problem
\begin{align}\label{2.6}
	\begin{cases}
		\partial_t v_j+\partial_x v_j+\partial_x^3 v_j=\tilde{f}_j,&t \in (0,T),\ x \in (0,l_j),\ j=1,...,N,\\
		v_j(t,0)=v_1(t,0),&t\in (0,T),\ j=1,...,N,\\
		\displaystyle\sum_{j=1}^N\partial_x^2 v_j(t,0)=-\alpha v_1(t,0),&t \in (0,T),\\				
		v_j(t,l_j)=\partial_x v_j(t,l_j)=0,&t\in (0,T),\ j=1,...,N,\\
		v_j(0,x)=u_j^0(x),&x \in (0,l_j),\ j=1,...,N,
	\end{cases}
\end{align}
given by Proposition \ref{prop 2.1}. For each $t\in[0,T]$, define $u(t,\cdot)\in\mathbb L^2(\mathcal T)$ by setting
\begin{align}\label{2.7}
 u_j(t,x)=v_j(t,x)+\phi_j(x)g_0(t)
 +\psi_j(x)p_j(t)+\theta_j(x)g_j(t).
\end{align}
Since $\phi_j,\psi_j,\theta_j\in C^\infty\left([0,l_j]\right)$, $g_0,p_j,g_j\in C^2\left([0,T]\right)$ and $v\in  C^1\left([0,T],\mathbb{L}^2(\mathcal{T})\right)\cap C\left([0,T],D(A)\right)$, it follows that
\begin{align*}
u \in C^1\left([0,T],\mathbb{L}^2(\mathcal{T})\right)\cap C\left([0,T],\mathbb{H}_{c}^{3}(\mathcal{T})\right).
\end{align*}
Furthermore, we have
\begin{align*}
\partial_t u_j+\partial_x u_j+\partial_x^3 u_j=&\left(\partial_t v_j+\partial_x v_j+\partial_x^3 v_j\right)+\phi_jg_{0t}+\left(\phi_{j}^{\prime}+\phi_{j}^{\prime\prime\prime}\right)g_0\\
&+\psi_jp_{jt}+\left(\psi_{j}^{\prime}+\psi_{j}^{\prime\prime\prime}\right)p_j+\theta_jg_{jt}+\left(\theta_{j}^{\prime}+\theta_{j}^{\prime\prime\prime}\right)g_j
\end{align*}
so from \eqref{2.6},
\begin{align*}
	\partial_t u_j+\partial_x u_j+\partial_x^3 u_j=\tilde{f}_j&+\phi_jg_{0t}+\left(\phi_{j}^{\prime}+\phi_{j}^{\prime\prime\prime}\right)g_0+\psi_jp_{jt}+\left(\psi_{j}^{\prime}+\psi_{j}^{\prime\prime\prime}\right)p_j\\
	&+\theta_jg_{jt}+\left(\theta_{j}^{\prime}+\theta_{j}^{\prime\prime\prime}\right)g_j
\end{align*}
and, using the definition of $\tilde{f}_j$, it follows that
\begin{align}\label{2.8}
		\partial_t u_j+\partial_x u_j+\partial_x^3 u_j=f_j,\quad j=1,...,N.
\end{align}
By \eqref{2.8} and by the choice of the lifting functions $\phi, \psi, \theta$, after some straightforward computations, we conclude that $u$ is a classical solution to \eqref{LkdV}. For uniqueness, suppose that $\tilde{u}$ is a classical solution of \eqref{LkdV}. Then defining $\tilde{v}=(\tilde{v}_1,...,\tilde{v}_N)$, where
\begin{align}\label{2.14}
\tilde{v}_j(t,x)=\tilde{u}_j(t,x)-\phi_j(x)g_0(t)-\psi_j(x)p_j(t)-\theta_j(x)g_j(t),\ x \in (0,l_j),\ j=1,...,N,
\end{align}
one can see that $\tilde{v}$ solves \eqref{2.6} and, by uniqueness guaranteed in Proposition \ref{prop 2.1}, results $\tilde{v}=v$. But from \eqref{2.7} we have
\begin{align}\label{2.15}
	v_j(t,x)=u_j(t,x)-\phi_j(x)g_0(t)-\psi_j(x)p_j(t)-\theta_j(x)g_j(t),\ x \in (0,l_j),\ j=1,...,N.
\end{align}
Subtracting \eqref{2.15} from \eqref{2.14} we obtain
$\tilde{u}_j=u_j$, for all $j=1,\cdots,N$, finishing the proof. 
\end{proof}

Our next goal is to introduce a notion of mild solution for problem \eqref{LkdV} with less regular data, namely,
\[
u^0 \in \mathbb{L}^2(\mathcal{T}), \quad 
g_0 \in H^{-\frac{1}{3}}(0,T), \quad 
g_j \in L^2(0,T), \quad 
p_j \in H^{\frac{1}{3}}(0,T) \quad \text{for } j = 1, \ldots, N,
\]
and
\[
f \in L^1(0,T; \mathbb{L}^2(\mathcal{T})).
\]
To this end, it is convenient to recall some results concerning the well-posedness of the following single equations:
\begin{align}\label{2.21b}
	\left\{
	\begin{array}{ll}
		\partial_t\theta+\partial_x \theta+\partial_x^3\theta=f&t \in (0,T), \ x \in (0,L),\\
		\theta(t,0)=h_1(t),\ \ \ \ \theta(t,L)=h_2(t),\ \ \ \ \partial_x\theta(t,L)=h_3(t),&t\in(0,T)\\
		\theta(0,x)=\phi(x)&x \in (0,L)
	\end{array}
	\right.
\end{align}
and
\begin{align}\label{2.22b}
\left\{
\begin{array}{ll}
\partial_t \psi+\partial_x\psi+\partial_x^3\psi=f,&t \in (0,T), \ x \in (0,L),\\
\partial_x^2\psi(t,0)+\frac{\alpha}{N}\psi(t,0)=h_1(t),&t\in(0,T)\\
\psi(t,L)=h_2(t),\ \ \ \ \ \partial_x\psi(t,L)=h_3(t),&t \in (0,T),\\
\psi(0,x)=\phi(x),&x \in (0,L).
\end{array}
\right.
\end{align}
Here, it is necessary to introduce some definitions and notations. Given $T,L>0$ and $s\geq 0$ we define the space
\begin{align*}
	X_{s,T}:=H^s(0,L)\times H^{\frac{s+1}{3}}(0,T)\times  H^{\frac{s+1}{3}}(0,T)\times  H^{\frac{s}{3}}(0,T)
\end{align*}
with the norm
\begin{align*}
	\|(\phi,h_1,h_2,h_3)\|_{X_{s,T}}=\left(\|\phi\|_{H^s(0,L)}^2+\|h_1\|_{H^{\frac{s+1}{3}}(0,T)}^2+\|h_2\|_{H^{\frac{s+1}{3}}(0,T)}^2+\|h_3\|_{H^{\frac{s}{3}}(0,T)}^2\right)^\frac{1}{2},
\end{align*}
the space
\begin{align*}
	Y_{s,T}:=H^s(0,L)\times H^{\frac{s-1}{3}}(0,T)\times  H^{\frac{s+1}{3}}(0,T)\times  H^{\frac{s}{3}}(0,T)
\end{align*}
with the norm
\begin{align*}
	\|(\phi,h_1,h_2,h_3)\|_{Y_{s,T}}=\left(\|\phi\|_{H^s(0,L)}^2+\|h_1\|_{H^{\frac{s-1}{3}}(0,T)}^2+\|h_2\|_{H^{\frac{s+1}{3}}(0,T)}^2+\|h_3\|_{H^{\frac{s}{3}}(0,T)}^2\right)^\frac{1}{2}
\end{align*}
and the space 
\begin{align*}
	Z_{s,T}=C([0,T],H^s(0,L))\cap L^2(0,T,H^{s+1}(0,L))
\end{align*}
with the norm
\begin{align*}
	\|v\|_{Z_{s,T}}=\max_{t\in [0,T]}\|v(t,\cdot)\|_{H^s(0,L)}+\|v\|_{L^2(0,T,H^{s+1}(0,L))}.
\end{align*}
For simplicity, we write $\vec h=(h_1,h_2,h_3)$. The energy-level result for \eqref{2.21b} is taken from \cite[Lemma~3.3, pp.~1418--1419]{BSZ 2003}. We also record an additional $H^1$ estimate, justified below, which will be used in the Dirichlet lifting.  
Throughout the paper, $C>0$ denotes a generic constant that may change from line to line and is independent of the data. Unless
otherwise stated, $C$ may depend on the fixed parameters, including the time horizon $T$.

\begin{proposition}[Bona, Sun and Zhang, \cite{BSZ 2003}]\label{prop 2.3}
	Given $T>0$, for every $(\phi,\vec{h}) \in X_{0,T}$ and $f \in L^1(0,T,L^2(0,L))$ there exists a unique solution $\theta \in Z_{0,T}$ to the problem \eqref{2.21b} which satisfies
	\begin{align*}
	\partial_x^k\theta \in C_b\left([0,L]; H^\frac{1-k}{3}(0,T)\right),\ k=0,1,2.
	\end{align*}
	Moreover,
	\begin{align*}
		\|\theta\|_{Z_{0,T}}+\sum_{k=0}^2\|\partial_x^k\theta\|_{C_b([0,L]; H^\frac{1-k}{3}(0,T))}\leq C\left(\|(\phi,\vec{h})\|_{X_{0,T}}+\|f\|_{L^1(0,T;L^2(0,L))}\right)
	\end{align*}
	for some constant $C=C(T,L)>0$. In addition, suppose that $f=0$, $(\phi,\vec h)\in X_{1,T}$, and
\[
 \phi(0)=h_1(0),\qquad \phi(L)=h_2(0).
\]
Then the corresponding solution $\theta$ of \eqref{2.21b} belongs to $Z_{1,T}$ and satisfies
\begin{equation}\label{eq:scalar-H1-lifting-estimate}
 \|\theta\|_{Z_{1,T}}
 +\|\partial_x^2\theta\|_{C_b([0,L];L^2(0,T))}
 \leq C_{T,L}\|(\phi,\vec h)\|_{X_{1,T}}.
\end{equation}
\end{proposition}
\begin{proof} For $f=0$, the $Z_{1,T}$ estimate follows from
\cite[Lemma~3.3]{BSZ 2003}. 

\medskip
\noindent\textbf{Proof of the additional $H^1$ estimate.}
\medskip

We provide the additional hidden-trace estimate, which is needed below. Let $\mathcal S_0=X_{0,T}$ and define the closed linear subspace
\[\mathcal S_3:=\{(\phi,\vec h)\in X_{3,T}:\phi(0)=h_1(0),\quad\phi(L)=h_2(0),\quad\phi'(L)=h_3(0)\}.\]
At the lower endpoint, the first assertion of Proposition~\ref{prop 2.3} gives
\[\|\partial_x^2\theta\|_{C_b([0,L];H^{-1/3}(0,T))}\leq C_{T,L}\|(\phi,\vec h)\|_{\mathcal S_0}.\]
For data in $\mathcal S_3$, the time-differentiation argument in
\cite[Lemma~3.3]{BSZ 2003} identifies $v=\partial_t\theta$ as the
solution of the same scalar linear problem with initial datum
$-\phi'-\phi'''$ and boundary data $(h_1',h_2',h_3')$.
Applying the lower-endpoint estimate to $v$ gives
\[\|\partial_x^2v\|_{C_b([0,L];H^{-1/3}(0,T))}\leq C_{T,L}\|(\phi,\vec h)\|_{\mathcal S_3}.\]
Since $$\theta(t)=\phi+\int_0^t v(\tau)\,d\tau,$$ the bounded
primitive operator from $H^{-1/3}(0,T)$ to $H^{2/3}(0,T)$, normalized
to vanish at $t=0$, and the embedding $H^3(0,L)\hookrightarrow
C^2([0,L])$ imply
\[\|\partial_x^2\theta\|_{C_b([0,L];H^{2/3}(0,T))}\leq C_{T,L}\|(\phi,\vec h)\|_{\mathcal S_3}.\]
The finite-constraint interpolation argument used in the proof of
Lemma~\ref{lem:interpolation-constraints}, applied here to the three
linear functionals
\[\phi(0)-h_1(0),\qquad\phi(L)-h_2(0),\qquad\phi'(L)-h_3(0),\]
gives
\begin{equation}\label{eq:scalar-H1-compatible-interpolation}
 (\mathcal S_0,\mathcal S_3)_{1/3,2}=\{(\phi,\vec h)\in X_{1,T}:\phi(0)=h_1(0),\ \phi(L)=h_2(0)\}
\end{equation}
with equivalence of norms. Indeed, the two value constraints have
breaking index $1/6$, and the derivative constraint has breaking
index $1/2$. Their independent spatial liftings can be chosen as
cutoffs equal to one near $0$ and $L$, and a cutoff times $x-L$
near $L$, respectively. The corresponding $L^2$ and $H^3$ scaling
bounds are the order-zero and order-one bounds used in that lemma.
Thus $\vartheta=1/3$ is noncritical and imposes the two value constraints in \eqref{eq:scalar-H1-compatible-interpolation}. 
Interpolate the two bounds for the linear trace operator, using
\[(H^{-1/3}(0,T),H^{2/3}(0,T))_{1/3,2}=L^2(0,T).\]
The constants are uniform in $x\in[0,L]$. Density of $\mathcal S_3$
in the interpolation space then gives continuity in $x$ with values
in $L^2(0,T)$. This proves the trace part of
\eqref{eq:scalar-H1-lifting-estimate}, and hence the desired estimate.
\end{proof}

For the well-posedness of the problem \eqref{2.22b}, we start by considering the IBVP
\begin{align}\label{2.23b}
\begin{cases}
	\partial_t v+\partial_x^3v=f,&t \in (0,T), \ x \in (0,L),\\
	\partial_x^2v(t,0)=h_1(t),\ \ \ \ \ v(t,L)=h_2(t),\ \ \ \ \ \partial_xv(t,L)=h_3(t),&t \in (0,T),\\
	v(0,x)=\phi(x),&x \in (0,L).
\end{cases}
\end{align}
For this system, we use the following energy-level result from
\cite{Caicedo and Zhang 2017}.
\begin{proposition}[Caicedo and Zhang, \cite{Caicedo and Zhang 2017}]
\label{prop 2.4}
Let $T>0$. For every $f\in L^1(0,T;L^2(0,L))$ and
$(\phi,\vec h)\in Y_{0,T}$, the IBVP \eqref{2.23b} admits a
unique solution $v\in Z_{0,T}$ satisfying
\[
 \partial_x^k v\in C_b\left([0,L];H^{(1-k)/3}(0,T)\right),
 \qquad k=0,1,2.
\]
Moreover, there exists $C=C(T,L)>0$ such that
\begin{align*}
 \|v\|_{Z_{0,T}}
 +\sum_{k=0}^2
 \|\partial_x^k v\|_{C_b([0,L];H^{(1-k)/3}(0,T))}
 \leq C\left(
 \|(\phi,\vec h)\|_{Y_{0,T}}
 +\|f\|_{L^1(0,T;L^2(0,L))}\right).
\end{align*}
\end{proposition}

For estimates on variable time intervals, we equip the Sobolev
spaces with their restriction norms:
\[
\|z\|_{H^\sigma(0,\tau)}
:=
\inf\left\{
\|Z\|_{H^\sigma(\mathbb R)}:
Z|_{(0,\tau)}=z
\right\}.
\]
These norms are equivalent to the usual Sobolev norms on each fixed
interval. All uniformity statements below refer to this choice
of norms. We will use the following product estimates.

\begin{lemma}\label{lemma 2.1}
Let $T>0$ and $0\leq s\leq3$, and set
\[
r_s:=\frac{s+1}{3},
\qquad
q_s:=\frac{s-1}{3}.
\]
There exist constants $C=C(s,T)>0$ and $\mu_s>0$, independent
of $0<\tau\leq T$, such that the following assertions hold
for every $g,h\in H^{r_s}(0,\tau)$.

\begin{enumerate}
\item If $0\leq s<5/2$, then
\begin{equation}\label{eq:time-product-subcritical}
\|gh\|_{H^{q_s}(0,\tau)}
\leq
C\tau^{\mu_s}
\|g\|_{H^{r_s}(0,\tau)}
\|h\|_{H^{r_s}(0,\tau)}.
\end{equation}

\item If $5/2\leq s\leq3$, then
\begin{equation}\label{eq:time-product-renormalized}
\|gh-g(0)h(0)\|_{H^{q_s}(0,\tau)}
\leq
C\tau^{\mu_s}
\|g\|_{H^{r_s}(0,\tau)}
\|h\|_{H^{r_s}(0,\tau)}.
\end{equation}
In this case, one may take $\mu_s=(4-s)/3$.
In particular, \eqref{eq:time-product-subcritical} also holds
in this range whenever $g(0)h(0)=0$.
\end{enumerate}
\end{lemma}
\begin{proof}
Write $r=r_s$ and $q=q_s$. Choose extensions $G,H\in H^r(\mathbb R)$ whose norms are bounded by twice the corresponding restriction norms; the case of a zero function is immediate. If $q\leq0$, then $r\geq1/3$. The embedding $H^{1/3}(\mathbb R)\hookrightarrow L^6(\mathbb R)$ and H\"older's inequality gives
\[\|gh\|_{H^q(0,\tau)}\leq\|gh\|_{L^2(0,\tau)}\leq\tau^{1/6}\|G\|_{L^6(\mathbb R)}\|H\|_{L^6(\mathbb R)}\leq C\tau^{1/6}\|G\|_{H^r(\mathbb R)}\|H\|_{H^r(\mathbb R)}.\]
The first inequality follows by extending $gh$ by zero as an$L^2$ function. This proves the first assertion when $s\leq1$. Suppose next that $0<q<1/2$. Since $r=q+2/3>1/2$, $H^r(\mathbb R)$ is an algebra. Choose $q<\eta<r-1/2$ and put $F=GH$. The Sobolev embedding into bounded H\"older functions yields
\[\|F\|_{C_b^\eta(\mathbb R)}\leq C\|G\|_{H^r(\mathbb R)}\|H\|_{H^r(\mathbb R)}.\]
Let $\chi\in C_c^\infty(\mathbb R)$ be equal to one on $[0,1]$, and set $\chi_\tau(t)=\chi(t/\tau)$. Using the Gagliardo seminorm and $\eta>q$, we obtain
\[\|\chi_\tau F\|_{H^q(\mathbb R)}^2\leq C\bigl(\tau^{1-2q}+\tau\bigr)\|F\|_{C_b^\eta(\mathbb R)}^2.\]
Indeed, write the difference of the product as the sum of the difference of $\chi_\tau$ times $F$ and the difference of $F$ times $\chi_\tau$. The first term is estimated by scaling the $H^q$ seminorm of $\chi$, while the second is estimated using the H\"older bound for distances at most one and the supremum norm for larger distances. Since $\chi_\tau F$ extends $gh$, it follows that
\[\|gh\|_{H^q(0,\tau)}\leq C_T\tau^{1/2-q}\|G\|_{H^r(\mathbb R)}\|H\|_{H^r(\mathbb R)}.\]
This proves the first assertion for $1<s<5/2$. Finally, let $5/2\leq s\leq3$. Then $r>1$ and $1/2\leq q\leq2/3$. Set $z=gh-g(0)h(0)$. We have $z\in H^1(0,\tau)$, $z(0)=0$, and
\[\|z'\|_{L^2(0,\tau)}\leq C\|g\|_{H^r(0,\tau)}\|h\|_{H^r(0,\tau)},\]
with a constant independent of $\tau$, by the restriction norm
convention and the embedding $H^r(\mathbb R)\hookrightarrow
H^1(\mathbb R)\cap L^\infty(\mathbb R)$. Extend $z$ to $(\tau,2\tau)$ by reflection about $\tau$ and then
by zero outside $[0,2\tau]$, obtaining $Ez\in H^1(\mathbb R)$.
Since $z(0)=0$, Poincar\'e's inequality gives
\[\|Ez\|_{L^2(\mathbb R)}\leq C\tau\|z'\|_{L^2(0,\tau)},\qquad\|Ez\|_{H^1(\mathbb R)}\leq C_T\|z'\|_{L^2(0,\tau)}.\]
Interpolation between $L^2(\mathbb R)$ and $H^1(\mathbb R)$ now gives
\[\|z\|_{H^q(0,\tau)}\leq\|Ez\|_{H^q(\mathbb R)}\leq C_T\tau^{1-q}\|z'\|_{L^2(0,\tau)}.\]
Since $1-q=(4-s)/3$, this proves \eqref{eq:time-product-renormalized}.
\end{proof}

Thanks to Proposition~\ref{prop 2.4} and the previous lemma, we can establish the well-posedness of problem \eqref{2.22b} by a fixed-point argument.
\begin{proposition}\label{prop 2.5}
Let $T>0$ be given. For every $f\in L^1(0,T;L^2(0,L))$ and
$(\phi,\vec h)\in Y_{0,T}$, the IBVP \eqref{2.22b} admits a unique
solution $\psi\in Z_{0,T}$ satisfying
\[\partial_x^k\psi\in C_b\left([0,L];H^{\frac{1-k}{3}}(0,T)\right),\qquad k=0,1,2.\]
Moreover, there exists $C=C(T,L,\alpha)>0$ such that
\begin{align*}
 \|\psi\|_{Z_{0,T}}
 +\sum_{k=0}^2
 \|\partial_x^k\psi\|_{
 C_b([0,L];H^{\frac{1-k}{3}}(0,T))}
 \leq C\left(
 \|(\phi,\vec h)\|_{Y_{0,T}}
 +\|f\|_{L^1(0,T;L^2(0,L))}
 \right).
\end{align*}
\end{proposition}
\begin{proof}
For $\beta\in(0,T]$, consider the Banach space
\begin{align*}
 \mathcal Z_\beta
 =\biggl\{v\in Z_{0,\beta}:\
 \partial_x^k v\in
 C_b\left([0,L];H^{\frac{1-k}{3}}(0,\beta)\right),
 \ k=0,1,2\biggr\},
\end{align*}
endowed with the norm
\begin{align*}
 \|v\|_{\mathcal Z_\beta}
 =\|v\|_{Z_{0,\beta}}
 +\sum_{k=0}^2
 \|\partial_x^k v\|_{
 C_b([0,L];H^{\frac{1-k}{3}}(0,\beta))}.
\end{align*}
For $v\in\mathcal Z_\beta$, let $\Gamma v$ be the solution of
\begin{align*}
\begin{cases}
 \partial_t\psi+\partial_x^3\psi=f-\partial_xv,
     &t\in(0,\beta),\ x\in(0,L),\\
 \partial_x^2\psi(t,0)
     =h_1(t)-\dfrac{\alpha}{N}v(t,0),
     &t\in(0,\beta),\\
 \psi(t,L)=h_2(t),\quad \partial_x\psi(t,L)=h_3(t),
     &t\in(0,\beta),\\
 \psi(0,x)=\phi(x),&x\in(0,L).
\end{cases}
\end{align*}
Here and below, $f$ and the boundary data are restricted to
$(0,\beta)$. Proposition~\ref{prop 2.4}, with $s=0$, shows that
$\Gamma$ is well-defined and gives
\[
\|\Gamma v\|_{\mathcal Z_\beta} \leq C_T \biggl( \|(\phi,\vec h)\|_{Y_{0,\beta}} + \|f\|_{L^1(0,\beta;L^2(0,L))} + \frac{\alpha}{N} \|v(\cdot,0)\|_{H^{-1/3}(0,\beta)} + \|\partial_x v\|_{L^1(0,\beta;L^2(0,L))} \biggr),
\]
where $C_T$ may be chosen uniformly for $0<\beta\leq T$. By H\"older's inequality,
\begin{align*}
 \|\partial_xv\|_{L^1(0,\beta;L^2(0,L))}\leq \beta^{1/2}
 \|\partial_xv\|_{L^2(0,\beta;L^2(0,L))}\leq \beta^{1/2}\|v\|_{\mathcal Z_\beta}.
\end{align*}
Applying Lemma~\ref{lemma 2.1} with $s=0$, one factor equal to $1$,
and absorbing the norm of this fixed factor into $C_T$, we obtain
\begin{align*}
 \|v(\cdot,0)\|_{H^{-1/3}(0,\beta)}
 \leq C_T\beta^\mu
 \|v(\cdot,0)\|_{H^{1/3}(0,\beta)}
 \leq C_T\beta^\mu\|v\|_{\mathcal Z_\beta}.
\end{align*}
Consequently,
\begin{align}
 \|\Gamma v\|_{\mathcal Z_\beta}
 \leq C_TD_\beta
 +C_T\left(\frac{\alpha}{N}\beta^\mu+\beta^{1/2}\right)
 \|v\|_{\mathcal Z_\beta},
\label{2.25b}
\end{align}
where
\[D_\beta:=\|(\phi,\vec h)\|_{Y_{0,\beta}}+\|f\|_{L^1(0,\beta;L^2(0,L))}.\]
Choose $\beta\in(0,T]$ so that
\begin{align*}
q_\beta:=C_T\left(\frac{\alpha}{N}\beta^\mu+\beta^{1/2}\right)\leq\frac12.
\end{align*}
With $r=2C_TD_\beta$, estimate \eqref{2.25b} shows that $\Gamma(B_r)\subset B_r$, where $B_r=\{v\in\mathcal Z_\beta:\|v\|_{\mathcal Z_\beta}\leq r\}$. For $v,w\in B_r$, the difference $\Gamma v-\Gamma w$ solves the corresponding problem with zero prescribed data, forcing $-\partial_x(v-w)$, and left boundary datum $-(\alpha/N)(v-w)(\cdot,0)$. Hence the same estimates give
\begin{align*}
 \|\Gamma v-\Gamma w\|_{\mathcal Z_\beta}
 \leq q_\beta\|v-w\|_{\mathcal Z_\beta}
 \leq\frac12\|v-w\|_{\mathcal Z_\beta}.
\end{align*}
Thus $\Gamma$ has a unique fixed point $\psi\in B_r$. This fixed point
solves \eqref{2.22b} on $[0,\beta]$, and
\begin{align}
 \|\psi\|_{\mathcal Z_\beta}\leq2C_TD_\beta.
\label{eq:psi-25-local}
\end{align}
The choice of $\beta$ depends only on $T$, $L$, and the fixed parameter $\alpha$, and not on the prescribed data. We may therefore repeat the same argument, after a time translation, on a finite partition of $[0,T]$ into subintervals of length at most $\beta$, taking the terminal
value on each subinterval as the initial value on the next one. There is no corner compatibility condition at the level $s=0$. Iterating \eqref{eq:psi-25-local} over this finite partition yields
\begin{align*}
 \|\psi\|_{\mathcal Z_T}
 \leq C(T,L,\alpha)\left(
 \|(\phi,\vec h)\|_{Y_{0,T}}
 +\|f\|_{L^1(0,T;L^2(0,L))}
 \right).
\end{align*}
Uniqueness follows by applying the contraction estimate successively on the same subintervals.
\end{proof}

We will give one final technical result, which will be useful in proving the main theorem of this subsection.
\begin{lemma}\label{lemma 2.2}
Let $T>0$ be given. For every $u^0\in D(A)$,
$g_0,g_j,p_j\in C_l^2([0,T])$, $j=1,\ldots,N$, and
$f\in C^1([0,T];\mathbb L^2(\mathcal T))$, the corresponding solution
$u$ of problem \eqref{LkdV} satisfies
\begin{align*}
 \partial_x^k u_j\in C([0,T]\times[0,l_j]),
 \qquad k=0,1,2,\quad j=1,\ldots,N.
\end{align*}
\end{lemma}
\begin{proof}
By Proposition~\ref{prop 2.2}, $u\in C\left([0,T];\mathbb H_c^3(\mathcal T)\right)$. For each $j=1,\ldots,N$ and $k=0,1,2$, the Sobolev embedding $H^{3-k}(0,l_j)\hookrightarrow C([0,l_j])$ yields
\begin{align*}
\|\partial_x^ku_j(t,\cdot)-\partial_x^ku_j(t_0,\cdot)\|_{C([0,l_j])}\leq C\|u_j(t,\cdot)-u_j(t_0,\cdot)\|_{H^3(0,l_j)}.
\end{align*}
Since $t\mapsto u_j(t,\cdot)$ is continuous from $[0,T]$ into $H^3(0,l_j)$, it follows that
\[t\longmapsto\partial_x^ku_j(t,\cdot)\quad\text{belongs to}\quad C\left([0,T];C([0,l_j])\right).\]
Consequently, $(t,x)\mapsto\partial_x^ku_j(t,x)$ is continuous on $[0,T]\times[0,l_j]$, which proves the result.
\end{proof}

For later use, we also record the following graph-norm estimate. For
$j=1,\ldots,N$, let
\[P_j:H^3(0,l_j)\longrightarrow L^2(0,l_j),\qquad P_jw=w_x+w_{xxx}.\]
By \cite[Lemma A.2]{KdV_flatness}, there exists $C>0$ such that
\begin{align}\label{H^3<P_j}
\|w\|_{H^3(0,l_j)} \leq C\left(\|w\|_{L^2(0,l_j)}+\|P_jw\|_{L^2(0,l_j)}\right), \qquad w\in H^3(0,l_j).
\end{align}
The standing equilateral assumption is essential in the argument below, since it allows sums and differences of the edge components to be treated as solutions of scalar initial-boundary value problems on the same interval $(0,L)$. We are now in a position to state the main result of this subsection, whose proof follows from Propositions~\ref{prop 2.3} and \ref{prop 2.5}, together with Lemma~\ref{lemma 2.2}.
\begin{proposition}\label{W.P._s=0}
Given $T>0$, there exist unique bounded linear maps
\begin{align*}
\begin{array}{rcl}
\Psi:\mathbb{L}^2(\mathcal{T})\times H^{-\frac{1}{3}}(0,T)\times \left[L^2(0,T)\right]^{N}\times \left[H^\frac{1}{3}(0,T)\right]^N\times L^1(0,T,\mathbb{L}^2(\mathcal{T}))&\rightarrow&\mathbb{B}_T\\
(u^0,g_0,g,p,f)&\mapsto&\Psi(u^0,g_0,g,p,f),
\end{array}
\end{align*}
and, for $k=0,1,2$,
\begin{align*}
\begin{array}{rcl}
\Psi_j^k:\mathbb{L}^2(\mathcal{T})\times H^{-\frac{1}{3}}(0,T)\times \left[L^2(0,T)\right]^{N}\times \left[H^\frac{1}{3}(0,T)\right]^N\times L^1(0,T,\mathbb{L}^2(\mathcal{T}))&\rightarrow&C_b\left([0,l_j]; H^\frac{1-k}{3}(0,T)\right)\\
(u^0,g_0,g,p,f)&\mapsto&\Psi_j^k(u^0,g_0,g,p,f)
\end{array}
\end{align*}
such that, when $u^0 \in D(A)$, $g_0,g_j,p_j\in C_l^2([0,T])$ for $j=1,...,N$ and $f \in C^1\left([0,T];\mathbb{L}^2(\mathcal{T})\right)$ we have the following items:
\begin{enumerate}
\item[i)]$u:=\Psi(u^0,g_0,g,p,f)$ is the unique classical solution of \eqref{LkdV};
\item[ii)] $\Psi_j^k(u^0,g_0,g,p,f)=\partial_x^k u_j$.
\end{enumerate}
\end{proposition}
\begin{proof} 
	First we define	
	\begin{align*}
		\begin{array}{rcl}
			\Psi:D(A)\times C_l^2([0,T])\times\left[C_l^2([0,T])\right]^{N}\times \left[C_l^2([0,T])\right]^N\times C^1([0,T],\mathbb{L}^2(\mathcal{T}))&\rightarrow&\mathbb{B}_T\\
(u^0,g_0,g,p,f)&\mapsto&\Psi(u^0,g_0,g,p,f),
		\end{array}
	\end{align*}
and
\begin{align*}
	\begin{array}{rcl}
		\Psi_j^k:D(A)\times C_l^2([0,T])\times\left[C_l^2([0,T])\right]^{N}\times \left[C_l^2([0,T])\right]^N\times C^1([0,T],\mathbb{L}^2(\mathcal{T}))&\rightarrow&C_b\left([0,l_j]; H^\frac{1-k}{3}(0,T)\right)\\
(u^0,g_0,g,p,f)&\mapsto&\Psi_j^k(u^0,g_0,g,p,f)
	\end{array}
\end{align*}
for $k=0,1,2$. By proving that $\Psi$ and $\Psi_j^k$ are continuous with respect to the norm of
$$\mathbb{L}^2(\mathcal{T})\times H^{-\frac{1}{3}}(0,T)\times \left[L^2(0,T)\right]^{N}\times \left[H^\frac{1}{3}(0,T)\right]^N\times L^1(0,T,\mathbb{L}^2(\mathcal{T})),$$
we can extend them to this space by density. Suppose $u^0 \in D(A)$, $g_0,g_j,p_j \in C_l^2([0,T])$ and $f \in C^1\left([0,T];\mathbb{L}^2(\mathcal{T})\right)$. From Proposition \ref{prop 2.2}, there exists a unique classical solution 
\begin{align*}
u\in  C^1\left([0,T],\mathbb{L}^2(\mathcal{T})\right)\cap C\left([0,T],\mathbb{H}^{3}_{c}(\mathcal{T})\right)
\end{align*}
for \eqref{LkdV}. This gives in particular $u_j \in C([0,T],L^2(0,l_j)),\ j=1,...,N$. Furthermore, due to Lemma \ref{lemma 2.2}, we have that
\begin{align*}
\int_0^T\|u_j(t,\cdot)\|_{H^1(0,l_j)}^2&=\int_0^T\int_0^{l_j}|u_j(t,x)|^2+\int_0^T\int_0^{l_j}|\partial_xu_j(t,x)|^2<\infty
\end{align*}
hence $u_j \in L^2(0,T,H^1(0,l_j)),\ j=1,...,N$. Consequently,
\begin{align}\label{2.29}
u_j \in C([0,T],L^2(0,l_j))\cap L^2(0,T,H^1(0,l_j)),\ \ j=1,...,N
\end{align}
as well as $u \in \mathbb{B}=C([0,T],\mathbb{L}^2(\mathcal{T}))\cap L^2(0,T;\mathbb{H}^1(\mathcal{T}))$. So we are motivated to define $\Psi(u^0,g_0,g,p,f)=u.$ Consider $\psi$, $\theta$ defined by
\begin{align*}
\psi:=\displaystyle\sum_{j=1}^Nu_j \quad \text{ \ and \ }\quad \theta:=Nu_1-\psi.
\end{align*}
Note that, by the regularity \eqref{2.29}, we have $\theta \in  Z_{0,T}=C([0,T],L^2(0,L))\cap L^2(0,T,H^1(0,L))$. Moreover, $\theta$ is a classical solution to the IBVP
\begin{align*}
	\begin{cases}
		\partial_t \theta+\partial_x \theta+\partial_x^3 \theta=Nf_1-\displaystyle\sum_{j=1}^Nf_j,&t \in (0,T),\ x \in (0,L),\ j=1,...,N,\\
		\theta(t,0)=0,&t\in (0,T),\ j=1,...,N,\\
		\theta(t,L)=Np_1(t)-\displaystyle\sum_{j=1}^Np_j(t),&t\in (0,T),\ j=1,...,N,\\
		\partial_x\theta(t,L)=Ng_1(t)-\displaystyle\sum_{j=1}^Ng_j(t),&t\in (0,T),\ j=1,...,N,\\	
		\theta(0,x)=Nu_1^0(x)-\displaystyle\sum_{j=1}^Nu_j^0(x),&x \in (0,L),\ j=1,...,N.
\end{cases}
\end{align*}
Then Proposition~\ref{prop 2.3} ensures that
$\partial_x^k\theta\in
C_b([0,L];H^{(1-k)/3}(0,T))$ for $k=0,1,2$, and
\begin{align*}
\|\theta\|_{Z_{0,T}}+\sum_{k=0}^2\|\partial_x^k\theta\|_{C_b([0,L]; H^\frac{1-k}{3}(0,T))}\leq& C\left( \left\|Nu_1^0-\sum_{j=1}^Nu_j^0\right\|_{L^2(0,L)}+\left\|Np_1-\sum_{j=1}^Np_j\right\|_{H^\frac{1}{3}(0,T)}\right.\\
&\left.+\left\|Ng_1-\sum_{j=1}^Ng_j\right\|_{L^2(0,T)}+\left\|Nf_1-\sum_{j=1}^Nf_j\right\|_{L^1(0,T;L^2(0,L))}\right)
\end{align*}
From this inequality, we obtain
\begin{align*}
\|\theta\|_{Z_{0,T}}+\sum_{k=0}^2\|\partial_x^k\theta\|_{C_b([0,L]; H^\frac{1-k}{3}(0,T))}\leq&	C\left(\sum_{j=1}^N\|u_j^0\|_{L^2(0,L)}+\sum_{j=1}^N\|p_j\|_{H^\frac{1}{3}(0,T)}\right.\\
&\left.  +\sum_{j=1}^N\|g_j\|_{L^2(0,T)}+\sum_{j=1}^N\|f_j\|_{L^1(0,T;L^2(0,L))}\right).
\end{align*}
Analogously, \eqref{2.29} ensures that $\psi \in Z_{0,T}$. Observe that $\partial_x^2\psi(t,0)+\frac{\alpha}{N}\psi(t,0)=g_0(t)$.
Consequently, $\psi$ is a classical solution of
\begin{align*}
	\begin{cases}
		\partial_t \psi+\partial_x \psi+\partial_x^3 \psi=\displaystyle\sum_{j=1}^Nf_j,&t \in (0,T),\ x \in (0,L),\\
		\partial_x^2\psi(t,0)+(\alpha/N)\psi(t,0)=g_0(t),&t \in (0,T),\\ \psi(t,L)=\displaystyle\sum_{j=1}^Np_j(t),\ \ \partial_x \psi(t,L)=\sum_{j=1}^Ng_j(t)&t \in (0,T),\\	
		\psi(0,x)=\displaystyle\sum_{j=1}^Nu_j^0(x),&x \in (0,L).
	\end{cases}
\end{align*}
Thus Proposition~\ref{prop 2.5} gives
$\partial_x^k\psi\in
C_b([0,L];H^{(1-k)/3}(0,T))$ for $k=0,1,2$, and
\begin{align*}
	\|\psi\|_{Z_{0,T}}+\sum_{k=0}^2\|\partial_x^k\psi\|_{C_b([0,L]; H^\frac{1-k}{3}(0,T))}\leq& C\left(\sum_{j=1}^N\|u_j^0\|_{L^2(0,L)}+\sum_{j=1}^N\|p_j\|_{H^\frac{1}{3}(0,T)}+\|g_0\|_{H^{-\frac{1}{3}}(0,T)}\right.\\
	&\left.+\sum_{j=1}^N\|g_j\|_{L^2(0,T)}+\sum_{j=1}^N\|f_j\|_{L^1(0,T;L^2(0,L))}\right).
\end{align*}
Since $u_1=\frac{1}{N}(\theta+\psi)$ it follows that $u_1 \in Z_{0,T}$, $\partial_x^ku_1 \in C_b\left([0,L]; H^\frac{1-k}{3}(0,T)\right)$,  for $k=0,1,2,$ and
\begin{align}\label{2.30}
\begin{aligned}
	\|u_1\|_{Z_{0,T}}+&\sum_{k=0}^2\|\partial_x^ku_1\|_{C_b([0,L]; H^\frac{1-k}{3}(0,T))}\leq C\left(\sum_{j=1}^N\|u_j^0\|_{L^2(0,L)}+\|g_0\|_{H^{-\frac{1}{3}}(0,T)}\right.\\	
	&\left.+\sum_{j=1}^N\|g_j\|_{L^2(0,T)}+\sum_{j=1}^N\|p_j\|_{H^\frac{1}{3}(0,T)}+\sum_{j=1}^N\|f_j\|_{L^1(0,T;L^2(0,L))}\right).
\end{aligned}
\end{align}
Hence, $u_1(\cdot,0)\in H^\frac{1}{3}(0,T)$ and, for each $j=2,....,N$, $u_j\in Z_{0,T}$ is the solution of the following problem
\begin{align*}
\begin{cases}
\partial_t u_j+\partial_x u_j+\partial_x^3 u_j=f_j,&t \in (0,T),\ x \in (0,L),\\
u_j(t,0)=u_1(t,0),\ \ \ u_j(t,L)=p_j(t),\ \ \ \partial_x u_j(t,L)=g_j(t),&t \in (0,T),\\
u_j(0,x)=u_j^0(x),&x \in (0,L).
\end{cases}
\end{align*}
By Proposition \ref{prop 2.3} it follows that $\partial_x^k u_j \in C_b\left([0,L]; H^\frac{1-k}{3}(0,T)\right),\ k=0,1,2,$ and
\begin{align*}
	\|u_j\|_{Z_{0,T}}+\sum_{k=0}^2\|\partial_x^ku_j\|_{C_b([0,L]; H^\frac{1-k}{3}(0,T))}\leq& C\Big(\|u_j^0\|_{L^2(0,L)}+\|u_1(\cdot,0)\|_{H^\frac{1}{3}(0,T)}+\|p_j\|_{H^\frac{1}{3}(0,T)}\\
	&+\|g_j\|_{L^2(0,T)}+\|f_j\|_{L^1(0,T;L^2(0,L))}\Big).
\end{align*}
Using \eqref{2.30} we get
\begin{align}\label{2.31}
\begin{aligned}
	\|u_j\|_{Z_{0,T}}+&\sum_{k=0}^2\|\partial_x^ku_j\|_{C_b([0,L]; H^\frac{1-k}{3}(0,T))}\leq C\left(\sum_{j=1}^N\|u_j^0\|_{L^2(0,L)}+\|g_0\|_{H^{-\frac{1}{3}}(0,T)}\right.\\
	&\left.+\sum_{j=1}^N\|g_j\|_{L^2(0,T)}+\sum_{j=1}^N\|p_j\|_{H^\frac{1}{3}(0,T)}+\sum_{j=1}^N\|f_j\|_{L^1(0,T;L^2(0,L))}\right),
\end{aligned}
\end{align}
for $j=1,...,N$. Adding for $j=1,...,N$ this implies, in particular
\begin{align}\label{2.34}
	\begin{aligned}
		\max_{t\in [0,T]}\|u(t,\cdot)\|_{\mathbb{L}^2(\mathcal{T})}+&\left(\int_0^T\|u(t,\cdot)\|_{\mathbb{H}^1}^2dt\right)^\frac{1}{2}\leq C\left(\sum_{j=1}^N\|u_j^0\|_{L^2(0,L)}+\|g_0\|_{H^{-\frac{1}{3}}(0,T)}\right.\\
		&\left.+\sum_{j=1}^N\|g_j\|_{L^2(0,T)}+\sum_{j=1}^N\|p_j\|_{H^\frac{1}{3}(0,T)}+\sum_{j=1}^N\|f_j\|_{L^1(0,T;L^2(0,L))}\right).
	\end{aligned}
\end{align}
Observe that
\begin{align}\label{2.35}
\sum_{j=1}^N\|f_j\|_{L^1(0,T;L^2(0,L))}\leq N^\frac{1}{2}\|f\|_{L^1(0,T;\mathbb{L}^2(\mathcal{T}))}
\end{align}
as well as
\begin{align}\label{2.36}
\begin{aligned}
&\sum_{j=1}^N\|u_j^0\|_{L^2(0,L)}+\|g_0\|_{H^{-\frac{1}{3}}(0,T)}+\sum_{j=1}^N\|g_j(t)\|_{L^2(0,T)}+\sum_{j=1}^N\|p_j(t)\|_{H^\frac{1}{3}(0,T)}\\
&\leq N^\frac{1}{2}\left(\|u^0\|_{\mathbb{L}^2(\mathcal{T})}+\|g_0\|_{H^{-\frac{1}{3}}(0,T)}+\|g\|_{\left[L^2(0,T)\right]^{N}}+\|p\|_{\left[H^\frac{1}{3}(0,T)\right]^N}\right).
\end{aligned}
\end{align}
From \eqref{2.34}, \eqref{2.35} and \eqref{2.36} it follows that
\begin{align*}
\|u\|_{\mathbb{B}}&\leq C\left(\|u^0\|_{\mathbb{L}^2(\mathcal{T})}+\|g_0\|_{H^{-\frac{1}{3}}(0,T)}+\|g\|_{\left[L^2(0,T)\right]^{N}}+\|p\|_{\left[H^\frac{1}{3}(0,T)\right]^N}+\|f\|_{L^1(0,T;\mathbb{L}^2(\mathcal{T}))}\right).
\end{align*}
Therefore, due to this previous inequality, $\Psi$ is continuous. Now, combining \eqref{2.31}, \eqref{2.35} and \eqref{2.36} we have
\begin{align*}
\sum_{k=0}^2\|\partial_x^ku_j\|_{C_b([0,L]; H^\frac{1-k}{3}(0,T))}\leq &C\left(\|u^0\|_{\mathbb{L}^2(\mathcal{T})}+\|g_0\|_{H^{-\frac{1}{3}}(0,T)}+\|g\|_{\left[L^2(0,T)\right]^{N}}\right.\\
&\left.+\|p\|_{\left[H^\frac{1}{3}(0,T)\right]^N}+\|f\|_{L^1(0,T;\mathbb{L}^2(\mathcal{T}))}\right).
\end{align*}
Hence, defining $\Psi_j^k(u^0,g_0,g,p,f)=\partial_x^ku_j,\ k=0,1,2,$ we have that $\Psi_j^k$ is continuous. As mentioned before, by density, each map $\Psi$ and $\Psi_j^k$, for $k=0,1,2$, has a unique continuous extension to the space $\mathbb{L}^2(\mathcal{T})\times H^{-\frac{1}{3}}(0,T)\times \left[L^2(0,T)\right]^{N}\times \left[H^\frac{1}{3}(0,T)\right]^N\times L^1(0,T,\mathbb{L}^2(\mathcal{T}))$ which we will also denote by $\Psi$ and $\Psi_j^k$, respectively. Hence, the proposition is proved.
\end{proof}

Proposition \ref{W.P._s=0} motivates the introduction of the notion of \textit{mild solution} for system \eqref{LkdV}, guaranteeing its well-posedness in this sense, including existence, uniqueness, and continuous dependence on the data.
\begin{definition}
Given $T>0$, $u^0\in \mathbb{L}^2(\mathcal{T}),\ g_0 \in H^{-\frac{1}{3}}(0,T),\ g \in \left[L^2(0,T)\right]^{N},\ p\in\left[H^\frac{1}{3}(0,T)\right]^N$ and $f \in L^1(0,T;\mathbb{L}^2(\mathcal{T}))$, we define the mild solution $u\in \mathbb{B}$ of \eqref{LkdV} as a function $u=\Psi(u^0,g_0,g,p,f)$, where $\Psi$ is the map given in Proposition \ref{W.P._s=0}.
\end{definition}

\begin{remark}
Note that, according to Proposition \ref{W.P._s=0}, if $u^0 \in D(A)$, $g_0,g_j,p_j\in C_l^2([0,T])$, for $j=1,...,N$, and $f \in C^1\left([0,T];\mathbb{L}^2(\mathcal{T})\right)$, then $u$ is actually a classical solution to the system \eqref{LkdV}.
\end{remark}
\begin{remark}\label{rem:linear-restart}
For $0\leq t_0<t_1\leq T$, the restriction of a mild solution to $[t_0,t_1]$ is the mild solution with initial value $u(t_0)$ and
the restricted data. Conversely, finitely many mild solutions with matching $\mathbb L^2$ values at the joining times form a global mild solution whenever their data are restrictions
of the same global data.
\end{remark}
\subsection{The case \texorpdfstring{$s=3$}{s=3}}
In this subsection, we consider data
\[(u^0,g_0,g,p)\in \mathbb H^3(\mathcal T)\times H^{2/3}(0,T)\times[H^1(0,T)]^N\times[H^{4/3}(0,T)]^N\]
satisfying the linear $3$-compatibility conditions, and a forcing
term
\[f\in W^{1,1}(0,T;\mathbb L^2(\mathcal T))\cap L^2(0,T;\mathbb H^1(\mathcal T)).\]
We denote the set of these compatible initial and boundary data
by $\mathbb X_{3,T}$. This notation refers to the linear problem
and is distinct from $\mathbb X^0_{3,T}$, which denotes the
nonlinear compatibility set. Suppose that $u$ is the solution of \eqref{LkdV}. Formally
differentiating the first equation of \eqref{LkdV} with respect to
$t$, we obtain
\begin{align*}
\begin{cases}
 \partial_t(\partial_tu_j)+\partial_x(\partial_tu_j)
 +\partial_x^3(\partial_tu_j)=\partial_tf_j,
 &t\in(0,T),\ x\in(0,l_j),\ j=1,\ldots,N,\\
 \partial_tu_j(t,0)=\partial_tu_1(t,0),
 &t\in(0,T),\ j=1,\ldots,N,\\
 \displaystyle\sum_{j=1}^N\partial_x^2(\partial_tu_j)(t,0)
 =-\alpha\partial_tu_1(t,0)+g_0'(t),
 &t\in(0,T),\\
 \partial_tu_j(t,l_j)=p_j'(t),\quad
 \partial_x(\partial_tu_j)(t,l_j)=g_j'(t),
 &t\in(0,T),\ j=1,\ldots,N.
\end{cases}
\end{align*}
Moreover,
\[\partial_tu_j(0,x)=-\partial_xu_j^0(x)-\partial_x^3u_j^0(x)+f_j(0,x),\qquad x\in(0,l_j).\]
Thus, setting $v_j=\partial_tu_j$, we are led to
\begin{align}\label{v_j}
\begin{cases}
 \partial_tv_j+\partial_xv_j+\partial_x^3v_j=\partial_tf_j,
 &t\in(0,T),\ x\in(0,l_j),\ j=1,\ldots,N,\\
 v_j(t,0)=v_1(t,0),
 &t\in(0,T),\ j=1,\ldots,N,\\
 \displaystyle\sum_{j=1}^N\partial_x^2v_j(t,0)
 =-\alpha v_1(t,0)+g_0'(t),
 &t\in(0,T),\\
 v_j(t,l_j)=p_j'(t),\quad
 \partial_xv_j(t,l_j)=g_j'(t),
 &t\in(0,T),\ j=1,\ldots,N,\\
 v_j(0,x)=-\partial_xu_j^0(x)-\partial_x^3u_j^0(x)+f_j(0,x),
 &x\in(0,l_j),\ j=1,\ldots,N.
\end{cases}
\end{align}
We then set
\begin{align}\label{u_j}
 u_j(t,\cdot)=u_j^0+\int_0^t v_j(s,\cdot)\,ds,
 \qquad j=1,\ldots,N.
\end{align}
We will show that the mild solution of \eqref{LkdV} is given by
\eqref{u_j} and has the desired regularity. Under the assumptions
above, continuity of differentiation from $H^\sigma(0,T)$ into
$H^{\sigma-1}(0,T)$ gives
\begin{align*}
 p_j'\in H^{1/3}(0,T),\qquad
 g_j'\in L^2(0,T),\qquad
 g_0'\in H^{-1/3}(0,T).
\end{align*}
Also,
\begin{align*}
 -\partial_xu_j^0-\partial_x^3u_j^0+f_j(0,\cdot)
 &\in L^2(0,l_j),\qquad\text{and}\qquad
 \partial_tf\in L^1(0,T;\mathbb L^2(\mathcal T)).
\end{align*}
Therefore, Proposition~\ref{W.P._s=0} gives a unique mild solution
$v\in\mathbb B_T$ of \eqref{v_j}, satisfying
\begin{align*}
 \|v\|_{\mathbb B_T}\leq C\biggl(&
 \|-\partial_xu^0-\partial_x^3u^0+f(0,\cdot)\|_{\mathbb L^2(\mathcal T)}
 +\|g_0'\|_{H^{-1/3}(0,T)}\\
 &+\|g'\|_{[L^2(0,T)]^N}
 +\|p'\|_{[H^{1/3}(0,T)]^N}
 +\|\partial_tf\|_{L^1(0,T;\mathbb L^2(\mathcal T))}
 \biggr).
\end{align*}
Consequently, the functions $u_j$ defined by \eqref{u_j}, and hence
$u=(u_1,\ldots,u_N)$, are well-defined in
$\mathbb L^2(\mathcal T)$ for every $t\in[0,T]$.

\begin{lemma}\label{lemma for v}
Let $(u^0,g_0,g,p)\in\mathbb X_{3,T}$,
\[
 f\in W^{1,1}(0,T;\mathbb L^2(\mathcal T))
 \cap L^2(0,T;\mathbb H^1(\mathcal T)),
\]
and let $v$ be the solution of \eqref{v_j}. Then, for every
$j=1,\ldots,N$ and $\beta=0,1,2,3$, the following statements hold:
\begin{enumerate}
 \item[i.] The map
 \begin{align*}
 \begin{array}{rcl}
 [0,T]&\longrightarrow&\mathcal D'(0,l_j),\\
 s&\longmapsto&\partial_x^\beta v_j(s,\cdot)
 \end{array}
 \end{align*}
 is integrable on $[0,T]$.

 \item[ii.] For every $t\in[0,T]$,
 \[
  \int_0^t v_j(s,\cdot)\,ds\in H^3(0,l_j),
 \]
 and
 \begin{align*}
  \partial_x^\beta\int_0^t v_j(s,\cdot)\,ds
  =\int_0^t\partial_x^\beta v_j(s,\cdot)\,ds,
  \qquad \beta=0,1,2,3,
 \end{align*}
 where the identities are understood in $\mathcal D'(0,l_j)$.
\end{enumerate}
\end{lemma}

\begin{proof}
Since $v_j\in C([0,T];L^2(0,l_j))$, the map
$s\mapsto\partial_x^\beta v_j(s,\cdot)$ is continuous, and hence
integrable, with values in $\mathcal D'(0,l_j)$ for $\beta=0,1,2,3$. This proves item~i. Fix $t\in[0,T]$ and set
\[w_j(t):=\int_0^t v_j(s,\cdot)\,ds.\]
Integrating the equation for $v_j$ in time, in the sense of
distributions in $x$, gives
\begin{align*}
 (\partial_x+\partial_x^3)w_j(t)
 =f_j(t)-f_j(0)-v_j(t)+v_j(0).
\end{align*}
The right-hand side belongs to $L^2(0,l_j)$, because
$v_j\in C([0,T];L^2(0,l_j))$ and
$f_j\in W^{1,1}(0,T;L^2(0,l_j))$. Since $w_j(t)\in L^2(0,l_j)$,
the one-dimensional graph-domain characterization of
$\partial_x+\partial_x^3$ (see \cite[Lemma A.2]{KdV_flatness}) yields
$w_j(t)\in H^3(0,l_j)$.
Finally, differentiation commutes with integration in
$\mathcal D'(0,l_j)$, which proves the identities in item~ii.
\end{proof}

 The previous lemma can give us the regularity for $u$.
\begin{proposition}\label{W.P._s=3}
Let $(u^0,g_0,g,p)$ belong to
\[
 \mathbb H^3(\mathcal T)\times H^{2/3}(0,T)
 \times[H^1(0,T)]^N\times[H^{4/3}(0,T)]^N,
\]
assume that these data satisfy the linear compatibility conditions for
$s=3$, and let
\[
 f\in W^{1,1}(0,T;\mathbb L^2(\mathcal T))
 \cap L^2(0,T;\mathbb H^1(\mathcal T)).
\] 
For $j=1,...,N$ the function $t\mapsto u_j(t,\cdot)$ given by \eqref{u_j} satisfies the following items:
\begin{enumerate}
	\item[i.] $u_j \in C^1([0,T],L^2(0,l_j))\cap C([0,T];H^3(0,l_j))$;
	\item[ii.] $u\in \mathbb{B}_{3,T}=C([0,T];\mathbb{H}^3(\mathcal{T}))\cap L^2(0,T;\mathbb{H}^4(\mathcal{T}))$ is the unique classical solution to the problem \eqref{LkdV}; and
	\begin{equation*}
	\begin{split}
		\|u\|_{\mathbb{B}_{3,T}}\leq &C\left( \|u^0\|_{\mathbb{H}^3(\mathcal{T})}+\|g_0\|_{H^\frac{3-1}{3}(0,T)}+\|g\|_{[H^\frac{3}{3}(0,T)]^N}\right.\\
			&\left.+\|p\|_{[H^\frac{3+1}{3}(0,T)]^N}+\|f\|_{W^{1,1}(0,T;\mathbb{L}^2(\mathcal{T}))\cap L^2(0,T;\mathbb{H}^1(\mathcal{T}))}\right)
\end{split}
	\end{equation*}
	for some constant $C>0$.
	\item[iii.] $u \in H^1(0,T;\mathbb{H}^1(\mathcal{T}))$ and
	\begin{equation*}
		\begin{split}
		\|u\|_{H^1(0,T;\mathbb{H}^1(\mathcal{T}))}\leq& C\left(\|u^0\|_{\mathbb{H}^3(\mathcal{T})}+\|g_0\|_{H^\frac{3-1}{3}(0,T)}+\|g\|_{[H^\frac{3}{3}(0,T)]^N}\right.\\ &\left.+\|p\|_{[H^\frac{3+1}{3}(0,T)]^N}+\|f\|_{W^{1,1}(0,T;\mathbb{L}^2(\mathcal{T}))\cap L^2(0,T;\mathbb{H}^1(\mathcal{T}))}\right),
		\end{split}
	\end{equation*}
		for some constant $C>0$.
		\item [iv.] $\partial_x^ku_j\in L_x^\infty\left([0,l_j];H^\frac{3+1-k}{3}(0,T)\right)$ and 
		\begin{align*}
			\|\partial_x^ku_j\|_{L_x^\infty\left([0,l_j];H^\frac{3+1-k}{3}(0,T)\right)}\leq& C\left(\|u^0\|_{\mathbb{H}^3(\mathcal{T})}+\|g_0\|_{H^\frac{3-1}{3}(0,T)}+\|g\|_{[H^\frac{3}{3}(0,T)]^N}\right.\\ &\left.+\|p\|_{[H^\frac{3+1}{3}(0,T)]^N}+\|f\|_{W^{1,1}(0,T;\mathbb{L}^2(\mathcal{T}))\cap L^2(0,T;\mathbb{H}^1(\mathcal{T}))}\right),
		\end{align*}
		for $k=0,1,2.$
\end{enumerate}
\end{proposition}
\begin{proof}
Consider $j \in\{1,...,N\}$. 

\vspace{0.2 cm}
\noindent\textbf{i.} Since $v_j\in C([0,T],L^2(0,l_j))$ it is immediate that  by \eqref{u_j} that $u_j \in C^{1}([0,T],L^2(0,l_j))$.

\vspace{0.2cm}
\noindent \textbf{ii.} Let $t \in [0,T]$ be given. From \eqref{u_j} and item ii of Lemma \ref{lemma for v} we have $u_j(t,\cdot) \in H^3(0,l_j)$.
From \eqref{v_j} and for $r \in [0,T]$, $t\geq r$
\begin{align*}
\partial_x^3u_j(t,\cdot)-\partial_x^3u_j(r,\cdot)&=-v_j(t,\cdot)+v_j(r,\cdot)-\int_0^t\partial_x v_j(s,\cdot)ds+\int_0^r\partial_x v_j(s,\cdot)ds+f_j(t,\cdot)-f_j(r,\cdot)
\end{align*}
hence
\begin{align*}
\left\|\partial_x^3u_j(t,\cdot)-\partial_x^3u_j(r,\cdot)\right\|_{L^2(0,l_j)}\leq& \|v_j(t,\cdot)-v_j(r,\cdot)\|_{L^2(0,l_j)}+\left\|\int_0^t\partial_x v_j(s,\cdot)ds-\int_0^r\partial_x v_j(s,\cdot)ds\right\|_{L^2(0,l_j)}\\
&+\|f_j(t,\cdot)-f_j(r,\cdot)\|_{L^2(0,l_j)}.
\end{align*}
Since $v \in \mathbb{B}$ and due to continuous embedding $W^{1,1}(0,T;\mathbb{L}^2(\mathcal{T}))\hookrightarrow C([0,T];\mathbb{L}^2(\mathcal{T}))$ (see \cite[Section 5.9, Theorem 2]{Evans}) we have
\begin{align*}
\|v_j(t,\cdot)-v_j(r,\cdot)\|_{L^2(0,l_j)},\ \ \|f_j(t,\cdot)-f_j(r,\cdot)\|_{L^2(0,l_j)}\rightarrow 0\quad\text{as}\quad r\rightarrow t.
\end{align*}
On the other hand, the Hölder inequality gives us
\begin{align*}
\left\|\int_0^t\partial_x v_j(s,\cdot)ds-\int_0^r\partial_x v_j(s,\cdot)ds\right\|_{L^2(0,l_j)}&\leq \int_r^t\left\|\partial_x v_j(s,\cdot)\right\|_{L^2(0,l_j)}ds\leq |t-r|^\frac{1}{2}\|v\|_\mathbb{B}.
\end{align*}
Hence
\begin{align}\label{eq 2.38}
\left\|\int_0^t\partial_x v_j(s,\cdot)ds-\int_0^r\partial_x v_j(s,\cdot)ds\right\|_{L^2(0,l_j)}\rightarrow 0\quad \text{as}\quad r\rightarrow t.
\end{align}
We conclude that
\begin{align*}
\left\|\partial_x^3u_j(t,\cdot)-\partial_x^3u_j(r,\cdot)\right\|_{L^2(0,l_j)}\rightarrow 0\quad \text{as}\quad r\rightarrow t
\end{align*}
and, therefore,  $\partial_x^3u_j \in C([0,T];L^2(0,l_j))$. In the same way,  $\partial_xu_j \in C([0,T];L^2(0,l_j))$ since
\begin{align*}
\partial_xu_j(t,\cdot)-\partial_xu_j(r,\cdot)=\int_0^t\partial_x v_j(s,\cdot)ds-\int_0^r\partial_x v_j(s,\cdot)ds,
\end{align*}
so, by \eqref{eq 2.38}, $\partial_xu_j(r,\cdot)\rightarrow \partial_xu_j(t,\cdot)$ as $r\rightarrow t$. From \cite[Lemma A.2]{KdV_flatness} and \eqref{H^3<P_j}, we have
\begin{align*}
\left\|u_j(t,\cdot)-u_j(r,\cdot)\right\|_{H^3(0,l_j)}&\leq C\left(\left\|u_j(t,\cdot)-u_j(r,\cdot)\right\|_{L^2(0,l_j)}+\left\|P_ju_j(t,\cdot)-P_ju_j(r,\cdot)\right\|_{L^2(0,l_j)}\right)\\
&\leq C\left(\left\|u_j(t,\cdot)-u_j(r,\cdot)\right\|_{L^2(0,l_j)}+ \left\|\partial_x^3u_j(t,\cdot)-\partial_x^3u_j(r,\cdot)\right\|_{L^2(0,l_j)}\right.\\
&\ \ \left.  +\left\|\partial_xu_j(t,\cdot)-\partial_xu_j(r,\cdot)\right\|_{L^2(0,l_j)}\right).
\end{align*}
Since $u_j,\partial_x u_j, \partial_x^3 u_j \in C([0,T];L^2(0,l_j))$, it follows that, $u_j \in C([0,T];H^3(0,l_j))$. Now, by \eqref{u_j} we have
\begin{equation}\label{eq 2.39}
\begin{split}
\partial_tu_j(t,\cdot)+\partial_x u_j(t,\cdot)+\partial_x^3 u_j(t,\cdot)=&v_j(t,\cdot)+\partial_x u_j^0+\partial_x^3 u_j^0\\&+\int_0^t \partial_x v_j(s,\cdot)ds+\int_0^t \partial_x^3 v_j(s,\cdot)ds.
\end{split}
\end{equation}
From \eqref{v_j}, we have $v_j'(s,\cdot)=- \partial_x v_j(s,\cdot)-\partial_x^3 v_j(s,\cdot) + f_j'(s,\cdot).$ Using Lemma \ref{lemma for v} and the hypothesis about $f$, it follows that $s\mapsto v'_j(s,\cdot)$ is integrable, where $'$ denotes the (distributional) derivative with respect to the temporal variable. Hence,
\begin{align*}
\int_0^t  v_j'(s,\cdot)ds=v_j(t,\cdot)-v_j(0,\cdot)
\end{align*}
On the other hand,
\begin{align*}
\int_0^t v_j'(s,\cdot)ds=-\int_0^t \partial_x v_j(s,\cdot)ds-\int_0^t \partial_x^3 v_j(s,\cdot)ds +\int_0^t f_j'(s,\cdot)ds
\end{align*}
hence
\begin{align*}
\int_0^t \partial_x v_j(s,\cdot)ds+\int_0^t \partial_x^3 v_j(s,\cdot)ds&=-\int_0^t v_j'(s,\cdot)ds+\int_0^t f_j'(s,\cdot)ds\\
&=v_j(0,\cdot)-v_j(t,\cdot)+f_j(t,\cdot)-f_j(0,\cdot).
\end{align*}
Using this in \eqref{eq 2.39} we obtain
\begin{align*}
	\partial_tu_j(t,\cdot)+\partial_x u_j(t,\cdot)+\partial_x^3 u_j(t,\cdot)&=\partial_x u_j^0+\partial_x^3 u_j^0+v_j(0,\cdot)+f_j(t,\cdot)-f_j(0,\cdot).
\end{align*}
From this, \eqref{v_j} gives us
\begin{align*}
	\partial_tu_j(t,\cdot)+\partial_x u_j(t,\cdot)+\partial_x^3 u_j(t,\cdot)&=f_j(t,\cdot).
\end{align*}
Moreover, using \eqref{v_j} and \eqref{u_j} and the compatibility conditions, we deduce that $u$ is a classical solution to \eqref{LkdV}.

\medskip

\noindent By the previous part we have $u\in C([0,T];\mathbb{H}^3(\mathcal{T}))$. On the other hand, note that
\begin{align*}
\partial_x^3 u_j(t,\cdot)=f_j(t,\cdot)-\partial_x u_j(t,\cdot)-\partial_t u_j(t,\cdot)=f_j(t,\cdot)-\partial_x u_j(t,\cdot)-v_j(t,\cdot)\in H^1(0,l_j)
\end{align*}
so $u_j(t,\cdot)\in H^4(0,l_j)$, for almost every $t\in(0,T)$. Moreover $$\partial_x^4 u_j(t,\cdot)=\partial_x f_j(t,\cdot)-\partial_x^2 u_j(t,\cdot)-\partial_x v_j(t,\cdot),$$ which gives us
\begin{align*}
\|\partial_x^4 u_j(t,\cdot)\|_{L^2(0,l_j)}^2\leq C\left(\|\partial_x f_j(t,\cdot)\|_{L^2(0,l_j)}^2+\|\partial_x^2 u_j(t,\cdot)\|_{L^2(0,l_j)}^2+\|\partial_x v_j(t,\cdot)\|_{L^2(0,l_j)}^2\right)
\end{align*}
Thus
\begin{align}\label{eq 2.40}
\|u_j(t,\cdot)\|_{H^4(0,l_j)}^2&\leq C^2\left(\|u_j(t,\cdot)\|_{H^3(0,l_j)}^2+\|\partial_x f_j(t,\cdot)\|_{L^2(0,l_j)}^2+\|\partial_x v_j(t,\cdot)\|_{L^2(0,l_j)}^2\right).
\end{align}
Integrating from $0$ to $T$ and using the regularities for $f,u$ and $v$ we get $u_j \in L^2(0,T;H^4(0,l_j))$ and, consequently, $u\in C([0,T];\mathbb{H}^3(\mathcal{T}))\cap L^2(0,T;\mathbb{H}^4(\mathcal{T}))$. In addition, observe that
\begin{equation*}
\begin{aligned}
	\|u_j(t,\cdot)\|_{H^3(0,l_j)}&\leq C\left(\|u_j(t,\cdot)\|_{L^2(0,l_j)}+\|P_ju_j(t,\cdot)\|_{L^2(0,l_j)}\right)\\
	&=C\left(\|u_j(t,\cdot)\|_{L^2(0,l_j)}+\|\partial_x u_j(t,\cdot)+\partial_x^3 u_j(t,\cdot)\|_{L^2(0,l_j)}\right)\\
	&=C\left(\|u_j(t,\cdot)\|_{L^2(0,l_j)}+\|f_j(t,\cdot)-v_j(t,\cdot)\|_{L^2(0,l_j)}\right)\\
	&\leq C\left(\|u_j(t,\cdot)\|_{L^2(0,l_j)}+\|f_j(t,\cdot)\|_{L^2(0,l_j)}+\|v_j(t,\cdot)\|_{L^2(0,l_j)}\right).
\end{aligned}
\end{equation*}
Therefore
\begin{align}\label{eq 2.41}
	\|u_j(t,\cdot)\|_{H^3(0,l_j)}^2\leq C\left(\|u_j(t,\cdot)\|_{L^2(0,l_j)}^2+\|f_j(t,\cdot)\|_{L^2(0,l_j)}^2+\|v_j(t,\cdot)\|_{L^2(0,l_j)}^2\right).
\end{align}
and consequently
\begin{align*}
	\|u(t,\cdot)\|_{\mathbb{H}^3(\mathcal{T})}\leq C\left(\|u(t,\cdot)\|_{\mathbb{L}^2(\mathcal{T})}+\|f(t,\cdot)\|_{\mathbb{L}^2(\mathcal{T})}+\|v(t,\cdot)\|_{\mathbb{L}^2(\mathcal{T})}\right).
\end{align*}
Since $u,v\in C([0,T];\mathbb{L}^2(\mathcal{T}))$ and $f\in W^{1,1}(0,T;\mathbb{L}^2(\mathcal{T}))\hookrightarrow C([0,T];\mathbb{L}^2(\mathcal{T}))$, passing to the maximum with $t \in [0,T]$, we obtain
\begin{align}\label{eq 2.43}
	\|u\|_{C([0,T];\mathbb{H}^3(\mathcal{T}))}&\leq C\left(\|u\|_{C([0,T];\mathbb{L}^2(\mathcal{T}))}+\|f\|_{W^{1,1}(0,T;\mathbb{L}^2(\mathcal{T}))}+\|v\|_{C([0,T];\mathbb{L}^2(\mathcal{T}))}\right).
\end{align}
Using \eqref{eq 2.41} in \eqref{eq 2.40} it follows that
\begin{align*}
	\|u_j(t,\cdot)\|_{H^4(0,l_j)}^2&\leq C\left(\|u_j(t,\cdot)\|_{L^2(0,l_j)}^2+\| f_j(t,\cdot)\|_{H^1(0,l_j)}^2+\| v_j(t,\cdot)\|_{H^1(0,l_j)}^2\right).
\end{align*}
Taking the sum with respect to $j$ and integrating from $0$ to $T$ we have
for some $C=C(T)>0$
\begin{align}\label{eq 2.42}
	\|u\|_{L^2(0,T;\mathbb{H}^4(\mathcal{T}))}&\leq C\left(\|u\|_{C([0,T];\mathbb{L}^2(\mathcal{T}))}+\|f\|_{L^2(0,T;\mathbb{H}^1(\mathcal{T}))}+\|v\|_{L^2(0,T;\mathbb{H}^1(\mathcal{T}))}\right).
\end{align}
Adding \eqref{eq 2.42} and \eqref{eq 2.43} it follows that
\begin{align}\label{eq 2.44}
\|u\|_{\mathbb{B}_{3,T}}\leq C\left(\|u\|_{\mathbb{B}_{0,T}}+\|f\|_{W^{1,1}(0,T;\mathbb{L}^2(\mathcal{T}))\cap L^2(0,T;\mathbb{H}^1(\mathcal{T}))}+\|v\|_{\mathbb{B}_{0,T}}\right).
\end{align}
Using Proposition \ref{W.P._s=0} we have
\begin{align*}
\|u\|_{\mathbb{B}_{0,T}}\leq C\left(\|u^0\|_{\mathbb{L}^2(\mathcal{T})}+\|g_0\|_{H^{-\frac{1}{3}}(0,T)}+\|g\|_{[L^2(0,T)]^N}+\|p\|_{[H^\frac{1}{3}(0,T)]^N}+\|f\|_{L^1(0,T;\mathbb{L}^2(\mathcal{T}))}\right)
\end{align*}
and
\begin{equation*}
\begin{split}
	\|v\|_{\mathbb{B}_{0,T}}\leq& C\left( \|-\partial_x u^0-\partial_x^3 u^0+f(0,\cdot)\|_{\mathbb{L}^2(\mathcal{T})}+\|g_0'\|_{H^{-\frac{1}{3}}(0,T)}\right.\\
		&\left.+\|g'\|_{\left[L^2(0,T)\right]^{N}}+\|p'\|_{[H^\frac{1}{3}(0,T)]^N}+\|\partial_t f\|_{L^1(0,T;\mathbb{L}^2(\mathcal{T}))}\right).
		\end{split}
\end{equation*}
Thus, the following inequality holds
\begin{equation}\label{eq 2.45}
\begin{split}
\|u\|_{\mathbb{B}_{0,T}}+\|v\|_{\mathbb{B}_{0,T}}\leq &C\left(\|u^0\|_{\mathbb{L}^2(\mathcal{T})}+\|\partial_x u^0\|_{\mathbb{L}^2(\mathcal{T})}+\|\partial_x^3 u^0\|_{\mathbb{L}^2(\mathcal{T})}+ \|f(0,\cdot)\|_{\mathbb{L}^2(\mathcal{T})}\right.\\
&+\|f\|_{L^1(0,T;\mathbb{L}^2(\mathcal{T}))}+\|\partial_t f\|_{L^1(0,T;\mathbb{L}^2(\mathcal{T}))}+\|g_0\|_{H^{-\frac{1}{3}}(0,T)}+\|g_0'\|_{H^{-\frac{1}{3}}(0,T)}\\
&\left.+\|g\|_{[L^2(0,T)]^N}+\|g'\|_{\left[L^2(0,T)\right]^{N}}+\|p\|_{[H^\frac{1}{3}(0,T)]^N}+\|p'\|_{[H^\frac{1}{3}(0,T)]^N}\right).
\end{split}
\end{equation}
Due to embedding $W^{1,1}(0,T;\mathbb{L}^2(\mathcal{T}))\hookrightarrow C([0,T];\mathbb{L}^2(\mathcal{T}))$ we have
\begin{align}\label{eq 2.47}
\|f(0,\cdot)\|_{\mathbb{L}^2(\mathcal{T})}\leq \|f\|_{C([0,T];\mathbb{L}^2(\mathcal{T}))}\leq C \|f\|_{W^{1,1}(0,T;\mathbb{L}^2(\mathcal{T}))}.
\end{align}
Now, due to the embedding
\begin{align*}
H^{\frac{3-1}{3}}(0,T)\hookrightarrow L^2(0,T)\hookrightarrow H^{-\frac{1}{3}}(0,T)\quad \text{and}\quad H^{\frac{3+1}{3}}(0,T)\hookrightarrow H^\frac{1}{3}(0,T)
\end{align*}
the following inequalities hold
\begin{align}\label{eq 2.48}
\|g_0\|_{H^{-\frac{1}{3}}(0,T)}\leq C\|g_0\|_{H^{\frac{3-1}{3}}(0,T)}\quad \text{and}\quad \|p\|_{[H^\frac{1}{3}(0,T)]^N}\leq C \|p\|_{[H^{\frac{3+1}{3}}(0,T)]^N}.
\end{align}
Finally, since $\frac{3+1}{3}-1\geq 0$, $\frac{3-1}{3}>0$, $\frac{3-1}{3}-1<0$ and $\frac{3-1}{3}-\frac{1}{2}\notin \mathbb{Z}$
the operator $\frac{d}{dt}$ maps $H^\frac{3+1}{3}(0,T)$ (resp. $H^\frac{3-1}{3}(0,T))$ to  $H^{\frac{1}{3}}(0,T)$ (resp. $H^{-\frac{1}{3}}(0,T)$) continuously. Thus,
\begin{align}\label{eq 2.49}
	\|g_0'\|_{H^{-\frac{1}{3}}(0,T)}\leq c\|g_0\|_{H^{\frac{3-1}{3}}(0,T)}\quad \text{and}\quad \|p'\|_{[H^\frac{1}{3}(0,T)]^N}\leq c \|p\|_{[H^{\frac{3+1}{3}}(0,T)]^N}.
\end{align}
From \eqref{eq 2.44}-\eqref{eq 2.49} it follows that
\begin{equation*}
\begin{split}
\|u\|_{\mathbb{B}_{3,T}}\leq &C\left(\|u^0\|_{\mathbb{H}^3(\mathcal{T})}+\|g_0\|_{H^\frac{3-1}{3}(0,T)}+\|g\|_{[H^\frac{3}{3}(0,T)]^N}\right.\\
&\left.+\|p\|_{[H^\frac{3+1}{3}(0,T)]^N}+\|f\|_{W^{1,1}(0,T;\mathbb{L}^2(\mathcal{T}))\cap L^2(0,T;\mathbb{H}^1(\mathcal{T}))}\right).
\end{split}
\end{equation*}

\vspace{0.2 cm}
\noindent \textbf{iii.} Since $u \in \mathbb{B}_{3,T}$ we have, in particular, $u \in L^2(0,T;\mathbb{H}^1(\mathcal{T}))$ and $\partial_t u=-\partial_x u-\partial_x^3 u+f \in L^2(0,T;\mathbb{H}^1(\mathcal{T}))$. Therefore, $u \in H^1(0,T;\mathbb{H}^1(\mathcal{T}))$. Furthermore,
\begin{align*}
\|u\|_{H^1(0,T;\mathbb{H}^1(\mathcal{T}))}^2=&\|u\|_{L^2(0,T;\mathbb{H}^1(\mathcal{T}))}^2+\|\partial_t u\|_{L^2(0,T;\mathbb{H}^1(\mathcal{T}))}^2\\
=& \|u\|_{L^2(0,T;\mathbb{H}^1(\mathcal{T}))}^2+\|-\partial_x u-\partial_x^3 u+f\|_{L^2(0,T;\mathbb{H}^1(\mathcal{T}))}^2\\
\leq& C\left(\|u\|_{\mathbb{B}_{3,T}}^2+\|f\|_{L^2(0,T;\mathbb{H}^1(\mathcal{T}))}^2\right)\\
\leq& C \left(\|u^0\|_{\mathbb{H}^3(\mathcal{T})}+\|g_0\|_{H^\frac{3-1}{3}(0,T)}+\|g\|_{[H^\frac{3}{3}(0,T)]^N}\right.\\
&\left.+\|p\|_{[H^\frac{3+1}{3}(0,T)]^N}+\|f\|_{W^{1,1}(0,T;\mathbb{L}^2(\mathcal{T}))\cap L^2(0,T;\mathbb{H}^1(\mathcal{T}))}\right)^2.
\end{align*}
 \textbf{iv.} Since $u \in H^1(0,T;\mathbb{H}^1(\mathcal{T}))$ we have
$$
|\partial_t^k u_j(t,x)|\leq \|\partial_t^k u_j(t,\cdot)\|_{L^\infty(0,l_j)}\leq C\|\partial_t^k u_j(t,\cdot)\|_{H^1(0,l_j)}, \qquad k=0,1.$$
Taking the square and integrating this inequality we see that $u_j(\cdot,x)\in H^1(0,T)$ with
\begin{align}\label{50}
\|u_j(\cdot,x)\|_{H^1(0,T)}\leq C\|u\|_{H^1(0,T;\mathbb{H}^1(\mathcal{T}))},
\end{align}
for every $x \in (0,l_j)$. Moreover
\begin{align}\label{51}
\partial_t u_j(\cdot,x)=v_j(\cdot,x)\in H^\frac{1}{3}(0,T),\ \forall x \in (0,l_j).
\end{align}
From the theory of Sobolev spaces (see, for instance, \cite{Bhattacharyya}), it follows that $u_j\in L_x^\infty \left(0,l_j;H^\frac{3+1-0}{3}(0,T)\right)$. On the other hand $u_j \in C \left([0,T];H^3(0,l_j)\right)$ implies that
\begin{align}\label{52}
|\partial_x^k u_j(t,x)|\leq \|\partial_x^k u_j(t,\cdot)\|_{L^\infty(0,l_j)}\leq C\|\partial_x^k u_j(t,\cdot)\|_{H^1(0,l_j)}\leq C\|u_j(t,\cdot)\|_{H^3(0,l_j)}, \qquad k=1,2,
\end{align}
that is, $\partial_x^k u_j(\cdot,x)\in L^2(0,T)$. In addition,
\begin{align}\label{53}
 \partial_t\partial_x^ku_j(\cdot,x)
 =\partial_x^kv_j(\cdot,x)
 \in H^{\frac{1-k}{3}}(0,T),
 \qquad k=1,2.
\end{align}
and therefore $\partial_x^k u_j \in L_x^\infty \left(0,l_j;H^\frac{3+1-k}{3}(0,T)\right)$ for $k=1,2$. Finally, thanks to \eqref{50}-\eqref{53}, the trace estimates hold, concluding the proof.
\end{proof}


\subsection{The case \texorpdfstring{$s\in(0,3)$}{s in (0,3)}}
We consider the intermediate noncritical indices
$s\in(0,3)\setminus\{1/2,3/2,5/2\}$. To state precisely the
consequence of the endpoint linear estimates, define
\begin{align*}
\begin{cases}
\mathcal F_{0,T}&:=L^1(0,T;\mathbb L^2(\mathcal T)),\\
\mathcal F_{3,T}&:=W^{1,1}(0,T;\mathbb L^2(\mathcal T))
\cap L^2(0,T;\mathbb H^1(\mathcal T)),\\
\mathcal F_{s,T}&:=(\mathcal F_{0,T},\mathcal F_{3,T})_{s/3,2},
\qquad 0<s<3,
\end{cases}
\end{align*}
where the intersection is endowed with the sum norm and interpolation
is understood in the real $K$-method sense. 

\begin{corollary}\label{W.P._s in [0,3]}
Let $T>0$ and
$s\in(0,3)\setminus\{1/2,3/2,5/2\}$. Suppose that
\[
u^0\in\mathbb H^s(\mathcal T),\quad
g_0\in H^{(s-1)/3}(0,T),\quad
g\in[H^{s/3}(0,T)]^N,\quad
p\in[H^{(s+1)/3}(0,T)]^N
\]
satisfy the linear $s$-compatibility conditions, and let
$f\in\mathcal F_{s,T}$. Then \eqref{LkdV} admits a unique solution
$u\in\mathcal X_{s,T}$, and
\begin{align*}
\|u\|_{\mathcal X_{s,T}}
\leq C_{s,T}\Bigl(&\|u^0\|_{\mathbb H^s(\mathcal T)}
+\|g_0\|_{H^{(s-1)/3}(0,T)}
+\|g\|_{[H^{s/3}(0,T)]^N}+\|p\|_{[H^{(s+1)/3}(0,T)]^N}
+\|f\|_{\mathcal F_{s,T}}\Bigr).
\end{align*}
The solution depends linearly and continuously on the compatible
data and the forcing term.
\end{corollary}
\begin{proof}
We use the augmented linear spaces $\widehat{\mathbb X}^{0}_{0,T}$ and $\widehat{\mathbb X}^{0}_{3,T}$ defined in Appendix~\ref{app:compatibility-interpolation}. Lemma~\ref{lem:interpolation-constraints} is purely functional and independent of the nonlinear well-posedness results. Here we use the linear projection
\[\mathcal Q(u^0,h_0,g,p,a)=(u^0,h_0,g,p),\]
not the nonlinear parametrization $\mathcal P$ (see \eqref{eq:nonlinear-chart}). Let $\mathcal R(U,f)$ be the energy-level solution of
\eqref{LkdV} with data $\mathcal Q(U)$ and forcing $f$. Propositions~\ref{W.P._s=0} and~\ref{W.P._s=3}, including their hidden trace estimates, show that this is a bounded linear
operator at both endpoints:
\[\mathcal R:\widehat{\mathbb X}^{0}_{k,T}\times\mathcal F_{k,T}\longrightarrow\mathcal X_{k,T},\qquad k=0,3.\]
Indeed, at the upper endpoint the equations $\Lambda U=0$ imply exactly the linear compatibility conditions for $\mathcal Q(U)$. The two endpoint realizations agree by linear uniqueness. With $\vartheta=s/3$, real interpolation of this operator and of finite products yields
\[\mathcal R:\widehat{\mathbb X}^{0}_{s,T}\times\mathcal F_{s,T}\longrightarrow(\mathcal X_{0,T},\mathcal X_{3,T})_{\vartheta,2}.\]
The latter space embeds continuously into $\mathcal X_{s,T}$. This follows by interpolating the component maps into $L_t^2\mathbb H_x^1$ and $L_t^2\mathbb H_x^4$, into $L_t^2\mathbb H_x^1$ and $H_t^1\mathbb H_x^1$, and into each of the hidden time-trace spaces. The time evaluation maps give the
$\mathbb H^s$ bound uniformly in time; approximation by upper-endpoint elements, using the finite interpolation index $2$, gives time continuity. For the prescribed linear compatible data, set $h_0=g_0$ and choose $a=0$ when $s<1/2$, and $a=u_1^0(0)$ when $s>1/2$. Lemma~\ref{lem:interpolation-constraints} gives $U=(u^0,g_0,g,p,a)\in\widehat{\mathbb X}^{0}_{s,T}$, with norm bounded by the sum of the four physical data norms in the statement. The interpolated operator estimate proves the asserted bound. Since $\mathcal F_{s,T}\hookrightarrow\mathcal F_{0,T}$, uniqueness also follows from energy-level linear uniqueness.
\end{proof}
\begin{remark}
The linear compatibility conditions used above are
\[\begin{cases}
u_j^0(0)=u_1^0(0),\qquad u_j^0(l_j)=p_j(0),&s>1/2,\\
\partial_xu_j^0(l_j)=g_j(0),&s>3/2,\\
\displaystyle\sum_{j=1}^N\partial_x^2u_j^0(0)
=-\alpha u_1^0(0)+g_0(0),&s>5/2.
\end{cases}\]
Only the conditions corresponding to the given range of $s$ are imposed, for $j=1,\ldots,N$.
\end{remark}

\begin{remark}
We do not identify $\mathcal F_{s,T}$ with a specific Sobolev space, since this is not needed in the proofs. However, one can show that
\[W^{s/3,1}(0,T;\mathbb L^2(\mathcal T))\cap L^{6/(6-s)}(0,T;\mathbb H^{s/3}(\mathcal T))\hookrightarrow \mathcal F_{s,T}.\]
Consequently, for such $f$, the forcing contribution in the estimate of the previous lemma may be replaced by
\[\|f\|_{W^{s/3,1}(0,T;\mathbb L^2(\mathcal T))}+\|f\|_{L^{6/(6-s)}(0,T;\mathbb H^{s/3}(\mathcal T))}.\]
\end{remark}
\subsection{Auxiliary estimates and a variable-coefficient problem}
We next record the product estimates needed in the sequel. Their constants are uniform on subintervals of a fixed time interval.
\begin{lemma}\label{lemma 2.4}
Let $T>0$. There exists $C=C(T,L,N)>0$, independent of
$0<\tau\leq T$, such that the following estimates hold.
\begin{enumerate}
\item[i.]
For $u,v\in\mathbb B_{0,\tau}$,
\[
\|u\partial_xv\|_{L^1(0,\tau;\mathbb L^2)}
\leq C(\tau^{1/4}+\tau^{1/2})
\|u\|_{\mathbb B_{0,\tau}}\|v\|_{\mathbb B_{0,\tau}}.
\]
\item[ii.]
For $u,v\in\mathbb B_{3,\tau}$,
\[
\|u\partial_xv\|_{L^2(0,\tau;\mathbb H^1)}
\leq C\tau^{1/2}
\|u\|_{\mathbb B_{3,\tau}}\|v\|_{\mathbb B_{3,\tau}}.
\]
\item[iii.]
For $u,v\in\mathbb B^*_{3,\tau}$,
\[
\|u\partial_xv\|_{W^{1,1}(0,\tau;\mathbb L^2)}
\leq C(\tau+\tau^{1/2})
\|u\|_{\mathbb B^*_{3,\tau}}\|v\|_{\mathbb B^*_{3,\tau}}.
\]
\end{enumerate}
Here $\mathbb B^*_{3,\tau}
=\mathbb B_{3,\tau}\cap H^1(0,\tau;\mathbb H^1)$,
with the sum norm.
\end{lemma}

\begin{proof}
For each edge, the one-dimensional Sobolev inequality gives
\[
\|u_j(t)\|_{L^\infty}^2
\leq C\bigl(\|u_j(t)\|_{L^2}^2
+\|u_j(t)\|_{L^2}\|\partial_xu_j(t)\|_{L^2}\bigr).
\]
Integrating and applying H\"older's inequality yields
\[
\|u_j\|_{L_t^2L_x^\infty}
\leq C(\tau^{1/2}+\tau^{1/4})\|u_j\|_{B_{0,\tau}},
\]
where $B_{0,\tau}=C([0,\tau];L^2(0,L))
\cap L^2(0,\tau;H^1(0,L))$ with the sum norm.
Combining this with the $L_t^2L_x^2$ bound for $\partial_xv_j$
and summing over the edges proves i. For ii, use the pointwise-in-time estimate
\[
\|u\partial_xv\|_{\mathbb H^1}
\leq C\|u\|_{\mathbb H^3}\|v\|_{\mathbb H^3}
\]
and integrate in time. For iii, the product rule in distributions gives
\[
\partial_t(u\partial_xv)
=(\partial_tu)\partial_xv+u\partial_x\partial_tv.
\]
The undifferentiated term has $L_t^1\mathbb L_x^2$ norm at most
$C\tau\|u\|_{C_t\mathbb H_x^3}\|v\|_{C_t\mathbb H_x^3}$.
The differentiated terms have total norm at most
\[
C\tau^{1/2}\Bigl(
\|\partial_tu\|_{L_t^2\mathbb H_x^1}
\|v\|_{C_t\mathbb H_x^3}
+\|u\|_{C_t\mathbb H_x^3}
\|\partial_tv\|_{L_t^2\mathbb H_x^1}\Bigr).
\]
These bounds prove iii and justify the product rule by approximation.
\end{proof}

We first record the elementary Sobolev gluing property that will be used for the time traces.
\begin{lemma}\label{lemma:finite-Sobolev-gluing}
Let $-\frac12<\sigma<\frac12$ and let $0=t_0<t_1<\cdots<t_m=T$ be a finite partition of $[0,T]$. Set $I_r=(t_{r-1},t_r)$ for
$r=1,\ldots,m$. If $z_r\in H^\sigma(I_r)$, let $E_rz_r$ denote its canonical
extension by zero to $\mathbb R$, and define
\[
 z:=\left.\sum_{r=1}^m E_rz_r\right|_{(0,T)}.
\]
Then $z\in H^\sigma(0,T)$ and $z|_{I_r}=z_r$ for every $r$.
Moreover,
\begin{equation}\label{eq:finite-Sobolev-gluing}
\|z\|_{H^\sigma(0,T)}
\leq
C_{\sigma,\mathcal P}
\left(
\sum_{r=1}^m\|z_r\|_{H^\sigma(I_r)}^2
\right)^{1/2},
\end{equation}
where $\mathcal P=\{t_0,\ldots,t_m\}$. The same conclusion holds for
functions with values in a Hilbert space.
\end{lemma}

\begin{proof}
For $-\frac12<\sigma<\frac12$, extension by zero defines a bounded
operator from $H^\sigma(I_r)$ into $H^\sigma(\mathbb R)$; see, for
instance, \cite{Triebel}. Denoting this extension by $E_r$, the
distribution
\[
Z=\sum_{r=1}^m E_rz_r
\]
belongs to $H^\sigma(\mathbb R)$ and its restriction to $(0,T)$ is
$z$. Hence,
\begin{align*}
\|z\|_{H^\sigma(0,T)}
&\leq \|Z\|_{H^\sigma(\mathbb R)}
\leq \sum_{r=1}^m\|E_rz_r\|_{H^\sigma(\mathbb R)}\leq C_{\sigma,\mathcal P}
\sum_{r=1}^m\|z_r\|_{H^\sigma(I_r)}
\leq C_{\sigma,\mathcal P}\sqrt m
\left(
\sum_{r=1}^m\|z_r\|_{H^\sigma(I_r)}^2
\right)^{1/2}.
\end{align*}
Absorbing $\sqrt m$ into the constant proves
\eqref{eq:finite-Sobolev-gluing}. The Hilbert-valued statement follows
from the same argument.
\end{proof}

\section{Local well-posedness of the nonlinear problem}\label{sec4}
In this section we prove Theorem~\ref{main} for the nonlinear
problem \eqref{kdV}, using the solution spaces introduced above. It is important to point out that, at the energy level, a mild solution of \eqref{kdV} is a function
$u\in\mathbb B_{0,T}$ satisfying
\[u=\Psi\left(u^0, g_0-\frac N3u_1(\cdot,0)^2,g,p,-u\partial_xu\right),\]
where $\Psi$ is the linear solution operator from
Proposition~\ref{W.P._s=0}.
The nonlinear terms belong to the required data spaces by the
spatial product and trace estimates. Let us first establish a uniform estimate for the linear problem with zero initial data. All scalar time-Sobolev norms below are restriction norms, as specified before Lemma~\ref{lemma 2.1}.

\begin{lemma}\label{lem:uniform-zero-data-H3}
Let $T>0$. There exists $C_T>0$ such that, for every $0<\tau\leq T$,
the solution $y$ of
\begin{equation}\label{eq:zero-data-H3-problem}
\begin{cases}
\partial_ty_j+\partial_xy_j+\partial_x^3y_j=F_j,
&t\in(0,\tau),\ x\in(0,L),\\
y_j(t,0)=y_1(t,0),&t\in(0,\tau),\\
\displaystyle\sum_{j=1}^N\partial_x^2y_j(t,0)
=-\alpha y_1(t,0)+h(t),&t\in(0,\tau),\\
y_j(t,L)=\partial_xy_j(t,L)=0,&t\in(0,\tau),\\
y_j(0,x)=0,&x\in(0,L),
\end{cases}
\end{equation}
with
\[
F\in W^{1,1}(0,\tau;\mathbb L^2(\mathcal T))
\cap L^2(0,\tau;\mathbb H^1(\mathcal T)),\qquad
h\in H^{2/3}(0,\tau),
\]
and $F(0)=0$, $h(0)=0$, satisfies
\begin{equation}\label{eq:uniform-zero-data-H3}
\|y\|_{\mathcal X_{3,\tau}}
\leq C_T\left(
\|\partial_tF\|_{L^1(0,\tau;\mathbb L^2(\mathcal T))}
+\|F\|_{L^2(0,\tau;\mathbb H^1(\mathcal T))}
+\|h\|_{H^{2/3}(0,\tau)}
\right).
\end{equation}
We denote this solution by $\mathcal L_\tau(F,h)$.
\end{lemma}

\begin{proof}
Write $P=\partial_x+\partial_x^3$. Consider the differentiated
linear problem for $v$, with zero initial data, interior forcing
$\partial_tF$, central boundary datum $h'$, and zero external
boundary data. The restriction norm convention gives
\[
\|h'\|_{H^{-1/3}(0,\tau)}
\leq C\|h\|_{H^{2/3}(0,\tau)}
\]
uniformly in $\tau$. Extend $h'$ to $(0,T)$ with controlled
$H^{-1/3}$ norm, and extend $\partial_tF$ by zero as an
$L^1(0,T;\mathbb L^2)$ function. Proposition~\ref{W.P._s=0}
on the fixed interval $(0,T)$ gives a solution $\widetilde v$ with
\[\|\widetilde v\|_{\mathcal X_{0,T}}\leq C_T A_\tau,\qquad A_\tau:=\|\partial_tF\|_{L^1(0,\tau;\mathbb L^2)}+\|h\|_{H^{2/3}(0,\tau)}.\]
Put $v=\widetilde v|_{(0,\tau)}$ and
$y(t)=\int_0^t v(r)\,dr$. Integrating the equation for $v$ in
distributions, and using $F(0)=0$ and $v(0)=0$, yields
\[Py=F-v.\]
The spatial graph estimate now gives
$y\in C([0,\tau];\mathbb H^3)$ and
\[\|y\|_{C([0,\tau];\mathbb H^3)}\leq C_T A_\tau.\]
Here we used $\|y\|_{C\mathbb L^2}\leq T\|v\|_{C\mathbb L^2}$ and, importantly,
$\|F\|_{C\mathbb L^2}\leq\|\partial_tF\|_{L^1\mathbb L^2}$. Applying the corresponding spatial estimate with $Py\in\mathbb H^1$ also gives
\[\|y\|_{L^2(0,\tau;\mathbb H^4)}\leq C_T\left(A_\tau+\|F\|_{L^2(0,\tau;\mathbb H^1)}\right).\]
Since $\partial_ty=v$, the same bound controls
$\|y\|_{H^1(0,\tau;\mathbb H^1)}$. For the hidden time traces, use the bounded primitive operator
\[J:H^r(0,T)\longrightarrow H^{r+1}(0,T),\qquad\partial_tJw=w,\quad (Jw)(0)=0,\]
for $r=1/3,0,-1/3$. Applying it on the fixed interval $(0,T)$ to the traces of $\partial_x^k\widetilde v_j$, and then restricting to $(0,\tau)$, yields
\[\sum_{j=1}^N\sum_{k=0}^2\|\partial_x^ky_j\|_{L_x^\infty(0,L;H_t^{(4-k)/3}(0,\tau))}\leq C_T A_\tau.\]
These primitives agree with the spatial derivatives of $y$ in distributions. Integrating the boundary equations for $v$ gives the boundary conditions in \eqref{eq:zero-data-H3-problem}; the central
condition uses $h(0)=0$. Thus $y=\mathcal L_\tau(F,h)$, and the preceding estimates prove \eqref{eq:uniform-zero-data-H3}. Uniqueness follows from the energy-level linear uniqueness result.
\end{proof}

\begin{proof}[Proof of Theorem \ref{main}] Let us split the proof in several steps. 

\vspace{0.2cm}
\noindent\textbf{Step 1.} \textit{The case $s=0$.}
\vspace{0.2cm}

Let $d=(u^0,g_0,g,p)$, and set
\[D_0:=\|u^0\|_{\mathbb L^2(\mathcal T)}+\|g_0\|_{H^{-1/3}(0,T)}+\|g\|_{[L^2(0,T)]^N}+\|p\|_{[H^{1/3}(0,T)]^N}.\]
No pointwise compatibility condition at $t=0$ is imposed at this regularity. Let $\bar u$ be the solution of the linear problem \eqref{LkdV} on $(0,T)$ with data $d$ and zero interior forcing. Proposition~\ref{W.P._s=0}, including its hidden trace estimates,
gives
\begin{equation*}
\|\bar u\|_{\mathcal X_{0,T}}\leq C_{\mathrm{lin},T}D_0.
\end{equation*}
Put $M=C_{\mathrm{lin},T}D_0$. For $0<\tau\leq T$, let $\mathcal L^{(0)}_\tau(F,h)$ denote the energy-level solution of \eqref{eq:zero-data-H3-problem}, now for
\[F\in L^1(0,\tau;\mathbb L^2(\mathcal T)),\qquad h\in H^{-1/3}(0,\tau).\]
Thus the initial datum and both external boundary data are zero. Recall that in this range, we do not require initial traces of $F$ or $h$. There exists a constant $C_T$, independent of $0<\tau\leq T$, such that
\begin{equation}\label{eq:uniform-zero-data-L2}
\|\mathcal L^{(0)}_\tau(F,h)\|_{\mathcal X_{0,\tau}}
\leq C_T\left(
\|F\|_{L^1(0,\tau;\mathbb L^2)}
+\|h\|_{H^{-1/3}(0,\tau)}\right).
\end{equation}
Indeed, extend $F$ by zero as an $L^1$ function and extend $h$ to $(0,T)$ using its restriction norm. Apply Proposition~\ref{W.P._s=0} on the fixed interval $(0,T)$ and restrict the solution to $(0,\tau)$. Linear uniqueness shows that this restriction is independent of the chosen extension of $h$.
This also proves the asserted uniformity and linearity of $\mathcal L^{(0)}_\tau$. We first establish the two small-time product estimates used below. For $a,b\in\mathbb B_{0,\tau}$, item~i of Lemma \ref{lemma 2.4} gives
\begin{equation}\label{eq:L2-interior-small-time}
\|a\partial_xb\|_{L^1(0,\tau;\mathbb L^2)}
\leq C(\tau^{1/4}+\tau^{1/2})
\|a\|_{\mathbb B_{0,\tau}}
\|b\|_{\mathbb B_{0,\tau}}.
\end{equation}
For $a,b\in\mathcal X_{0,\tau}$, their spatial boundary traces belong to $H^{1/3}(0,\tau)$, with norms controlled by the $\mathcal X_{0,\tau}$ norm. To see this for arbitrary iterands, use the spatial continuity into $L^2(0,\tau)$ furnished by $L_t^2H_x^1$, together with weak compactness of the uniformly bounded $H^{1/3}(0,\tau)$ traces at interior spatial points. Lemma~\ref{lemma 2.1}, with $s=0$ and the exponent $\mu_0=1/6$ obtained in its proof, therefore gives
\begin{equation}\label{eq:L2-boundary-small-time}
\|a_1(\cdot,0)b_1(\cdot,0)\|_{H^{-1/3}(0,\tau)}
\leq C\tau^{1/6}
\|a\|_{\mathcal X_{0,\tau}}
\|b\|_{\mathcal X_{0,\tau}}.
\end{equation}
Set $\omega(\tau):=\tau^{1/4}+\tau^{1/2}+\tau^{1/6}$. We work in the closed subspace and its closed ball
\[\mathcal Z_{0,\tau} :=\{z\in\mathcal X_{0,\tau}:z(0,\cdot)=0\}, \qquad B_{0,\tau}(M):=\{z\in\mathcal Z_{0,\tau}:\|z\|_{\mathcal X_{0,\tau}}\leq M\}.\]
The initial evaluation map into $\mathbb L^2$ is continuous.
Thus, this ball is a nonempty complete metric space, including when
$M=0$. Every function in the affine set
$\bar u|_{(0,\tau)}+B_{0,\tau}(M)$ has initial datum $u^0$. For $z\in B_{0,\tau}(M)$, set $u=\bar u+z$ on $(0,\tau)$ and define
\[\Gamma^{(0)}_\tau z:=\mathcal L^{(0)}_\tau\left(-u\partial_xu,-\frac N3u_1(\cdot,0)^2\right).\]
The two forcing terms belong to the required spaces by \eqref{eq:L2-interior-small-time} and
\eqref{eq:L2-boundary-small-time}, so the map is well defined and its output belongs to $\mathcal Z_{0,\tau}$. Since $\|u\|_{\mathcal X_{0,\tau}}\leq2M$, we obtain
\begin{equation*}
\|\Gamma^{(0)}_\tau z\|_{\mathcal X_{0,\tau}}
\leq C_T\omega(\tau)M^2.
\end{equation*}
If $z,w\in B_{0,\tau}(M)$, write $u=\bar u+z$ and
$\widetilde u=\bar u+w$. Using
\[
u\partial_xu-\widetilde u\partial_x\widetilde u
=(u-\widetilde u)\partial_xu
+\widetilde u\partial_x(u-\widetilde u)
\]
and
\[
u_1^2-\widetilde u_1^2
=(u_1-\widetilde u_1)(u_1+\widetilde u_1),
\]
linearity of $\mathcal L^{(0)}_\tau$ and the same estimates give
\begin{equation*}
\|\Gamma^{(0)}_\tau z-\Gamma^{(0)}_\tau w\|_{\mathcal X_{0,\tau}}
\leq C_T\omega(\tau)M\|z-w\|_{\mathcal X_{0,\tau}}.
\end{equation*}
Enlarge $C_T$ so that it applies in both estimates, and choose
$0<\tau\leq\min\{1,T\}$ such that
\begin{equation}\label{eq:L2-time-choice}
C_T\omega(\tau)(1+M)\leq\frac12.
\end{equation}
Then $\Gamma^{(0)}_\tau$ maps $B_{0,\tau}(M)$ into itself and is a
contraction. Its fixed point $z$ gives a solution $u=\bar u+z$ of
\eqref{kdV}, with
\begin{equation}\label{eq:local-L2-data-bound}
\|u\|_{\mathcal X_{0,\tau}}
\leq 2C_{\mathrm{lin},T}D_0.
\end{equation}
The choice \eqref{eq:L2-time-choice} depends only on $D_0$ and
the fixed parameters, so a common positive existence time can be
chosen on each bounded set of data.

To prove uniqueness in the entire class $\mathcal X_{0,\tau}$,
consider two solutions with the same data. Their difference solves
the zero-initial-data linear problem with the differences of the
interior and central nonlinearities as forcing terms. On a
sufficiently short interval, \eqref{eq:uniform-zero-data-L2} and
the two product estimates force this difference to vanish.
The interval length can be chosen using the norms of the two
solutions on $(0,\tau)$. Their continuity into $\mathbb L^2$
allows this argument to be restarted at each partition point,
proving uniqueness on a finite partition of $[0,\tau]$. Finally, let $d,\widetilde d$ range over a fixed bounded set of data, and choose a common time $\tau$ satisfying the preceding smallness conditions. The corresponding linear references satisfy
\[\|\bar u-\overline{\widetilde u}\|_{\mathcal X_{0,\tau}}
\leq C_{\mathrm{lin},T}\|d-\widetilde d\|_{\mathrm{data},0},\]
where $\|\cdot\|_{\mathrm{data},0}$ is the sum of the four norms defining $D_0$. For the fixed-point corrections $z,w$, the difference estimates, after reducing the common time if necessary, give
\[\|z-w\|_{\mathcal X_{0,\tau}} \leq\frac12\left(
\|z-w\|_{\mathcal X_{0,\tau}} +\|\bar u-\overline{\widetilde u}\|_{\mathcal X_{0,\tau}} \right).\]
Absorption proves
\begin{equation}\label{eq:local-L2-data-lipschitz}
\|u-\widetilde u\|_{\mathcal X_{0,\tau}}
\leq C_{T,R}\|d-\widetilde d\|_{\mathrm{data},0},
\end{equation}
where $R$ bounds the sizes of both data sets. This establishes the
local Lipschitz dependence on the energy-level data.

\medskip
\noindent\textbf{Step 2.} \textit{The case $s=3$.}
\vspace{0.2cm}

We consider problem \eqref{kdV} with $3$-compatible data
$(u^0,g_0,g,p)$ with
\[
u^0\in\mathbb H^3(\mathcal T),\quad
g_0\in H^{2/3}(0,T),\quad
g\in[H^1(0,T)]^N,\quad
p\in[H^{4/3}(0,T)]^N.
\]
Set
\[
b:=u_1^0(0),\qquad Q_j^0:=u_j^0\partial_xu_j^0,
\qquad \widehat g_0:=g_0-\frac N3b^2.
\]
Since $u^0\in\mathbb H^3$, we have $Q^0\in\mathbb H^2$.
The compatibility conditions imply
\[
\sum_{j=1}^N\partial_x^2u_j^0(0)
=-\alpha b+\widehat g_0(0).
\]
Let $\bar u$ solve the linear problem on the fixed interval $(0,T)$:
\begin{equation*}
\begin{cases}
\partial_t\bar u_j+\partial_x\bar u_j+\partial_x^3\bar u_j=-Q_j^0,
&t\in(0,T),\ x\in(0,L),\\
\bar u_j(t,0)=\bar u_1(t,0),&t\in(0,T),\\
\displaystyle\sum_{j=1}^N\partial_x^2\bar u_j(t,0)
=-\alpha\bar u_1(t,0)+\widehat g_0(t),&t\in(0,T),\\
\bar u_j(t,L)=p_j(t),\quad\partial_x\bar u_j(t,L)=g_j(t),
&t\in(0,T),\\
\bar u_j(0,x)=u_j^0(x),&x\in(0,L).
\end{cases}
\end{equation*}
All its initial boundary compatibility conditions are linear and
are satisfied. Proposition~\ref{W.P._s=3} gives the linear solution
in $\mathbb B_{3,T}$. Its differentiated construction and
Proposition~\ref{W.P._s=0} give
$\partial_t\bar u\in\mathcal X_{0,T}$: the differentiated initial
datum is $-Pu^0-Q^0$, the interior forcing is zero, and the boundary
data are $g_0'$, $g'$, and $p'$. Time integration, including the
initial traces of $u^0$, therefore gives $\bar u\in\mathcal X_{3,T}$.
In particular, if
\[
D:=\|u^0\|_{\mathbb H^3}
+\|g_0\|_{H^{2/3}(0,T)}
+\|g\|_{[H^1(0,T)]^N}
+\|p\|_{[H^{4/3}(0,T)]^N},
\]
then
\begin{equation}\label{eq:H3-reference-bound}
M:=\|\bar u\|_{\mathcal X_{3,T}}
\leq C_T(D+D^2).
\end{equation}
For $0<\tau\leq T$, define
\[
\mathcal Z_{3,\tau}
:=\{z\in\mathcal X_{3,\tau}:z(0,\cdot)=0\},
\qquad
B_\tau:=\{z\in\mathcal Z_{3,\tau}:
\|z\|_{\mathcal X_{3,\tau}}\leq M\}.
\]
The initial evaluation map into $\mathbb H^3$ is continuous, so
$\mathcal Z_{3,\tau}$ is a closed linear subspace of
$\mathcal X_{3,\tau}$. Hence $B_\tau$ is a nonempty complete metric
space. The corresponding affine set of iterands is
$\bar u|_{(0,\tau)}+B_\tau$; every element of this set has initial
datum $u^0$. For $z\in B_\tau$, write $u=\bar u+z$ on $(0,\tau)$ and define
\[
F_z:=Q^0-u\partial_xu,\qquad
h_z:=-\frac N3\bigl(u_1(\cdot,0)^2-b^2\bigr),\qquad
\Gamma_\tau z:=\mathcal L_\tau(F_z,h_z).
\]
We have $F_z(0)=0$ in $\mathbb L^2$ and $h_z(0)=0$.
Moreover, $\|u\|_{\mathcal X_{3,\tau}}\leq 2M$ and
$\|u^0\|_{\mathbb H^3}\leq M$.
The one-dimensional spatial product inequalities give
\begin{align*}
\|F_z\|_{L^2(0,\tau;\mathbb H^1)}
&\leq C\tau^{1/2}M^2,
\\
\|\partial_tF_z\|_{L^1(0,\tau;\mathbb L^2)}
&\leq C\tau^{1/2}M^2.
\end{align*}
Indeed,
$\partial_tF_z=-(\partial_tu)\partial_xu-u\partial_x\partial_tu$,
and the second estimate follows from
$u\in C_t\mathbb H_x^3$ and $\partial_tu\in L_t^2\mathbb H_x^1$.
The first follows from
$\|u\partial_xu\|_{\mathbb H^1}\leq C\|u\|_{\mathbb H^3}^2$
and $\|Q^0\|_{\mathbb H^1}\leq C\|u^0\|_{\mathbb H^3}^2$.
These arguments also justify the time product rule in distributions. The renormalized estimate in Lemma~\ref{lemma 2.1}, with $s=3$,
gives
\begin{equation*}\label{eq:H3-residual-boundary}
\|h_z\|_{H^{2/3}(0,\tau)}
\leq C\tau^{1/3}
\|u_1(\cdot,0)\|_{H^{4/3}(0,\tau)}^2
\leq C\tau^{1/3}M^2.
\end{equation*}
It follows from Lemma~\ref{lem:uniform-zero-data-H3} that
\begin{equation}\label{eq:H3-affine-invariance}
\|\Gamma_\tau z\|_{\mathcal X_{3,\tau}}
\leq C_T(\tau^{1/2}+\tau^{1/3})M^2.
\end{equation}
We next prove the contraction estimate. Let $z,w\in B_\tau$, and
put $u=\bar u+z$, $\widetilde u=\bar u+w$. Factoring the difference
of the interior nonlinearities and applying the same spatial
inequalities and time product rule yields
\begin{align*}
&\|F_z-F_w\|_{L^2(0,\tau;\mathbb H^1)}
+\|\partial_t(F_z-F_w)\|_{L^1(0,\tau;\mathbb L^2)}
\leq
C\tau^{1/2}M\|z-w\|_{\mathcal X_{3,\tau}}.
\end{align*}
For example, differentiating
$u\partial_xu-\widetilde u\partial_x\widetilde u$
produces terms containing either $\partial_t(u-\widetilde u)$
times a spatial derivative of $u$ or $\widetilde u$, or
$u-\widetilde u$ times a spatial derivative of
$\partial_tu$ or $\partial_t\widetilde u$. Each is controlled by
the $C_t\mathbb H_x^3$ and $H_t^1\mathbb H_x^1$ components,
with the factor $\tau^{1/2}$ from H\"older's inequality. At the central vertex,
\[h_z-h_w=-\frac N3 (u_1-\widetilde u_1)(u_1+\widetilde u_1).\]
The first factor vanishes at $t=0$. Lemma~\ref{lemma 2.1}
therefore gives
\begin{equation*}
\|h_z-h_w\|_{H^{2/3}(0,\tau)}
\leq C\tau^{1/3}M\|z-w\|_{\mathcal X_{3,\tau}}.
\end{equation*}
By linearity of $\mathcal L_\tau$ and
Lemma~\ref{lem:uniform-zero-data-H3}, we obtain
\begin{equation}\label{eq:H3-affine-contraction}
\|\Gamma_\tau z-\Gamma_\tau w\|_{\mathcal X_{3,\tau}}
\leq C_T(\tau^{1/2}+\tau^{1/3})M
\|z-w\|_{\mathcal X_{3,\tau}}.
\end{equation}
Enlarge $C_T$ so that it applies in both
\eqref{eq:H3-affine-invariance} and
\eqref{eq:H3-affine-contraction}, and choose $0<\tau\leq\min\{1,T\}$
such that
\[
C_T(\tau^{1/2}+\tau^{1/3})(1+M)\leq\frac12.
\]
Then $\Gamma_\tau$ maps $B_\tau$ into itself and is a contraction.
Its fixed point $z$ gives $u=\bar u+z\in\mathcal X_{3,\tau}$.
Moreover, $\|u\|_{\mathcal X_{3,\tau}}\leq2M\leq C_T(D+D^2)$.
Adding the equations for $\bar u$ and $z$ proves that $u$ solves
\eqref{kdV}: the terms $Q^0$ and $Nb^2/3$ cancel exactly.
By \eqref{eq:H3-reference-bound}, the existence time has a positive
lower bound when the compatible data range over a bounded set.

For uniqueness in the entire solution class (not only in the chosen ball), take two $\mathcal X_{3,\tau}$ solutions with the same data. Their difference has zero initial data and homogeneous external boundary data. The interior forcing and central boundary datum are differences of the same nonlinearities as above, both with zero initial trace. Applying the preceding estimates on a sufficiently
short interval forces the difference to vanish. Repeating on a finite partition, with interval lengths chosen using the fixed norms of the two solutions on $(0,\tau)$, proves uniqueness on $[0,\tau]$. Finally, the reference map $d\mapsto\bar u$ is locally Lipschitz from the compatible data, with their ambient data norm, into $\mathcal X_{3,T}$. This follows from the fixed-time linear estimate and the locally Lipschitz maps $u^0\mapsto Q^0$ and
$u^0\mapsto b^2$. For two data sets $d,\widetilde d$ in a bounded set, choose a common sufficiently short interval. If $z,w$ are their fixed points, the preceding difference estimates give
\[\|z-w\|_{\mathcal X_{3,\tau}}\leq\frac12\|z-w\|_{\mathcal X_{3,\tau}}+C_R\|d-\widetilde d\|_{\mathrm{data}},\]
where $\|\cdot\|_{\mathrm{data}}$ is the sum of the four norms
defining $D$. In the boundary estimate, apply the renormalized
product inequality to
\[(u_1-\widetilde u_1)(u_1+\widetilde u_1)-(b-\widetilde b)(b+\widetilde b).\]
In the interior estimate, the additional constant term is
$Q^0-\widetilde Q^0$, whose $L_t^2\mathbb H_x^1$ norm is bounded by $C_R\tau^{1/2}\|u^0-\widetilde u^0\|_{\mathbb H^3}$. Absorption and the estimate for the references prove local Lipschitz dependence in $\mathcal X_{3,\tau}$ on the compatible data.

\medskip
\noindent\textbf{Step 3.} \textit{Persistence on a common energy-level interval.}
\medskip

For the remainder of the proof, write
\[
D_s(d):=\|d\|_{s,0}.
\]
Fix $R\geq1$. The case $s=0$ provides a common interval
$[0,\tau_R]$, with $0<\tau_R\leq\min\{1,T/2\}$, for all data
with $D_0(d)\leq R$. After decreasing $\tau_R$ when necessary,
the solution map $\mathcal S_{\tau_R}$ is Lipschitz on this ball
with values in $\mathcal X_{0,\tau_R}$, and
\begin{equation*}
\|u\|_{\mathcal X_{0,\tau_R}}
\leq C_TD_0(d)\leq A_R:=C_TR.
\end{equation*}
We show that, for $3$-compatible data in this ball,
\begin{equation}\label{eq:local-common-H3-tame}
\|\mathcal S_{\tau_R}(d)\|_{\mathcal X_{3,\tau_R}}
\leq C_{T,R}D_3(d).
\end{equation}
Keep the energy-level solution $u$ fixed. Consider the linear problem
\begin{equation}\label{eq:local-linearized-persistence}
\begin{cases}
v_{jt}+v_{jx}+v_{jxxx}=-\partial_x(u_jv_j),\\
v_j(t,0)=v_1(t,0),\\
\displaystyle\sum_{j=1}^Nv_{jxx}(t,0)
=-\alpha v_1(t,0)-\dfrac{2N}{3}u_1(t,0)v_1(t,0)+g_0'(t),\\
v_j(t,L)=p_j'(t),\qquad v_{jx}(t,L)=g_j'(t),\\
v_j(0)=-\partial_xu_j^0-\partial_x^3u_j^0-u_j^0\partial_xu_j^0,
\end{cases}
\end{equation}
where $0<t<\tau_R$, $0<x<L$, and $j=1,\ldots,N$.
By the linear energy-level estimate and
\eqref{eq:L2-interior-small-time}--\eqref{eq:L2-boundary-small-time},
the part of its solution operator that depends on $v$ has norm at most
\[C_T\omega(\tau_R)A_R,\qquad\omega(\tau)=\tau^{1/4}+\tau^{1/2}+\tau^{1/6}.\]
Choose $\tau_R$ so that this quantity is at most $1/2$.
The affine contraction principle then gives a unique
$v\in\mathcal X_{0,\tau_R}$. The construction use the energy bound for $u$. The one-dimensional interpolation inequality gives
\[\|u^0\partial_xu^0\|_{\mathbb L^2}\leq C\|u^0\|_{\mathbb L^2}\|u^0\|_{\mathbb H^3}.\]
Together with the boundedness of time differentiation on the
boundary data spaces, this yields
\begin{equation*}
\|v\|_{\mathcal X_{0,\tau_R}}
\leq C_T(1+R)D_3(d).
\end{equation*}
The case $s=3$ already proved gives an $\mathcal X_3$ solution on an initial nontrivial interval. It agrees there with $u$ by
energy-level uniqueness. On every interval on which this higher regularity persists, differentiation in the distributional sense, including in the time Sobolev spaces of the boundary traces, shows that $u_t\in\mathcal X_0$ solves
\eqref{eq:local-linearized-persistence}. Hence $u_t=v$ there. We record bounds that are uniform on any such interval $[0,t]$
with $t\leq\tau_R$. Put $G=u\partial_xu$. Since $G_t=v\partial_xu+u\partial_xv$, the interior product estimate gives
\begin{align*}
\|G\|_{C([0,t];\mathbb L^2)}
&\leq\|u^0\partial_xu^0\|_{\mathbb L^2}
+\|v\partial_xu+u\partial_xv\|_{L^1(0,t;\mathbb L^2)}\leq C_{T,R}D_3(d).
\end{align*}
For $P=\partial_x+\partial_x^3$, the graph-norm inequality
\[\|w\|_{H^3(0,L)}\leq C\bigl(\|w\|_{L^2(0,L)}+\|Pw\|_{L^2(0,L)}\bigr)\]
and $Pu=-v-G$ imply $\|u\|_{C([0,t];\mathbb H^3)}\leq C_{T,R}D_3(d)$.
Moreover,
\[\partial_x^4u=-\partial_xv-\partial_x^2u-\partial_xG, \qquad\|\partial_xG\|_{L^2(0,t;\mathbb L^2)}\leq C\|u\|_{C([0,t];\mathbb H^3)}\|u\|_{L^2(0,t;\mathbb H^1)}.\]
These estimates control $\|u\|_{L^2(0,t;\mathbb H^4)}$.
The identity $u_t=v$ controls the $H^1(0,t;\mathbb H^1)$ norm.
Finally, use
\[
\partial_x^ku_j(t,x)
=\partial_x^ku_j^0(x)+\int_0^t\partial_x^kv_j(\sigma,x)\,d\sigma,
\qquad k=0,1,2.
\]
For $k=2$, the integral denotes the distributional primitive with
zero initial value. On the fixed interval $(0,\tau_R)$, this
primitive maps $H^{(1-k)/3}$ boundedly into $H^{(4-k)/3}$.
Apply this to the already constructed $v$ on that interval and
then restrict to $(0,t)$. Since $H^3(0,L)\hookrightarrow C^2([0,L])$,
the hidden trace norms are bounded as well. Consequently,
\begin{equation}\label{eq:local-persistence-uniform}
\|u\|_{\mathcal X_{3,t}}
+\|u_t\|_{\mathcal X_{0,t}}
\leq C_{T,R}D_3(d)
\end{equation}
on every interval of higher regularity, with a constant independent
of $t\leq\tau_R$.

To continue the solution, restart the $s=3$ construction at any
time $t_0<\tau_R$ within such an interval. The new initial datum
has a uniformly bounded $\mathbb H^3$ norm by
\eqref{eq:local-persistence-uniform}. The translated boundary data
can be considered on a fixed interval of length $T/2$, since
$t_0\leq T/2$; their norms are bounded by the original data norms.
Their compatibility with $u(t_0)$ follows from the boundary
conditions, whose traces are continuous at this regularity.
The $s=3$ construction therefore supplies a positive restart time
depending only on $T,R,D_3(d)$ and the fixed parameters, not on
$t_0$. Restarting sufficiently close to a putative maximal time
extends higher regularity past it, with identification of the
solutions by energy-level uniqueness. The same argument covers
the endpoint $\tau_R$. This proves
\eqref{eq:local-common-H3-tame}.

\medskip
\noindent\textbf{Step 4.} \textit{The intermediate noncritical indices.}
\medskip

Let $0<s<3$, $s\notin\{1/2,3/2,5/2\}$, and set $\vartheta=s/3$.
Use the augmented spaces and chart defined in
Appendix~\ref{app:compatibility-interpolation}, with $\lambda=0$:
\[
A_0:=\widehat{\mathbb X}^{0}_{0,T},\qquad
A_1:=\widehat{\mathbb X}^{0}_{3,T},\qquad
A_\vartheta:=(A_0,A_1)_{\vartheta,2}
=\widehat{\mathbb X}^{0}_{s,T}.
\]
The ambient scale is
\[
\mathcal E^0_\vartheta
=\mathbb H^{3\vartheta}(\mathcal T)
\times H^{\vartheta-1/3}(0,T)
\times[H^\vartheta(0,T)]^N
\times[H^{\vartheta+1/3}(0,T)]^N\times\mathbb R.
\]
In particular, $A_1$ is the kernel of the linear constraints $\Lambda$ in $\mathcal E^0_1$, whereas $A_0=\mathcal E^0_0$.
The constraint identification in Lemma~\ref{lem:interpolation-constraints}, including its extension
to $\lambda=0$, is again purely functional and does not depend on local-or global nonlinear well-posedness. We give the nonlinear interpolation argument explicitly. For
$K\geq1$, let
\[
\Pi_KU=\min\{1,K/\|U\|_{A_0}\}\,U\quad(U\ne0),
\qquad \Pi_K0=0.
\]
This map is Lipschitz on $A_0$, takes values in its radius-$K$ ball,
preserves $A_1$, and satisfies
$\|\Pi_KU\|_{A_1}\leq\|U\|_{A_1}$ for $U\in A_1$.
Recall that
\[
\mathcal P(u^0,h_0,g,p,a)
=(u^0,h_0+(N/3)a^2\rho,g,p).
\]
There is $R_K\geq1$ such that
$D_0(\mathcal P(\Pi_KU))\leq R_K$ for every $U\in A_0$.
On the common interval $[0,\tau]$, where $\tau=\tau_{R_K}$, define
\[
\mathcal F_K(U):=\mathcal S_\tau\bigl(\mathcal P(\Pi_KU)\bigr),
\qquad U\in A_0.
\]
Set $B_0=\mathcal X_{0,\tau}$ and $B_1=\mathcal X_{3,\tau}$.
Energy-level Lipschitz dependence and
\eqref{eq:local-common-H3-tame} give
\begin{align}
\|\mathcal F_K(U)-\mathcal F_K(V)\|_{B_0}
&\leq C_K\|U-V\|_{A_0},\qquad U,V\in A_0,
\label{eq:local-chart-low-Lipschitz}\\
\|\mathcal F_K(V)\|_{B_1}
&\leq C_K\|V\|_{A_1},\qquad V\in A_1.
\label{eq:local-chart-high-tame}
\end{align}
For the second inequality \eqref{eq:local-chart-high-tame}, use that $\mathcal P(\Pi_KV)$ is $3$-compatible and that its quadratic chart term satisfies $|a|^2\leq C K|a|$ on this ball. The constants may depend on $K,T$ and the fixed parameters. Write $K_A$ and $K_B$ for the $K$-functionals of these two couples. For every $V\in A_1$ and $\eta>0$,
\begin{align*}
K_B(\eta,\mathcal F_K(U))
&\leq\|\mathcal F_K(U)-\mathcal F_K(V)\|_{B_0}
+\eta\|\mathcal F_K(V)\|_{B_1}\leq C_K\bigl(\|U-V\|_{A_0}+\eta\|V\|_{A_1}\bigr).
\end{align*}
Taking the infimum over $V$ and then the interpolation norm gives
\begin{equation}\label{eq:local-nonlinear-interpolation-bound}
\|\mathcal F_K(U)\|_{(B_0,B_1)_{\vartheta,2}}
\leq C_K\|U\|_{A_\vartheta}.
\end{equation}
This is the nonlinear $K$-functional interpolation argument, applied here to a map defined on the entire lower endpoint space. For completeness, it also gives continuity. If $U_n\to U$ in $A_\vartheta$, subadditivity and the preceding bound for the $K$-functional imply
\[K_B(\eta,\mathcal F_K(U_n)-\mathcal F_K(U))\leq C_K\bigl(2K_A(\eta,U)+K_A(\eta,U_n-U)\bigr).\]
The interpolation integral over $(0,\delta)$ is therefore controlled by the corresponding tail for $U$ and by
$C_K\|U_n-U\|_{A_\vartheta}$. Over $(\delta,\infty)$, use \eqref{eq:local-chart-low-Lipschitz}; the resulting norm is at most
$C_{K,\vartheta}\delta^{-\vartheta}\|U_n-U\|_{A_0}$. First let $n\to\infty$ and then $\delta\downarrow0$. The componentwise interpolation embeddings yield
\begin{equation}\label{eq:local-target-interpolation}
(\mathcal X_{0,\tau},\mathcal X_{3,\tau})_{\vartheta,2} \hookrightarrow\mathcal X_{s,\tau}.
\end{equation}
Indeed, the spatial $L_t^2$ components interpolate into $L^2(0,\tau;\mathbb H^{s+1})$, and the time components into $H^{s/3}(0,\tau;\mathbb H^1)$. The evaluation maps give a uniform $\mathbb H^s$ bound in time; density of the upper endpoint in the
real interpolation space, with finite second index, gives continuity in time. Finally, interpolation of each spatial trace map gives $L_x^\infty H_t^{(s+1-k)/3}$ for $k=0,1,2$. Only these continuous embeddings 
are required. For an $s$-compatible physical datum $d$, choose the lift
\[
a=0\quad\hbox{if }0<s<1/2,\qquad
a=u_1^0(0)\quad\hbox{if }1/2<s<3,
\qquad h_0=g_0-(N/3)a^2\rho.
\]
Lemma~\ref{lem:interpolation-constraints} shows that the resulting
$U=(u^0,h_0,g,p,a)$ belongs to $A_\vartheta$ and satisfies
$\mathcal P(U)=d$. This lift is continuous in the ambient physical
data norm, and
\[
\|U\|_{A_\vartheta}\leq C_{T,s}(1+D_s(d))D_s(d).
\]
On each bounded set of such physical data, choose $K$ large enough
that $\Pi_KU=U$ for all these lifts. Equations
\eqref{eq:local-nonlinear-interpolation-bound} and
\eqref{eq:local-target-interpolation} then give a common positive
existence time, the bound
\[
\|u\|_{\mathcal X_{s,\tau}}\leq C_{T,s,B}D_s(d)
\qquad\text{when }D_s(d)\leq B,
\]
and continuous dependence in $\mathcal X_{s,\tau}$. Uniqueness
follows from the energy-level uniqueness already proved. Together
with the endpoint cases, this proves Theorem \ref{main} for
$s\in[0,3]\setminus\{1/2,3/2,5/2\}$.
\end{proof}


\section{Global well-posedness of the nonlinear problem}\label{sec6}
In this section, we prove the global well-posedness of \eqref{kdV}, that is, Theorem \ref{global_(0,3)}.  We use the data spaces and compatibility sets introduced above. The parameter value $\lambda=0$ corresponds to the local data scale, whereas the global results below require $\lambda\geq1$. Throughout this section, $C>0$ denotes a constant that may change
from line to line and depends only on the fixed parameters. We use $F_T:[0,\infty)\to(0,\infty)$ to denote a positive nondecreasing function, allowed to change from line to line, whose dependence on the size of the data is displayed explicitly.

\subsection{The case \texorpdfstring{$s=0$}{s=0}: Global a priori estimate}
We first derive energy estimates for the problem
\begin{align}\label{Neumann}
	\left\{
	\begin{array}{ll}
		\partial_t u_j+\partial_x u_j+\partial_x^3 u_j+u_j\partial_x u_j=f_j,&t \in (0,T),\ x \in (0,l_j),\\
		u_j(t,0)=u_1(t,0),&t\in (0,T),\\
		\displaystyle\sum_{j=1}^N\partial_x^2 u_j(t,0)=-\alpha u_1(t,0)-\frac{N}{3}u_1^2(t,0)+g_0(t),&t \in (0,T)\\				
		u_j(t,l_j)=0,\ \ \ \ \ \partial_x u_j(t,l_j)=g_j(t),&t\in (0,T),\\
		u_j(0,x)=u_j^0(x),&x \in (0,l_j),
	\end{array}
	\right.
\end{align}
in the case $s=0$. We first establish an a priori estimate for \eqref{Neumann}. The cancellation of the cubic boundary terms gives a linear bound for the $C_t\mathbb L_x^2$ norm and the dissipative vertex traces. The full energy norm is estimated separately by a spatial multiplier.
\begin{proposition}\label{prop 3.1}
Let $T>0$ and assume, as throughout the paper, that $\alpha>N/2$.
Let $u^0\in\mathbb L^2(\mathcal T)$,
$g_0\in L^2(0,T)$, $g\in[L^2(0,T)]^N$, and
$f\in L^1(0,T;\mathbb L^2(\mathcal T))$. Set
\[
 D:=\|u^0\|_{\mathbb L^2}
 +\|g_0\|_{L^2(0,T)}
 +\|g\|_{[L^2(0,T)]^N}
 +\|f\|_{L^1(0,T;\mathbb L^2)}.
\]
If $u\in\mathbb B_{0,T}$ is a mild solution of \eqref{Neumann}, then
\begin{equation}\label{eq:homogeneous-Dirichlet-L2-energy}
 \|u\|_{C([0,T];\mathbb L^2)}
 +\|u_1(\cdot,0)\|_{L^2(0,T)}
 +\left(\sum_{j=1}^N
   \|\partial_xu_j(\cdot,0)\|_{L^2(0,T)}^2\right)^{1/2}
 \leq C_TD.
\end{equation}
Moreover,
\begin{equation}\label{eq:homogeneous-Dirichlet-energy}
 \|u\|_{\mathbb B_{0,T}}
 \leq C_T\bigl(D+D^{5/3}\bigr).
\end{equation}
The constants depend only on $T,L,N$, and $\alpha$.
\end{proposition}

\begin{proof}
Write
\[
 b(t):=u_1(t,0),\qquad d_j(t):=\partial_xu_j(t,0),
 \qquad E(t):=\|u(t)\|_{\mathbb L^2}^2,
 \qquad \kappa:=\alpha-\frac N2>0.
\]
We first perform the computations for smooth solutions and justify
the passage to energy solutions at the end of the proof.
Multiplying the equation on each edge by $u_j$, integrating in
space, and summing over the edges gives
\begin{equation}\label{eq:homogeneous-Dirichlet-energy-identity}
 \frac12 E'(t)+\kappa |b(t)|^2
 +\frac12\sum_{j=1}^N|d_j(t)|^2
 =\frac12\|g(t)\|_{\mathbb R^N}^2
 +g_0(t)b(t)+(f(t),u(t))_{\mathbb L^2}.
\end{equation}
Indeed, the cubic term produced by the central boundary condition
cancels the boundary term from $u_j\partial_xu_j$.
Young's inequality gives
\[|g_0b|\leq\frac\kappa2|b|^2+\frac1{2\kappa}|g_0|^2.\]
Set $M:=\sup_{0\leq t\leq T}\|u(t)\|_{\mathbb L^2}$.
Integrating the preceding identity and taking the supremum yields
\[M^2\leq C D^2+2M\|f\|_{L^1(0,T;\mathbb L^2)}.\]
Consequently, $M\leq CD$. Substitution into the integrated energy
identity also bounds $b$ and $(d_1,\ldots,d_N)$ in $L^2(0,T)$,
which proves \eqref{eq:homogeneous-Dirichlet-L2-energy}. To estimate the spatial derivative, set
\[J(t):=\sum_{j=1}^N\int_0^L x|u_j(t,x)|^2\,dx.\]
Multiplication by $xu_j$ and integration by parts give
\begin{equation}\label{eq:homogeneous-Dirichlet-weighted-identity}
\begin{split}
 \frac12J'(t)+\frac32\|\partial_xu(t)\|_{\mathbb L^2}^2
 ={}&\frac12E(t)+\frac L2\|g(t)\|_{\mathbb R^N}^2
 -b(t)\sum_{j=1}^Nd_j(t)\\
 &+\frac13\sum_{j=1}^N\int_0^L u_j(t,x)^3\,dx
 +\sum_{j=1}^N\int_0^L xf_j(t,x)u_j(t,x)\,dx.
 \end{split}
\end{equation}
Since $u_j(t,L)=0$, the one-dimensional inequality
\[\|z\|_{L^\infty(0,L)}^2\leq 2\|z\|_{L^2(0,L)}\|z'\|_{L^2(0,L)},\qquad z\in H^1(0,L),\quad z(L)=0,\]
implies, for every $\varepsilon>0$,
\[\sum_{j=1}^N\int_0^L|u_j|^3dx\leq C\|u\|_{\mathbb L^2}^{5/2}\|\partial_xu\|_{\mathbb L^2}^{1/2}\leq\varepsilon\|\partial_xu\|_{\mathbb L^2}^2+C_\varepsilon\|u\|_{\mathbb L^2}^{10/3}.\]
Integrate \eqref{eq:homogeneous-Dirichlet-weighted-identity} on
$(0,T)$. Use $J(T)\geq0$, $J(0)\leq L\|u^0\|_{\mathbb L^2}^2$,
and \eqref{eq:homogeneous-Dirichlet-L2-energy}. The vertex term
is bounded by
\[
 \left|\int_0^T b(t)\sum_{j=1}^Nd_j(t)\,dt\right|
 \leq\sqrt N\|b\|_{L^2(0,T)}
       \|(d_1,\ldots,d_N)\|_{L^2(0,T;\mathbb R^N)}
 \leq C_TD^2,
\]
and the forcing term is bounded by
$LM\|f\|_{L^1(0,T;\mathbb L^2)}\leq C_TD^2$.
Choosing $\varepsilon$ sufficiently small and absorbing the
derivative term, we obtain
\[\|\partial_xu\|_{L^2(0,T;\mathbb L^2)}^2 \leq C_T\bigl(D^2+D^{10/3}\bigr).\]
Together with $M\leq CD$, this proves \eqref{eq:homogeneous-Dirichlet-energy}. For completeness, these multiplier computations can be justified without assuming a global smooth nonlinear solution. For the given energy solution, the product estimate in Lemma~\ref{lemma 2.4} gives $u\partial_xu\in L^1(0,T;\mathbb L^2)$. Moreover,
\[\int_0^T|b(t)|^4\,dt\leq 4\|u_1\|_{C([0,T];L^2)}^2\|\partial_xu_1\|_{L^2(0,T;L^2)}^2.\]
Thus the same $u$ solves the linear problem with interior force $f-u\partial_xu$ and central datum $g_0-(N/3)b^2\in L^2(0,T)$.
Approximate these effective linear data, together with $u^0$ and $g$, by smooth compatible linear data, keeping the right Dirichlet datum identically zero. Proposition~\ref{W.P._s=0} gives convergence of the corresponding classical linear solutions in $\mathbb B_{0,T}$ and of their first-derivative vertex traces in $L^2(0,T)$. The preceding spatial trace inequality also gives convergence of their vertex values in $L^4(0,T)$. The integrated linear multiplier identities therefore pass to the limit. Inserting the effective
force and central datum and integrating the nonlinear terms in
space gives exactly the integrated forms of
\eqref{eq:homogeneous-Dirichlet-energy-identity} and
\eqref{eq:homogeneous-Dirichlet-weighted-identity}.
This proves the estimates at the stated regularity.
\end{proof}

We now remove the nonhomogeneous Dirichlet datum. This must be done using a lifting compatible with its initial trace. Since
$\lambda\geq 1$, we have
\[
g_0\in L^2(0,T),
\qquad
p\in \left[H^{2/3}(0,T)\right]^N.
\]
For $j=1,\ldots,N$, set $q_j(x)=\frac{x}{L}p_j(0)$, $x\in[0,L]$, and let $a_j$ be the solution of
\begin{equation}\label{eq:scalar-Dirichlet-lifting}
\left\{
\begin{array}{ll}
\partial_ta_j+\partial_xa_j+\partial_x^3a_j=0,
& t\in(0,T),\ x\in(0,L),\\
a_j(t,0)=0,\qquad a_j(t,L)=p_j(t),
& t\in(0,T),\\
\partial_xa_j(t,L)=0,
& t\in(0,T),\\
a_j(0,x)=q_j(x),
& x\in(0,L).
\end{array}
\right.
\end{equation}
The data in \eqref{eq:scalar-Dirichlet-lifting} are $1$-compatible,
since $q_j(0)=0$ and $q_j(L)=p_j(0)$. Applying \eqref{eq:scalar-H1-lifting-estimate} from
Proposition~\ref{prop 2.3} with
\[
 (\phi,h_1,h_2,h_3,f)=(q_j,0,p_j,0,0),
\]
and using $\|q_j\|_{H^1(0,L)}\leq C_L|p_j(0)|$ and the continuous
trace map $H^{2/3}(0,T)\to\mathbb R$, we obtain
\begin{equation}\label{eq:linear-lifting-bound}
 \|a\|_{\mathbb B_{1,T}}
 +\sum_{j=1}^N\|\partial_x^2a_j(\cdot,0)\|_{L^2(0,T)}
 \leq C_{T,L,N}\|p\|_{[H^{2/3}(0,T)]^N}.
\end{equation}
Let $u$ be a solution of \eqref{kdV}, and set
\begin{equation*}
w_j=u_j-a_j,
\qquad
w_j^0=u_j^0-q_j,
\qquad j=1,\ldots,N.
\end{equation*}
We also define the modified datum at the central vertex by
\begin{equation*}
h_a(t)
:=
g_0(t)-\sum_{j=1}^N\partial_x^2a_j(t,0).
\end{equation*}
Since $a_j(t,0)=0$ for every $j=1,\ldots,N$, the function
$w=(w_1,\ldots,w_N)$ satisfies
\begin{align}\label{Neumann+a}
\left\{
\begin{array}{ll}
\partial_tw_j+\partial_xw_j+\partial_x^3w_j
+w_j\partial_xw_j
=-a_j\partial_xa_j-\partial_x(a_jw_j),
& t\in(0,T),\ x\in(0,L),\\
w_j(t,0)=w_1(t,0),
& t\in(0,T),\\
\displaystyle
\sum_{j=1}^N\partial_x^2w_j(t,0)
=-\alpha w_1(t,0)-\frac{N}{3}w_1^2(t,0)+h_a(t),
& t\in(0,T),\\
w_j(t,L)=0,\qquad \partial_xw_j(t,L)=g_j(t),
& t\in(0,T),\\
w_j(0,x)=w_j^0(x),
& x\in(0,L),
\end{array}
\right.
\end{align}
Conversely, if $w$ solves \eqref{Neumann+a}, then $u=w+a$ solves \eqref{kdV}. By \eqref{eq:linear-lifting-bound}, the modified central datum satisfies
\begin{equation*}
\|h_a\|_{L^2(0,T)}
\leq
\|g_0\|_{L^2(0,T)}
+\sum_{j=1}^N
\|\partial_x^2a_j(\cdot,0)\|_{L^2(0,T)}
\leq
C_T\left(
\|g_0\|_{L^2(0,T)}
+\|p\|_{\left[H^{2/3}(0,T)\right]^N}
\right)
\leq C_T\|d\|_{0,\lambda}.
\end{equation*}
Moreover, since $H^{2/3}(0,T)\hookrightarrow C([0,T])$
continuously (see, for instance, \cite{Bhattacharyya}), the definition
of $q_j$ gives
\begin{align*}
\|q_j\|_{L^2(0,L)}^2
&=
\frac{|p_j(0)|^2}{L^2}\int_0^Lx^2\,dx
=\frac{L}{3}|p_j(0)|^2.
\end{align*}
Consequently,
\begin{equation*}
\begin{split}
\|w^0\|_{\mathbb L^2(\mathcal T)}
&\leq
\|u^0\|_{\mathbb L^2(\mathcal T)}
+\|q\|_{\mathbb L^2(\mathcal T)}
\\
&\leq
\|u^0\|_{\mathbb L^2(\mathcal T)}
+\left(\frac{L}{3}\right)^{1/2}
\|p(0)\|_{\mathbb R^N}
\\
&\leq
C_{T,L}\left(
\|u^0\|_{\mathbb L^2(\mathcal T)}
+\|p\|_{\left[H^{2/3}(0,T)\right]^N}
\right)
\\
&\leq C_{T,L}\|d\|_{0,\lambda}.
\end{split}
\end{equation*}
We can now establish the energy estimate for the original nonlinear
problem.

\begin{proposition}\label{energy_s=0}
Let $T>0$ and $\lambda\geq 1$, and let $d=(u^0,g_0,g,p)\in\mathcal D_{0,T}^{\lambda}$. If $u\in\mathbb B_{0,T}$ is a mild solution of \eqref{kdV}, then there exists a positive nondecreasing function $F_T:[0,\infty)\longrightarrow(0,\infty)$
such that
\begin{equation}\label{eq:energy-s0-tame}
\|u\|_{\mathbb B_{0,T}}
\leq
F_T\bigl(\|d\|_{0,\lambda}\bigr)\|d\|_{0,\lambda}.
\end{equation}
\end{proposition}
\begin{proof}
Set $r:=\|d\|_{0,\lambda}$. Let $a$ be the lifting constructed above,
and put $w=u-a$ and
\[
 h_a:=g_0-\sum_{j=1}^N\partial_x^2a_j(\cdot,0),
 \qquad w^0:=u^0-q.
\]
Then $w$ solves \eqref{Neumann+a}. In particular, $w_j(t,L)=0$ and
all its vertex values agree. By \eqref{eq:linear-lifting-bound}
and the estimates preceding the proposition,
\begin{equation}\label{eq:energy-s0-lifting-data}
 A:=\|a\|_{\mathbb B_{1,T}}\leq C_Tr,\qquad
 D:=\|w^0\|_{\mathbb L^2}
       +\|h_a\|_{L^2(0,T)}+\|g\|_{[L^2(0,T)]^N}
 \leq C_Tr.
\end{equation}
The constants below may depend on the fixed parameters but not on
the solution. Set
\[
 F_a:=-a\partial_xa,
 \qquad K(t):=\sum_{j=1}^N
                  \|\partial_xa_j(t)\|_{L^\infty(0,L)}.
\]
The spatial Sobolev embeddings give
\begin{align}
 \|F_a\|_{L^2(0,T;\mathbb L^2)}
 &\leq C_T\|a\|_{C([0,T];\mathbb H^1)}^2\leq C_TA^2,
 \label{eq:energy-s0-lifting-force}
 \end{align}
 and
 \begin{align}
 \|K\|_{L^1(0,T)}
 +\int_0^T\sum_{j=1}^N\|a_j(t)\|_{L^\infty(0,L)}\,dt
 &\leq C_TA.
 \label{eq:energy-s0-lifting-coefficient}
\end{align}
Write
\[
 b(t):=w_1(t,0),\qquad d_j(t):=\partial_xw_j(t,0),\qquad
 E(t):=\|w(t)\|_{\mathbb L^2}^2,\qquad
 \kappa:=\alpha-\frac N2>0.
\]
Multiplication by $w_j$ and summation over the edges yield
\begin{equation}\label{eq:energy-s0-w-identity}
\begin{split}
 \frac12E'(t)+\kappa |b(t)|^2
 +\frac12\sum_{j=1}^N|d_j(t)|^2
 ={}&\frac12\|g(t)\|_{\mathbb R^N}^2+h_a(t)b(t)
 +(F_a(t),w(t))_{\mathbb L^2}\\
 &-\frac12\sum_{j=1}^N
           \int_0^L\partial_xa_j(t,x)|w_j(t,x)|^2\,dx.
\end{split}
\end{equation}
The cubic vertex terms cancel exactly as in Proposition~\ref{prop 3.1}.
For the mixed term, the conditions $a_j(t,0)=0$ and $w_j(t,L)=0$
give
\[
 -\int_0^L w_j\partial_x(a_jw_j)\,dx
 =-\frac12\int_0^L\partial_xa_j\,w_j^2\,dx.
\]
Young's inequality applied to $h_ab$ and $(F_a,w)$ therefore gives
\begin{align*}
 E'(t)+\kappa |b(t)|^2+\sum_{j=1}^N|d_j(t)|^2
 \leq{}&\|g(t)\|_{\mathbb R^N}^2
       +\kappa^{-1}|h_a(t)|^2+\|F_a(t)\|_{\mathbb L^2}^2+(1+K(t))E(t).
\end{align*}
Gronwall's inequality, including the nonnegative dissipative terms,
implies
\begin{align}\label{eq:energy-s0-w-low-traces}
 M+\|b\|_{L^2(0,T)}
 +\left(\sum_{j=1}^N\|d_j\|_{L^2(0,T)}^2\right)^{1/2}
 &\leq C_Te^{C_TA}(D+A^2)\leq F_T(r)r,
\end{align}
where $M:=\|w\|_{C([0,T];\mathbb L^2)}$ and
\eqref{eq:energy-s0-lifting-data}--\eqref{eq:energy-s0-lifting-coefficient}
were used in the last step. It remains to estimate $\partial_xw$. Set
\[J(t):=\sum_{j=1}^N\int_0^L x|w_j(t,x)|^2\,dx.\]
Multiplication of the equation by $xw_j$ gives
\begin{equation}\label{eq:energy-s0-w-weighted}
\begin{split}
 \frac12J'(t)+\frac32\|\partial_xw(t)\|_{\mathbb L^2}^2
 ={}&\frac12E(t)+\frac L2\|g(t)\|_{\mathbb R^N}^2
-b(t)\sum_{j=1}^Nd_j(t)+   +\sum_{j=1}^N\int_0^L xF_{a,j}(t,x)w_j(t,x)\,dx\\
      &\frac13\sum_{j=1}^N\int_0^L w_j(t,x)^3\,dx+\frac12\sum_{j=1}^N\int_0^L
       \bigl(a_j(t,x)-x\partial_xa_j(t,x)\bigr)w_j(t,x)^2\,dx.
\end{split}
\end{equation}
As in Proposition~\ref{prop 3.1}, the homogeneous right Dirichlet
condition implies
\[
 \sum_{j=1}^N\int_0^L|w_j(t,x)|^3\,dx
 \leq\varepsilon\|\partial_xw(t)\|_{\mathbb L^2}^2
       +C_\varepsilon M^{10/3}.
\]
Integrate \eqref{eq:energy-s0-w-weighted} on $(0,T)$ and use
$J(T)\geq0$, $J(0)\leq L\|w^0\|_{\mathbb L^2}^2$.
The vertex term is bounded by
\[
 \sqrt N\|b\|_{L^2(0,T)}
 \left(\sum_{j=1}^N\|d_j\|_{L^2(0,T)}^2\right)^{1/2},
\]
the forcing term by $LM\|F_a\|_{L^1(0,T;\mathbb L^2)}$, and the
last term by $C_TAM^2$. After choosing $\varepsilon$ small and
absorbing the derivative term, we obtain
\begin{align*}
 \|\partial_xw\|_{L^2(0,T;\mathbb L^2)}^2
 \leq &C_T\Bigl(D^2+(1+A)M^2+M\|F_a\|_{L^1(0,T;\mathbb L^2)}
                  +M^{10/3}+\|b\|_{L^2(0,T)}^2
       +\sum_{j=1}^N\|d_j\|_{L^2(0,T)}^2\Bigr)\\
        \leq &F_T(r)r^2.
\end{align*}
In the last inequality we used \eqref{eq:energy-s0-w-low-traces},
$A+D\leq C_Tr$, and $\|F_a\|_{L^1}\leq C_TA^2$; polynomial
factors in $r$ have been included in the positive nondecreasing
function $F_T$. Taking square roots and combining with
\eqref{eq:energy-s0-w-low-traces} gives
\[\|w\|_{\mathbb B_{0,T}}\leq F_T(r)r.\]
Finally, $u=a+w$ and $\|a\|_{\mathbb B_{0,T}}\leq C_TA$ prove \eqref{eq:energy-s0-tame}. The integrated multiplier identities are valid for the stated mild solutions by the linear approximation argument used in Proposition~\ref{prop 3.1}. More explicitly, $w$ solves the linear problem with effective force $F_a-\partial_x(aw)-w\partial_xw\in L^1(0,T;\mathbb L^2)$
and central datum $h_a-(N/3)b^2\in L^2(0,T)$. Approximation of these effective data by smooth compatible linear data gives convergence in $\mathbb B_{0,T}$, convergence of the vertex values in $L^4(0,T)$, and convergence of the first-derivative
vertex traces in $L^2(0,T)$. Passing to the limit in the linear multiplier identities and then integrating the product terms in space gives \eqref{eq:energy-s0-w-identity} and \eqref{eq:energy-s0-w-weighted} in integrated form. All estimates also hold on $[0,t]$ for $0<t\leq T$, as long as the solution exists there, with the same right-hand side $F_T(r)r$.
Indeed, keep the lifting constructed on the fixed interval $[0,T]$ and restrict the multiplier identities to $[0,t]$; every coefficient and data norm is then bounded by the corresponding norm on $[0,T]$. 
\end{proof}

We are now in a position to prove global well-posedness at the energy level. Notice first that no compatibility condition is required when $s=0$. Hence, $\mathbb X_{0,T}^{\lambda}=\mathcal D_{0,T}^{\lambda}$.

\begin{theorem}\label{global s=0}
Let $T>0$, $\lambda\geq1$, and $d=(u^0,g_0,g,p)\in\mathcal D_{0,T}^{\lambda}$. 
Then problem \eqref{kdV} has a unique solution
$u\in\mathcal X_{0,T}$. Moreover, there exists a positive
nondecreasing function $F_T:[0,\infty)\longrightarrow(0,\infty)$
such that
\begin{equation}\label{eq:global-s0-tame}
\|u\|_{\mathcal X_{0,T}}
\leq
F_T\bigl(\|d\|_{0,\lambda}\bigr)\|d\|_{0,\lambda}.
\end{equation}
\end{theorem}
\begin{proof}
Set $R:=\|d\|_{0,\lambda}$. By Proposition~\ref{energy_s=0}, there
is a positive nondecreasing function $E_T$, which we may take to
satisfy $E_T\geq1$, such that
\begin{equation}\label{eq:uniform-energy-up-to-t}
 \|u\|_{\mathbb B_{0,t}}\leq E_T(R)R,
 \qquad 0<t\leq T,
\end{equation}
on every interval on which the solution exists. Here the lifting
is kept on the fixed interval $[0,T]$, as explained at the end of
the proof of that proposition. In particular,
\begin{equation}\label{eq:uniform-L2-restart}
 \sup_{0\leq s\leq t}\|u(s)\|_{\mathbb L^2}
 \leq E_T(R)R.
\end{equation}
We first make the uniformity of the restart time precise. Since
$\lambda\geq1$, the prescribed boundary data satisfy
\[g_0\in L^2(0,T),\qquad g\in[L^2(0,T)]^N,\qquad p\in[H^{1/3}(0,T)]^N,\]
with the sum of these norms bounded by $C_TR$.
Let $G_0,G,P$ denote their zero extensions to $\mathbb R$.
These extensions belong to $L^2(\mathbb R)$,
$[L^2(\mathbb R)]^N$, and $[H^{1/3}(\mathbb R)]^N$,
respectively, with norms bounded by $C_TR$.
For $0\leq t_0\leq T$, define on the same fixed interval $(0,T)$
\[
 g_0^{t_0}(s):=G_0(t_0+s),\qquad
 g^{t_0}(s):=G(t_0+s),\qquad
 p^{t_0}(s):=P(t_0+s).
\]
Restriction and translation invariance give uniform bounds for
these data. Using also $L^2(0,T)\hookrightarrow H^{-1/3}(0,T)$,
we obtain
\begin{equation}\label{eq:translated-data-uniform-bound}
 \|g_0^{t_0}\|_{H^{-1/3}(0,T)}
 +\|g^{t_0}\|_{[L^2(0,T)]^N}
 +\|p^{t_0}\|_{[H^{1/3}(0,T)]^N}
 \leq C_{\mathrm{ext},T}R,
\end{equation}
uniformly in $t_0$. We use these extensions only at the local energy regularity. 
On $(0,T-t_0)$ they coincide with the original translated boundary data. Put
\[B(R):=\bigl(E_T(R)+C_{\mathrm{ext},T}\bigr)R.\]
The energy-level data size of each restarted problem is at most $B(R)$ by \eqref{eq:uniform-L2-restart} and
\eqref{eq:translated-data-uniform-bound}; the same holds for the first problem with initial datum $u^0$. The choice \eqref{eq:L2-time-choice} in the local proof therefore provides a common time $\delta(R)>0$, which may be chosen
nonincreasing in $R$ and with $\delta(R)\leq\min\{1,T/2\}$. For example, let $C_T$ and $C_{\mathrm{lin},T}$ be the constants
in that proof, enlarging $C_T$ if necessary to cover the local Lipschitz estimate as well, and take
\[\delta(R):=\min\left\{1,\frac T2,\bigl[6C_T(1+C_{\mathrm{lin},T}B(R))\bigr]^{-6}\right\}.\]
Indeed, $\omega(\delta)\leq3\delta^{1/6}$ for $0<\delta\leq1$. Moreover, \eqref{eq:local-L2-data-bound} gives, on this common
interval and its restrictions,
\begin{equation*}
 \|u\|_{\mathcal X_{0,I}}
 \leq 2C_{\mathrm{lin},T}B(R),
\end{equation*}
after translating the time variable, whenever $I$ is one of the
local intervals of length at most $\delta(R)$ constructed below. Set
\[m:=\max\left\{1,\left\lceil\frac{T}{\delta(R)}\right\rceil\right\},\qquad t_k:=\frac{kT}{m},\quad k=0,\ldots,m,\qquad I_k:=(t_{k-1},t_k).\]
Then $|I_k|=T/m\leq\delta(R)$. Apply Theorem~\ref{main} with $s=0$ at $t_0=0$ and restrict the resulting solution to $[0,t_1]$. Its terminal value belongs to $\mathbb L^2$ and satisfies \eqref{eq:uniform-L2-restart}. Apply the same local theorem at $t_1$, using this terminal value and the boundary data in \eqref{eq:translated-data-uniform-bound}. There is no pointwise compatibility condition to impose at this regularity.\\

\noindent Proceeding by induction constructs a solution on every $[0,t_k]$. At each joining time the pieces have the same $\mathbb L^2$ value, so their concatenation is continuous into $\mathbb L^2$. Concatenating the mild formulations gives a mild solution with the original data on $[0,t_k]$. The finite gluing lemma supplies the time-trace regularity at this existence step, and the $L_t^2\mathbb H_x^1$ component glues directly. 
The a priori estimate \eqref{eq:uniform-energy-up-to-t} applies to the concatenation, providing the bound required for the next restart. After $m$ steps, the construction includes the terminal time $T$. This proves existence on the whole closed interval $[0,T]$. To estimate the full $\mathcal X_{0,T}$ norm, we use the linear
estimate on the fixed interval $(0,T)$. Set
\[b(t):=u_1(t,0),\qquad f_u:=-u\partial_xu,\qquad h_u:=g_0-\frac N3 b^2.\]
Lemma~\ref{lemma 2.4} gives
\begin{equation}\label{eq:global-s0-interior-forcing}
\|f_u\|_{L^1(0,T;\mathbb L^2)}\leq C_T\|u\|_{\mathbb B_{0,T}}^2.
\end{equation}
For the central datum, the spatial trace inequality
\[|z(0)|^2\leq C_L\bigl(\|z\|_{L^2(0,L)}^2+\|z\|_{L^2(0,L)}\|z'\|_{L^2(0,L)}\bigr),\qquad z\in H^1(0,L),\]
does not require a homogeneous right boundary value. Squaring and integrating in time gives
\[\|b\|_{L^4(0,T)}^4\leq C_L\left(T\|u_1\|_{C_tL_x^2}^4+\|u_1\|_{C_tL_x^2}^2\|\partial_xu_1\|_{L_t^2L_x^2}^2\right)\leq C_T\|u\|_{\mathbb B_{0,T}}^4.\]
Consequently,
\begin{equation}\label{eq:global-s0-central-forcing}
\|b^2\|_{H^{-1/3}(0,T)}\leq C_T\|b^2\|_{L^2(0,T)}\leq C_T\|u\|_{\mathbb B_{0,T}}^2.
\end{equation}
The nonlinear solution $u$ is therefore the solution of \eqref{LkdV} with data $(u^0,h_u,g,p,f_u)$. Proposition~\ref{W.P._s=0}, including all its hidden trace estimates, yields
\begin{align*}
 \|u\|_{\mathcal X_{0,T}}
 &\leq C_T\left(
 \|u^0\|_{\mathbb L^2}
 +\|g_0\|_{H^{-1/3}(0,T)}
 +\|g\|_{[L^2(0,T)]^N}
 +\|p\|_{[H^{1/3}(0,T)]^N}
 +\|u\|_{\mathbb B_{0,T}}^2\right)\\
 &\leq C_T\bigl(R+E_T(R)^2R^2\bigr).
\end{align*}
Here the additional $H^0_t\mathbb H_x^1$ component of
$\mathbb B^*_{0,T}$ is already part of $\mathbb B_{0,T}$.
Taking
\[
 F_T(R):=C_T\bigl(1+E_T(R)^2R\bigr)
\]
proves \eqref{eq:global-s0-tame}. This function is positive and
nondecreasing, and the estimate also includes $R=0$.

Uniqueness follows by successive application of local uniqueness
on the same partition. For completeness, continuous dependence
also follows quantitatively. On a bounded set
$\|d\|_{0,\lambda},\|\widetilde d\|_{0,\lambda}\leq R_0$, use the
common partition chosen for $R_0$. The zero-extension and
translation operators above are linear and bounded from the
global data space into the local data spaces. Thus
\eqref{eq:local-L2-data-lipschitz} bounds the difference on each
interval by a constant depending on $R_0$ times the sum of the
initial difference on that interval and
$\|d-\widetilde d\|_{0,\lambda}$. Induction and elementary gluing
of the $C_t\mathbb L_x^2$ and $L_t^2\mathbb H_x^1$ components give
the corresponding Lipschitz bound in $\mathbb B_{0,T}$.
The bilinear versions of \eqref{eq:global-s0-interior-forcing}
and \eqref{eq:global-s0-central-forcing}, followed by the same
linear estimate applied to the difference, then give
\[
 \|u-\widetilde u\|_{\mathcal X_{0,T}}
 \leq C_{T,R_0}\|d-\widetilde d\|_{0,\lambda}.
\]
In particular, the solution map is continuous in the stated
global data topology.
\end{proof}

\subsection{The case \texorpdfstring{$s=3$}{s=3}: Global a priori estimate} We begin with an auxiliary result obtained by combining a fixed-point
argument with Proposition~\ref{W.P._s=0}.

\begin{lemma}\label{lemma 5.1}
Let $T>0$ and let $v^0\in\mathbb L^2(\mathcal T)$, $\widetilde g_0\in H^{-1/3}(0,T)$, $\widetilde g\in[L^2(0,T)]^N$, $\widetilde p\in[H^{1/3}(0,T)]^N$,$f\in L^1(0,T;\mathbb L^2(\mathcal T))$, and $a\in\mathcal X_{0,T}$. Set $Q_T=(0,T)\times(0,L)$. Then the problem
\begin{equation*}
\begin{cases}
\partial_t v_j + \partial_x\big[(1+a_j)v_j\big] + \partial_x^3 v_j = f_j, & (t,x) \in Q_T, \ j=1,\ldots,N, \\
v_j(t,0) = v_1(t,0), & t \in (0,T), \ j=1,\ldots,N, \\
\displaystyle\sum_{j=1}^N \partial_x^2 v_j(t,0) = -\alpha v_1(t,0) - \frac{2N}{3} a_1(t,0) v_1(t,0) + \widetilde{g}_0(t), & t \in (0,T), \\
v_j(t,L) = \widetilde{p}_j(t), \qquad \partial_x v_j(t,L) = \widetilde{g}_j(t), & t \in (0,T), \ j=1,\ldots,N, \\
v_j(0,x) = v_j^0(x), & x \in (0,L), \ j=1,\ldots,N,
\end{cases}
\end{equation*}
has a unique solution $v\in\mathcal X_{0,T}$. Moreover, there exists a
nondecreasing function
\[
\mathfrak C_T:[0,\infty)\longrightarrow(0,\infty),
\]
depending only on $T$, $L$, $N$, and $\alpha$, such that
\begin{equation}\label{eq:lemma-5.1-estimate}
\begin{split}
\|v\|_{\mathcal X_{0,T}}
\leq \mathfrak C_T\big(\|a\|_{\mathcal X_{0,T}}\big)
\Big(&\|v^0\|_{\mathbb L^2(\mathcal T)}
+\|\widetilde g_0\|_{H^{-1/3}(0,T)}
+\|\widetilde g\|_{[L^2(0,T)]^N}\\
&+\|\widetilde p\|_{[H^{1/3}(0,T)]^N}
+\|f\|_{L^1(0,T;\mathbb L^2(\mathcal T))}\Big).
\end{split}
\end{equation}
\end{lemma}
\begin{proof}
Let $I=(t_0,t_0+\rho)\subset(0,T)$, and prescribe an initial datum
$z\in\mathbb L^2(\mathcal T)$ at $t=t_0$.
For $w\in\mathcal X_{0,I}$, let $\Gamma_{a,I}w$ be the solution
of \eqref{LkdV} on $I$ with initial datum $z$, external data
$(\widetilde g,\widetilde p)|_I$, interior force
\[f|_I-\partial_x(aw),\]
and central datum
\[\widetilde g_0|_I-\frac{2N}{3}a_1(\cdot,0)w_1(\cdot,0).\]
For fixed $a$, this is an affine map in $w$. Set
\[D_I(z):=\|z\|_{\mathbb L^2}+\|\widetilde g_0\|_{H^{-1/3}(I)}+\|\widetilde g\|_{[L^2(I)]^N}+\|\widetilde p\|_{[H^{1/3}(I)]^N}+\|f\|_{L^1(I;\mathbb L^2)}.\]
The linear estimate in Proposition~\ref{W.P._s=0} is uniform on these subintervals. Indeed, translate $I$ to $(0,\rho)$,
extend the interior force by zero and the boundary data using their restriction norms to a fixed interval of length $T$,
apply the proposition there, and restrict the solution. Linear uniqueness ensures that the restriction is independent
of the extensions. 
By item i of Lemma~\ref{lemma 2.4} and
$\partial_x(aw)=a\partial_xw+w\partial_xa$,
\begin{equation*}
 \|\partial_x(aw)\|_{L^1(I;\mathbb L^2)}\leq C(\rho^{1/4}+\rho^{1/2})\|a\|_{\mathbb B_{0,I}}\|w\|_{\mathbb B_{0,I}}.
\end{equation*}
The boundary trace estimate \eqref{eq:L2-boundary-small-time}, after translation in time, gives
\begin{equation*}
\|a_1(\cdot,0)w_1(\cdot,0)\|_{H^{-1/3}(I)}\leq C\rho^{1/6}\|a\|_{\mathcal X_{0,I}}\|w\|_{\mathcal X_{0,I}}.
\end{equation*}
Here the vertex values are the spatial traces supplied by
$L_t^2\mathbb H_x^1$; their $H^{1/3}$ regularity and bounds
follow as in the proof of \eqref{eq:L2-boundary-small-time}. Absorb the fixed factor $2N/3$ into a constant $C_0\geq1$, and put
\[\omega(\rho):=\rho^{1/4}+\rho^{1/2}+\rho^{1/6}.\]
Then $\Gamma_{a,I}$ is well defined on $\mathcal X_{0,I}$ and
\begin{equation*}
\|\Gamma_{a,I}w\|_{\mathcal X_{0,I}} \leq C_0D_I(z)+C_0\omega(\rho) \|a\|_{\mathcal X_{0,I}}\|w\|_{\mathcal X_{0,I}}.
\end{equation*}
For two iterands, all the prescribed data cancel. Consequently,
\begin{equation*}
 \|\Gamma_{a,I}w-\Gamma_{a,I}\widetilde w\|_{\mathcal X_{0,I}}
 \leq C_0\omega(\rho)\|a\|_{\mathcal X_{0,I}}\|w-\widetilde w\|_{\mathcal X_{0,I}}.
\end{equation*}
The constant $C_0$ depends only on $T,L,N$, and $\alpha$. Set $A:=\|a\|_{\mathcal X_{0,T}}$. With the restriction norm convention, $\|a\|_{\mathcal X_{0,I}}\leq A$. Choose
\[\delta(A):=\min\left\{T,1,[6C_0(1+A)]^{-6}\right\}.\]
This choice is positive and nonincreasing in $A$ and, since
$\omega(\rho)\leq3\rho^{1/6}$ for $0<\rho\leq1$, gives
\[C_0\omega(\rho)A\leq\frac12\qquad\text{whenever }0<\rho\leq\delta(A).\]
The map $\Gamma_{a,I}$ is therefore a contraction on the complete space $\mathcal X_{0,I}$. Its unique fixed point satisfies
\[\|v\|_{\mathcal X_{0,I}}\leq C_0D_I(z)+\frac12\|v\|_{\mathcal X_{0,I}}.\]
Subtracting the last term yields
\begin{equation}\label{eq:fixed-point-local-estimate}
 \|v\|_{\mathcal X_{0,I}}\leq2C_0D_I(z).
\end{equation}
Set
\[m=m(A):=\max\left\{1,\left\lceil\frac{T}{\delta(A)}\right\rceil\right\},\qquad t_k:=\frac{kT}{m},\quad k=0,\ldots,m,\]
and $I_k=(t_{k-1},t_k)$. All intervals have length at most $\delta(A)$. Starting with $z=v^0$ and then using $z=v(t_{k-1})$ on $I_k$, we construct a solution on all of $[0,T]$. The values at joining times agree in $\mathbb L^2$, so the concatenation is continuous into $\mathbb L^2$ and its mild
formulations concatenate. The finite gluing property for the trace orders $1/3$, $0$, and $-1/3$ gives membership in $\mathcal X_{0,T}$, by Lemma~\ref{lemma:finite-Sobolev-gluing}. This application is made for almost every spatial point, followed by the essential supremum in space. Let
\[D_T:=\|v^0\|_{\mathbb L^2}+\|\widetilde g_0\|_{H^{-1/3}(0,T)}+\|\widetilde g\|_{[L^2(0,T)]^N}+\|\widetilde p\|_{[H^{1/3}(0,T)]^N}+\|f\|_{L^1(0,T;\mathbb L^2)}.\]
Restriction of the prescribed data and
\eqref{eq:fixed-point-local-estimate} give, by induction,
\[\|v\|_{\mathcal X_{0,I_k}}\leq c_kD_T,\qquad c_1=2C_0,\qquad c_k=2C_0(1+c_{k-1}).\]
Since $C_0\geq1$, the sequence $(c_k)$ is increasing. The
$C_t\mathbb L_x^2$ and $L_t^2\mathbb H_x^1$ components therefore
satisfy the quantitative estimate
\begin{equation*}
 \|v\|_{\mathbb B_{0,T}}
 \leq(1+\sqrt m)c_mD_T.
\end{equation*}
To bound the full norm 
, apply the linear estimate once more on $(0,T)$. The product estimate and the spatial trace bound
\[\|b_1(\cdot,0)\|_{L^4(0,T)}\leq C_T\|b\|_{\mathbb B_{0,T}},\qquad b\in\mathbb B_{0,T},\]
proved in the energy-level theorem, imply
\[\|\partial_x(av)\|_{L^1(0,T;\mathbb L^2)}+\|a_1(\cdot,0)v_1(\cdot,0)\|_{H^{-1/3}(0,T)}\leq C_T\|a\|_{\mathbb B_{0,T}}\|v\|_{\mathbb B_{0,T}}.\]
For the boundary product, use H\"older in $L^4\times L^4$ and $L^2(0,T)\hookrightarrow H^{-1/3}(0,T)$. Proposition~\ref{W.P._s=0}, applied with effective force $f-\partial_x(av)$ and central datum
$\widetilde g_0-(2N/3)a_1v_1$, now gives
\[\|v\|_{\mathcal X_{0,T}}\leq C_T\left(D_T+A\|v\|_{\mathbb B_{0,T}}\right)\leq C_T\left[1+A(1+\sqrt{m(A)})c_{m(A)}\right]D_T.\]
Thus we may take
\[\mathfrak C_T(A):=C_T\left[1+A(1+\sqrt{m(A)})c_{m(A)}\right].\]
This function is positive and nondecreasing, because $m(A)$ and $c_m$ are nondecreasing. It proves \eqref{eq:lemma-5.1-estimate}. Finally, every solution in $\mathcal X_{0,T}$ restricts on $I_k$
to a fixed point of the same affine map with the appropriate initial value. Uniqueness on each interval, applied successively, proves uniqueness in the whole class. For fixed $a$, linearity of the problem and uniqueness also show that the data-to-solution map is bounded and linear.
\end{proof}

\begin{proposition}\label{energy_s=3}
Let $T>0$ and $\lambda\geq1$, and let $d=(u^0,g_0,g,p)\in\mathbb X_{3,T}^{\lambda}$  be $3$-compatible data. Suppose that $u\in\mathcal X_{3,T}$ is a
solution of \eqref{kdV} corresponding to $d$. Set $v=\partial_tu$.
Then $v\in\mathcal X_{0,T}$,
\begin{equation}\label{nlu_j}
u_j(t,\cdot)=u_j^0+\int_0^t v_j(s,\cdot)\,ds,
\qquad j=1,\ldots,N,
\end{equation}
and $v$ solves
\begin{equation}\label{nlv_j}
\begin{cases}
\partial_t v_j + \partial_x\big[(1+u_j)v_j\big] + \partial_x^3 v_j = 0, & t \in (0,T), \ x \in (0,L), \\
v_j(t,0) = v_1(t,0), & t \in (0,T), \\
\displaystyle\sum_{j=1}^N \partial_x^2 v_j(t,0) = -\alpha v_1(t,0) - \frac{2N}{3} u_1(t,0) v_1(t,0) + g_0'(t), & t \in (0,T), \\
v_j(t,L) = p_j'(t), \qquad \partial_x v_j(t,L) = g_j'(t), & t \in (0,T), \\
v_j(0,x) = -\partial_x u_j^0(x) - \partial_x^3 u_j^0(x) - u_j^0(x) \partial_x u_j^0(x), & x \in (0,L),
\end{cases}
\end{equation}
for $j=1,\ldots,N$. Moreover, there exists a positive nondecreasing
function $F_T:[0,\infty)\to(0,\infty)$, depending only on the fixed
parameters $T,L,N,\alpha$, and $\lambda$, such that
\begin{equation}\label{rrrrrr}
\|u\|_{\mathbb B_{3,T}}+\|v\|_{\mathcal X_{0,T}}
\leq
F_T\big(\|d\|_{0,\lambda}\big)\|d\|_{3,\lambda}.
\end{equation}
In particular, $F_T(0)>0$.
\end{proposition}
\begin{proof}
Write $d=(u^0,g_0,g,p)$, $r=\|d\|_{0,\lambda}$, and $R=\|d\|_{3,\lambda}$. By the continuous Sobolev embeddings between the corresponding data spaces, $r\leq C R$. In the estimates below, $F_T$ denotes a positive nondecreasing function which may change from one occurrence to the next and depends only on the fixed parameters. Since $u\in\mathcal X_{3,T}$, we have
\[u\in C([0,T];\mathbb H^3)\cap L^2(0,T;\mathbb H^4)\cap H^1(0,T;\mathbb H^1).\]
The equation gives
\[v:=\partial_tu=-\partial_xu-\partial_x^3u-u\partial_xu\in C([0,T];\mathbb L^2),\]
and the $H^1_t\mathbb H_x^1$ component gives
$v\in L^2(0,T;\mathbb H^1)$. Differentiation in time in the hidden
trace spaces gives, for $k=0,1,2$,
\[\partial_x^kv_j=\partial_t(\partial_x^ku_j)\in L_x^\infty(0,L;H^{(1-k)/3}(0,T)).\]
Consequently $v\in\mathcal X_{0,T}$
. We also justify differentiation at the endpoints. The hidden trace bounds for $u$, together with $u\in C_t\mathbb H_x^3$ and weak compactness in the time Sobolev spaces, give
\[\partial_x^ku_j(\cdot,\xi)\in H^{(4-k)/3}(0,T),\qquad \xi\in\{0,L\},\quad k=0,1,2.\]
Indeed, as interior spatial points tend to an endpoint, their
traces converge in $C([0,T])$, while their time Sobolev norms
remain bounded.
For $\zeta\in C_c^\infty(0,T)$, the identity
\[\int_0^T v_j(t,\cdot)\zeta(t)\,dt=-\int_0^T u_j(t,\cdot)\zeta'(t)\,dt\]
has a right-hand side in $H^3(0,L)$. Taking its spatial traces
therefore shows that
\[\left\langle\partial_x^kv_j(\cdot,\xi),\zeta\right\rangle=-\int_0^T\partial_x^ku_j(t,\xi)\zeta'(t)\,dt.\]
Thus, time differentiation commutes with the spatial boundary traces in the sense of distributions, including the second derivative trace. In particular, if $b=u_1(\cdot,0)$, then $b\in H^{4/3}(0,T)$
and $b'=v_1(\cdot,0)$. The ordinary Sobolev chain rule gives
\[\partial_t(b^2)=2bb'=2u_1(\cdot,0)v_1(\cdot,0).\]
Here $b\in H^1(0,T)\cap L^\infty(0,T)$, so the product on the right belongs to $L^2(0,T)$ and hence to $H^{-1/3}(0,T)$. Differentiating the equation in distributions and applying these
identities to the boundary conditions proves \eqref{nlv_j}. The continuous $\mathbb L^2$ representative of $v$ satisfies
\[v_j(0)=-\partial_xu_j^0-\partial_x^3u_j^0-u_j^0\partial_xu_j^0.\]
Finally, the fundamental theorem of calculus in $\mathbb L^2$ gives \eqref{nlu_j}. We first estimate the initial datum of $v$. In one space dimension, the embedding $H^1(0,L)\hookrightarrow L^\infty(0,L)$ and interpolation between $L^2(0,L)$ and $H^2(0,L)$ give
\begin{align*}
\|u_j^0\partial_xu_j^0\|_{L^2(0,L)}
&\leq C\|u_j^0\|_{H^1(0,L)}^2\leq C\|u_j^0\|_{L^2(0,L)}
\|u_j^0\|_{H^2(0,L)}\leq C\|u_j^0\|_{L^2(0,L)}
\|u_j^0\|_{H^3(0,L)}.
\end{align*}
After summing over the edges, we obtain
\begin{equation}\label{eq:v-initial-tame}
\|v(0)\|_{\mathbb L^2(\mathcal T)}
\leq
C\bigl(1+\|u^0\|_{\mathbb L^2(\mathcal T)}\bigr)
\|u^0\|_{\mathbb H^3(\mathcal T)}
\leq C(1+r)R.
\end{equation}
Furthermore, since $\lambda\geq1$, the boundedness of the time
derivative between the corresponding Sobolev spaces gives
\begin{align}
\|g_0'\|_{H^{-1/3}(0,T)}
+\|g'\|_{[L^2(0,T)]^N}
+\|p'\|_{[H^{1/3}(0,T)]^N}
\leq C R.
\label{eq:boundary-derivatives-tame}
\end{align}
We may therefore apply Lemma~\ref{lemma 5.1}, with $a=u$, $f=0$,
$\widetilde g_0=g_0'$, $\widetilde g=g'$, and
$\widetilde p=p'$. Using \eqref{eq:v-initial-tame} and
\eqref{eq:boundary-derivatives-tame}, we find
\begin{equation}\label{est._v}
\|v\|_{\mathcal X_{0,T}}
\leq
C(1+r)\mathfrak C_T
\bigl(\|u\|_{\mathcal X_{0,T}}\bigr)R.
\end{equation}
On the other hand, the energy-level estimate of
Theorem~\ref{global s=0} yields
\begin{equation}\label{eq:u-low-bound}
\|u\|_{\mathcal X_{0,T}}
\leq F_T(r)r.
\end{equation}
Since $\mathfrak C_T$ is nondecreasing, it follows from
\eqref{est._v} that
\begin{align*}
\|v\|_{\mathcal X_{0,T}}
\leq
C(1+r)\mathfrak C_T\bigl(F_T(r)r\bigr)R\leq F_T(r)R.
\end{align*}
Therefore,
\begin{equation}\label{eq:v-final-tame}
\|v\|_{\mathcal X_{0,T}}
\leq F_T(r)R.
\end{equation}
We next estimate the nonlinear term. Put $G=u\partial_xu$.
The product rule in distributions gives
\[\partial_tG=v\partial_xu+u\partial_xv.\]
Item i of Lemma~\ref{lemma 2.4}, applied separately to these
products, yields
\begin{align*}
 \|G\|_{L^1(0,T;\mathbb L^2)}
 &\leq C_T\|u\|_{\mathbb B_{0,T}}^2,\\
 \|\partial_tG\|_{L^1(0,T;\mathbb L^2)}
 &\leq C_T\|u\|_{\mathbb B_{0,T}}
               \|v\|_{\mathbb B_{0,T}}.
\end{align*}
Hence $G\in W^{1,1}(0,T;\mathbb L^2)$. Its initial value is
$G(0)=u^0\partial_xu^0$, since $u\in C_t\mathbb H_x^3$.
The initial product estimate used above gives
\[
 \|G(0)\|_{\mathbb L^2}
 \leq C\|u^0\|_{\mathbb L^2}\|u^0\|_{\mathbb H^3}
 \leq CrR.
\]
Using the integral representation of $G$ in $\mathbb L^2$,
\eqref{eq:u-low-bound}, and \eqref{eq:v-final-tame}, we obtain
\begin{equation}\label{uu_x}
 \begin{split}
 \|u\partial_xu\|_{C([0,T];\mathbb L^2)}
 \leq\|G(0)\|_{\mathbb L^2}
           +\|\partial_tG\|_{L^1(0,T;\mathbb L^2)}\leq CrR+C_T\|u\|_{\mathbb B_{0,T}}
                      \|v\|_{\mathbb B_{0,T}}
 \leq F_T(r)R.
 \end{split}
\end{equation}
For every $\phi\in H^3(0,L)$, the graph-norm estimate for $P=\partial_x+\partial_x^3$ reads
\begin{equation*}
\|\phi\|_{H^3(0,L)}
\leq
C\left(
\|\phi\|_{L^2(0,L)}
+\|P\phi\|_{L^2(0,L)}
\right).
\end{equation*}
 Since the KdV equation gives $Pu=-v-u\partial_xu$, we deduce from \eqref{eq:u-low-bound}, \eqref{eq:v-final-tame}, and
\eqref{uu_x} that
\begin{equation}\label{5.10}
\|u\|_{C([0,T];\mathbb H^3(\mathcal T))}
\leq F_T(r)R.
\end{equation}
It remains to control the fourth spatial derivative. Differentiating
the equation with respect to $x$ gives
\begin{equation}\label{5.11}
\partial_x^4u
=-\partial_xv-\partial_x^2u
-\partial_x(u\partial_xu).
\end{equation}
For each $t\in[0,T]$, the one-dimensional Sobolev embedding yields
\begin{align*}
 \|\partial_x(u\partial_xu)(t)\|_{\mathbb L^2}
 &\leq\|(\partial_xu(t))^2\|_{\mathbb L^2}
       +\|u(t)\partial_x^2u(t)\|_{\mathbb L^2}\leq C\|u(t)\|_{\mathbb H^3}\|u(t)\|_{\mathbb H^1}.
\end{align*}
Consequently, by \eqref{5.10} and \eqref{eq:u-low-bound},
\begin{align}
\|\partial_x(u\partial_xu)\|_
 {L^2(0,T;\mathbb L^2(\mathcal T))}
&\leq
C\|u\|_{C([0,T];\mathbb H^3(\mathcal T))}
 \|u\|_{L^2(0,T;\mathbb H^1(\mathcal T))}
\leq F_T(r)R.
\label{5.12}
\end{align}
Using \eqref{5.11}, \eqref{5.10}, \eqref{5.12}, and
\eqref{eq:v-final-tame}, we conclude that
\begin{equation}\label{5.13}
\begin{split}
\|u\|_{L^2(0,T;\mathbb H^4(\mathcal T))}
&\leq C\Bigl(
T^{1/2}
\|u\|_{C([0,T];\mathbb H^3(\mathcal T))}
+\|v\|_{L^2(0,T;\mathbb H^1(\mathcal T))}
+\|\partial_x(u\partial_xu)\|_
 {L^2(0,T;\mathbb L^2(\mathcal T))}
\Bigr)\\
&\leq F_T(r)R.
\end{split}
\end{equation}
Combining \eqref{5.10}, \eqref{5.13}, and
\eqref{eq:v-final-tame}, we finally obtain
\[\|u\|_{\mathbb B_{3,T}}+\|v\|_{\mathcal X_{0,T}}\leq F_T(r)R,\]
which is precisely \eqref{rrrrrr}.

The same a priori bound holds on any initial interval $[0,t]$,
$0<t\leq T$, on which the higher-regularity solution exists, with
$r$ and $R$ computed from the prescribed data on $(0,T)$ and with
the same function $F_T$. To see this without varying the horizon
in the linear estimate, let $u_E\in\mathcal X_{0,T}$ be the
global energy solution given by Theorem~\ref{global s=0}.
Energy-level uniqueness identifies the higher-regularity solution
with $u_E$ on $[0,t]$.
Apply Lemma~\ref{lemma 5.1} on the fixed interval $(0,T)$ with
$a=u_E$, $f=0$, differentiated boundary data $(g_0',g',p')$, and
initial datum $-\partial_xu^0-\partial_x^3u^0-u^0\partial_xu^0$.
It gives a solution $V\in\mathcal X_{0,T}$ satisfying
$\|V\|_{\mathcal X_{0,T}}\leq F_T(r)R$.
The uniqueness part of that lemma, restricted to $(0,t)$,
identifies $V=\partial_tu$ there.
The estimate for $G$ uses $G(0)$ and its time derivative in $L^1$,
so it is likewise uniform on $(0,t)$. The graph-norm estimate is
spatial, and all remaining time integrals can be bounded using
$t\leq T$. Thus
\[
 \|u\|_{\mathbb B_{3,t}}
 +\|\partial_tu\|_{\mathcal X_{0,t}}
 \leq F_T(r)R
\]
throughout the interval of higher regularity.
\end{proof}

The next result gives global well-posedness at the $H^3$ level.
\begin{theorem}\label{global s=3}
Let $T>0$ and $\lambda\geq1$, and let $d=(u^0,g_0,g,p)\in\mathbb X_{3,T}^{\lambda}$.
Then problem \eqref{kdV} has a unique solution
$u\in\mathcal X_{3,T}$. Moreover, there exists a positive
nondecreasing function
\[
F_T:[0,\infty)\longrightarrow(0,\infty),
\]
depending only on the fixed parameters $T,L,N,\alpha$, and $\lambda$,
such that
\begin{equation}\label{eq:global-H3-tame}
\|u\|_{\mathcal X_{3,T}}
\leq
F_T\bigl(\|d\|_{0,\lambda}\bigr)\|d\|_{3,\lambda}.
\end{equation}
\end{theorem}
\begin{proof}
Set $r:=\|d\|_{0,\lambda}$ and $R:=\|d\|_{3,\lambda}$. Theorem~\ref{global s=0} gives a unique global energy solution $u\in\mathcal X_{0,T}$ satisfying
\[\|u\|_{\mathcal X_{0,T}}\leq F_T(r)r.\]
Theorem~\ref{main}, with $s=3$, gives a solution in $\mathcal X_{3,\tau}$ for some $\tau>0$. By energy-level uniqueness, it agrees with $u$ on $[0,\tau]$. Define
\[\tau_*:=\sup\{t\in(0,T]:u|_{[0,t]}\in\mathcal X_{3,t}\}.\]
Then $\tau_*>0$, and $u|_{[0,t]}\in\mathcal X_{3,t}$ for every
$0<t<\tau_*$. The uniform initial-interval estimate proved at
the end of Proposition~\ref{energy_s=3} gives
\begin{equation*}
 \|u\|_{\mathbb B_{3,t}}+\|\partial_tu\|_{\mathcal X_{0,t}}
 \leq F_T(r)R,\qquad 0<t<\tau_*.
\end{equation*}
In particular,
\begin{equation*}
 \sup_{0\leq s\leq t}\|u(s)\|_{\mathbb H^3}
 \leq F_T(r)R.
\end{equation*}
The uniformity here uses the global energy solution as a coefficient in the linearized problem on the fixed interval $(0,T)$, as in that proposition.
We next justify a restart time independent of $t_0<\tau_*$. Since $\lambda\geq1$, the boundary data belong to the local $H^3$ data spaces
\[g_0\in H^{2/3}(0,T),\qquad g\in[H^1(0,T)]^N,\qquad p\in[H^{4/3}(0,T)]^N,\]
with the sum of these norms bounded by $C_TR$. Choose bounded linear Sobolev extension operators to $\mathbb R$ in these three spaces, and denote the resulting extensions by $G_0,G,P$. Thus
\[\|G_0\|_{H^{2/3}(\mathbb R)}+\|G\|_{[H^1(\mathbb R)]^N}+\|P\|_{[H^{4/3}(\mathbb R)]^N}\leq C_TR.\]
These are Sobolev extensions, not zero extensions. For every $0\leq t_0<T$, define on the fixed interval $(0,T)$
\[g_0^{t_0}(s):=G_0(t_0+s),\qquad g^{t_0}(s):=G(t_0+s),\qquad p^{t_0}(s):=P(t_0+s).\]
Translation invariance and restriction give
\begin{equation*}
 \|g_0^{t_0}\|_{H^{2/3}(0,T)}
 +\|g^{t_0}\|_{[H^1(0,T)]^N}
 +\|p^{t_0}\|_{[H^{4/3}(0,T)]^N}
 \leq C_TR,
\end{equation*}
uniformly in $t_0$. The extensions agree with the original data
throughout the part of the restarted interval lying before $T$. For $t_0<\tau_*$, all the compatibility conditions hold for the restarted data. Indeed, choose $t\in(t_0,\tau_*)$; on $[0,t]$
we have $u\in C_t\mathbb H_x^3$, and the boundary data $g_0,g,p$ are continuous in time. Hence the boundary conditions hold at $t_0$, giving
\[\begin{gathered}
 u_j(t_0,0)=u_1(t_0,0),\qquad
 u_j(t_0,L)=p_j^{t_0}(0),\qquad
 \partial_xu_j(t_0,L)=g_j^{t_0}(0),\\
 \sum_{j=1}^N\partial_x^2u_j(t_0,0)
 =-\alpha u_1(t_0,0)-\frac N3u_1(t_0,0)^2+g_0^{t_0}(0).
 \end{gathered}\]
The sum of the local $H^3$ data norms at every such restart is
bounded by $B:=F_T(r)R+C_TR$. The uniform local result on bounded compatible data sets therefore provides a common time $0<\theta\leq T$, depending only on $B$ and the fixed parameters, for all these restarted problems. Suppose that $\tau_*<T$. Choose $t_0<\tau_*$ so close to $\tau_*$ that $t_0+\theta>\tau_*$. The restarted solution exists
on $[t_0,t_0+\theta]$ and coincides with the energy solution on $[t_0,\min\{T,t_0+\theta\}]$ by uniqueness. The old and restarted higher-regularity solutions overlap on a nonempty time interval. They can therefore be joined in $\mathcal X_3$ using a smooth partition of unity in time
subordinate to the overlap. Multiplication by these smooth time cutoffs is bounded in each of the temporal Sobolev spaces appearing in $\mathcal X_3$, as well as in its other components. This proves higher regularity beyond $\tau_*$, a contradiction. Thus $\tau_*=T$. We still have to include the terminal time. Choose $t_0<T$ with $T-t_0<\theta$. The same restart construction gives a solution in $\mathcal X_3$ on an interval containing $[t_0,T]$, agreeing there with $u$. Join it to the previously known solution on $[0,t_1]$, where $t_0<t_1<T$, using their overlap as above. It follows that $u\in\mathcal X_{3,T}$, including continuity
into $\mathbb H^3$ at $T$. Notice that this overlap argument does not use Lemma~\ref{lemma:finite-Sobolev-gluing} for the
temporal exponents $2/3$ and $4/3$, which are outside its range. Uniqueness follows from uniqueness in $\mathcal X_{0,T}$. It remains to estimate the full norm. Set $v=\partial_tu$. Proposition~\ref{energy_s=3}, now on the whole interval, gives
\begin{equation}\label{eq:global-H3-B3-v}
\|u\|_{\mathbb B_{3,T}}+\|v\|_{\mathcal X_{0,T}}\leq F_T(r)R.
\end{equation}
Consequently,
\begin{equation}\label{eq:global-H3-time-H1}
 \|u\|_{H^1(0,T;\mathbb H^1)}\leq C\bigl(\|u\|_{L^2(0,T;\mathbb H^1)}+\|v\|_{L^2(0,T;\mathbb H^1)}\bigr)\leq F_T(r)R.
\end{equation}
For $-1/2<\sigma<1/2$, let $J_\sigma$ denote the bounded primitive operator
\[J_\sigma:H^\sigma(0,T)\longrightarrow H^{\sigma+1}(0,T),\qquad \partial_tJ_\sigma h=h,\qquad (J_\sigma h)(0)=0.\]
For negative $\sigma$ this is the distributional primitive; the normalization is meaningful because $\sigma+1>1/2$. For $k=0,1,2$, put $\sigma_k=(1-k)/3$. The identity $u(t)=u^0+\int_0^t v(s)\,ds$ and commutation with spatial derivatives give, for almost every $x\in(0,L)$,
\[\partial_x^ku_j(\cdot,x)=\partial_x^ku_j^0(x)+J_{\sigma_k}\bigl(\partial_x^kv_j(\cdot,x)\bigr).\]
Boundedness of the primitive gives
\[\|\partial_x^ku_j(\cdot,x)\|_{H^{(4-k)/3}(0,T)}\leq C_T\left(|\partial_x^ku_j^0(x)|+\|\partial_x^kv_j(\cdot,x)\|_{H^{(1-k)/3}(0,T)}\right).\]
Take the essential supremum in $x$, use $H^3(0,L)\hookrightarrow C^2([0,L])$, and sum over $j$ and $k$.
We obtain
\begin{equation}\label{eq:global-H3-traces}
\sum_{k=0}^2\sum_{j=1}^N\|\partial_x^ku_j\|_{L_x^\infty(0,L;H_t^{(4-k)/3}(0,T))}\leq C_T\bigl(\|u^0\|_{\mathbb H^3}+\|v\|_{\mathcal X_{0,T}}\bigr)\leq F_T(r)R.
\end{equation}
The finite sum over the edges is equivalent to the product norm used in the definition of $\mathcal X_{3,T}$. Combining \eqref{eq:global-H3-B3-v}, \eqref{eq:global-H3-time-H1}, and \eqref{eq:global-H3-traces} proves \eqref{eq:global-H3-tame}. Finally, continuous dependence follows by iterating the local $H^3$ Lipschitz estimate on compatible data, established in the
proof of Theorem~\ref{main}. On a bounded set of global $H^3$ data, the a priori bound above and the fixed Sobolev extension
operators give a common positive local time. Cover $[0,T]$ by finitely many such local intervals with positive overlaps.
The local difference estimates propagate from one interval to the next through the $\mathbb H^3$ initial values. Smooth
localization on this fixed overlapping cover then gives continuous dependence in the full $\mathcal X_{3,T}$ norm.
\end{proof}

\subsection{Nonlinear interpolation and proof of Theorem~\ref{global_(0,3)}}
We use a nonlinear interpolation argument based on the $K$-functional, following the approach of Bona and Scott
\cite[Section~2]{Bona and Scott 1976}. For completeness, we prove the required interpolation estimate and continuity directly from the endpoint estimates established above.
\begin{proof}[Proof of Theorem~\ref{global_(0,3)}]
Throughout the proof, all constants may depend on $T,L,N,\alpha$, and $\lambda$.
Constants indexed by $M$ may also depend on $M$ and are chosen nondecreasing in that parameter. Let $\mathcal S$ denote the global energy-level data-to-solution map provided by Theorem~\ref{global s=0}, and define $\widehat{\mathcal S}:=\mathcal S\circ\mathcal P$. Set
\[A_0:=\widehat{\mathbb X}_{0,T}^{\lambda}, \qquad A_1:=\widehat{\mathbb X}_{3,T}^{\lambda},
\qquad B_0:=\mathcal X_{0,T}, \qquad B_1:=\mathcal X_{3,T}.\]
These are Banach couples with continuous embeddings
$A_1\hookrightarrow A_0$ and $B_1\hookrightarrow B_0$.
Moreover, $\widehat{\mathcal S}$ is defined on all of $A_0$,
maps $A_1$ into $B_1$, and satisfies
$\widehat{\mathcal S}(0)=0$. Indeed, $\mathcal P(A_0)=\mathcal D_{0,T}^{\lambda}$, whereas
$\mathcal P(A_1)=\mathbb X_{3,T}^{\lambda}$ consists of
$3$-compatible data. The assertions therefore follow from
Theorems~\ref{global s=0} and~\ref{global s=3}, together with
uniqueness at the energy level.

We first establish the endpoint estimates for the lifted map.
For $q\in\{0,3\}$ and
$U\in\widehat{\mathbb X}_{q,T}^{\lambda}$, the definition of
$\mathcal P$ gives
\begin{equation}\label{eq:chart-endpoint-tame}
\|\mathcal P(U)\|_{q,\lambda}
\leq
C\bigl(1+\|U\|_{A_0}\bigr)
\|U\|_{\widehat{\mathbb X}_{q,T}^{\lambda}}.
\end{equation}
To see this, write $U=(u^0,h_0,g,p,a)$ and recall that
\[
\mathcal P(U)
=
\left(u^0,h_0+\frac N3a^2\rho,g,p\right).
\]
The only nonlinear term is $a^2\rho$, and
\[
\|a^2\rho\|_{H^{(q+\lambda-1)/3}(0,T)}
\leq C|a|^2
\leq
C\|U\|_{A_0}
\|U\|_{\widehat{\mathbb X}_{q,T}^{\lambda}}.
\]
Consequently, the endpoint theorems imply
\begin{equation}\label{eq:lifted-endpoint-bounds}
\|\widehat{\mathcal S}(U)\|_{\mathcal X_{q,T}}
\leq
F_{T,q}\bigl(\|U\|_{A_0}\bigr)
\|U\|_{\widehat{\mathbb X}_{q,T}^{\lambda}},
\qquad q\in\{0,3\}.
\end{equation}
The lifted maps are continuous at these endpoints because
$\mathcal P$ is continuous into the corresponding physical
data spaces and the endpoint solution maps are continuous. We next verify the weak Lipschitz estimate on bounded subsets
of $A_0$. Let $U,V\in A_0$, and write
\[
d=\mathcal P(U),\qquad
\widetilde d=\mathcal P(V),\qquad
u=\mathcal S(d),\qquad
\widetilde u=\mathcal S(\widetilde d).
\]
Set $w=u-\widetilde u$ and $z=\tfrac12(u+\widetilde u)$. Then $w$ solves
\[\begin{cases}
\partial_t w_j+\partial_x[(1+z_j)w_j]
+\partial_x^3w_j=0,
& t\in(0,T),\ x\in(0,L),\\
w_j(t,0)=w_1(t,0),
& t\in(0,T),\\
\displaystyle
\sum_{j=1}^N\partial_x^2w_j(t,0)
=
-\alpha w_1(t,0)
-\frac{2N}{3}z_1(t,0)w_1(t,0)
+g_0(t)-\widetilde g_0(t),
& t\in(0,T),\\
w_j(t,L)=p_j(t)-\widetilde p_j(t),
& t\in(0,T),\\
\partial_xw_j(t,L)=g_j(t)-\widetilde g_j(t),
& t\in(0,T),\\
w_j(0,x)=u_j^0(x)-\widetilde u_j^0(x),
& x\in(0,L).
\end{cases}
\]
Applying Lemma~\ref{lemma 5.1} with coefficient $z$, and using
the continuous embeddings of the boundary data spaces, gives
\begin{equation}\label{eq:physical-weak-lipschitz}
\|w\|_{B_0}
\leq
C\mathfrak C_T\bigl(\|z\|_{B_0}\bigr)
\|d-\widetilde d\|_{0,\lambda}.
\end{equation}
Fix $M\geq1$ and suppose that
$\|U\|_{A_0},\|V\|_{A_0}\leq M$.
By \eqref{eq:chart-endpoint-tame},
\[
\|\mathcal P(U)\|_{0,\lambda}
+\|\mathcal P(V)\|_{0,\lambda}
\leq C_M.
\]
The energy-level estimate therefore gives
$\|z\|_{B_0}\leq C_{T,M}$.
Furthermore, writing $a$ and $\widetilde a$ for the scalar
components of $U$ and $V$, respectively, we have $|a^2-\widetilde a^2| \leq (|a|+|\widetilde a|)|a-\widetilde a|$. Hence
\[\|\mathcal P(U)-\mathcal P(V)\|_{0,\lambda} \leq C_M\|U-V\|_{A_0}.\]
Combining these estimates with
\eqref{eq:physical-weak-lipschitz}, we obtain
\begin{equation}\label{eq:lifted-low-ball-lipschitz}
\|\widehat{\mathcal S}(U)-\widehat{\mathcal S}(V)\|_{B_0}
\leq C_{T,M}\|U-V\|_{A_0},
\qquad
\|U\|_{A_0},\|V\|_{A_0}\leq M.
\end{equation}
Also, \eqref{eq:lifted-endpoint-bounds} with $q=3$ yields
\begin{equation}\label{eq:lifted-high-ball-tame}
\|\widehat{\mathcal S}(V)\|_{B_1}
\leq C_{T,M}\|V\|_{A_1},
\qquad
V\in A_1,\quad \|V\|_{A_0}\leq M.
\end{equation}
Now fix $0<s<3$ and put $\vartheta=s/3$.
By definition,
\begin{equation*}\label{eq:input-interpolation-space}
A_\vartheta:=(A_0,A_1)_{\vartheta,2}
=
\widehat{\mathbb X}_{s,T}^{\lambda}.
\end{equation*}
Let $K_A$ and $K_B$ denote the $K$-functionals of the couples
$(A_0,A_1)$ and $(B_0,B_1)$, respectively. For $U\ne0$ with $\|U\|_{A_0}\leq M$ and $\eta>0$, choose
$V\in A_1$ such that
\[
\|U-V\|_{A_0}+\eta\|V\|_{A_1}
\leq 2K_A(\eta,U).
\]
Since $K_A(\eta,U)\leq\|U\|_{A_0}$, this choice satisfies
\[
\|V\|_{A_0}
\leq
\|U\|_{A_0}+\|U-V\|_{A_0}
\leq3M.
\]
Using \eqref{eq:lifted-low-ball-lipschitz} and
\eqref{eq:lifted-high-ball-tame} on the radius-$3M$ ball, we find
\begin{align*}
K_B(\eta,\widehat{\mathcal S}(U))
&\leq
\|\widehat{\mathcal S}(U)-\widehat{\mathcal S}(V)\|_{B_0}
+\eta\|\widehat{\mathcal S}(V)\|_{B_1}\\
&\leq
C_{T,3M}\bigl(
\|U-V\|_{A_0}+\eta\|V\|_{A_1}
\bigr)\\
&\leq C_{T,M}K_A(\eta,U),
\end{align*}
where the constant has been adjusted in the last line.
The same inequality holds for $U=0$ because
$\widehat{\mathcal S}(0)=0$.
Thus
\begin{equation}\label{eq:lifted-K-functional-bound}
K_B(\eta,\widehat{\mathcal S}(U))
\leq C_{T,M}K_A(\eta,U),
\qquad
\eta>0,\quad \|U\|_{A_0}\leq M.
\end{equation}
Multiplying by $\eta^{-\vartheta}$ and taking the
$L^2((0,\infty),d\eta/\eta)$ norm gives
\[
\|\widehat{\mathcal S}(U)\|_{(B_0,B_1)_{\vartheta,2}}
\leq C_{T,M}\|U\|_{A_\vartheta},
\qquad
U\in A_\vartheta,\quad \|U\|_{A_0}\leq M.
\]
Since $A_\vartheta\hookrightarrow A_0$, we conclude that $\widehat{\mathcal S}: \widehat{\mathbb X}_{s,T}^{\lambda}
\longrightarrow (\mathcal X_{0,T},\mathcal X_{3,T})_{\vartheta,2}$ and
\begin{equation*}\label{eq:BS-estimate}
\|\widehat{\mathcal S}(U)\|_
{(\mathcal X_{0,T},\mathcal X_{3,T})_{\vartheta,2}}
\leq
F_{T,s}\bigl(\|U\|_{\widehat{\mathbb X}_{s,T}^{\lambda}}\bigr)
\|U\|_{\widehat{\mathbb X}_{s,T}^{\lambda}}.
\end{equation*}
Here and below, $F_{T,s}$ denotes a positive nondecreasing function, allowed to change from one occurrence to the next. We also have
\begin{equation*}\label{eq:target-embedding}
(\mathcal X_{0,T},\mathcal X_{3,T})_{\vartheta,2}
\hookrightarrow\mathcal X_{s,T}.
\end{equation*}
Indeed, this is the embedding \eqref{eq:local-target-interpolation}, applied on the fixed interval $(0,T)$. Its componentwise proof applies to every $0<\vartheta<1$ and does not involve any compatibility conditions. 
Consequently,
\begin{equation*}\label{eq:lifted-final-estimate}
\|\widehat{\mathcal S}(U)\|_{\mathcal X_{s,T}}
\leq
F_{T,s}\bigl(\|U\|_{\widehat{\mathbb X}_{s,T}^{\lambda}}\bigr)
\|U\|_{\widehat{\mathbb X}_{s,T}^{\lambda}}.
\end{equation*}
We next prove continuity. Let $U_n\to U$ in $A_\vartheta$, and choose $M\geq1$
bounding their $A_0$ norms and that of $U$.
Set
\[W_n:=\widehat{\mathcal S}(U_n)-\widehat{\mathcal S}(U).\]
By subadditivity of the $K$-functional and
\eqref{eq:lifted-K-functional-bound},
\[
K_B(\eta,W_n)
\leq
C_{T,M}\bigl(
2K_A(\eta,U)+K_A(\eta,U_n-U)
\bigr).
\]
Therefore, for every $\delta>0$,
\begin{align*}
\left(\int_0^\delta[\eta^{-\vartheta}K_B(\eta,W_n)]^2\frac{d\eta}{\eta}\right)^{1/2}
\leq C_{T,M}\left[2\left(\int_0^\delta[\eta^{-\vartheta}K_A(\eta,U)]^2\frac{d\eta}{\eta}\right)^{1/2}+
\|U_n-U\|_{A_\vartheta}\right].
\end{align*}
On the other hand,
$K_B(\eta,W_n)\leq\|W_n\|_{B_0}$, so
\eqref{eq:lifted-low-ball-lipschitz} gives
\[\left(\int_\delta^\infty[\eta^{-\vartheta}K_B(\eta,W_n)]^2\frac{d\eta}{\eta}
\right)^{1/2}\leq\frac{C_{T,M}}{\sqrt{2\vartheta}}\,\delta^{-\vartheta}\|U_n-U\|_{A_0}.\]
Letting first $n\to\infty$ and then $\delta\downarrow0$, we obtain
\[\widehat{\mathcal S}(U_n)\longrightarrow\widehat{\mathcal S}(U)\quad\text{in }(B_0,B_1)_{\vartheta,2}.\]
The embedding \eqref{eq:target-embedding} proves continuity in $\mathcal X_{s,T}$. It remains to express the conclusions in terms of the physical data. The preceding argument, together with the endpoint estimates and continuity, shows that for every $0\leq s\leq3$ the lifted map is continuous and satisfies
\[\|\widehat{\mathcal S}(U)\|_{\mathcal X_{s,T}}
\leq
F_{T,s}\bigl(\|U\|_{\widehat{\mathbb X}_{s,T}^{\lambda}}\bigr)
\|U\|_{\widehat{\mathbb X}_{s,T}^{\lambda}}.
\]
Let
\[d\in\mathfrak X_{s,T}^{\lambda}=\mathcal P\bigl(\widehat{\mathbb X}_{s,T}^{\lambda}\bigr),\]
and choose $U\in\widehat{\mathbb X}_{s,T}^{\lambda}$ such that
$\mathcal P(U)=d$. By uniqueness at the energy level, $\widehat{\mathcal S}(U)$ is the solution corresponding to $d$ and is independent of the chosen representation. Uniqueness in $\mathcal X_{s,T}$ also follows from $\mathcal X_{s,T}\hookrightarrow\mathcal X_{0,T}$. If $\langle d\rangle_{s,\lambda}=0$, there exist representations $U_n\to0$ in $\widehat{\mathbb X}_{s,T}^{\lambda}$ with
$\mathcal P(U_n)=d$. The continuous embedding into $\mathcal E^\lambda_{s/3}$ and
the continuity of $\mathcal P$ into $\mathcal D^\lambda_{s,T}$ imply $d=0$. The asserted estimate is then immediate. If $\langle d\rangle_{s,\lambda}>0$, choose a representation
satisfying
\[\|U\|_{\widehat{\mathbb X}_{s,T}^{\lambda}}\leq2\langle d\rangle_{s,\lambda}.\]
Applying the lifted estimate and absorbing the factor $2$ into $F_{T,s}$ yields
\[\|u\|_{\mathcal X_{s,T}}\leq F_{T,s}\bigl(\langle d\rangle_{s,\lambda}\bigr) \langle d\rangle_{s,\lambda}.\]
Finally, suppose that
$s\notin\{1/2,3/2,5/2\}$.
By Lemma~\ref{lem:interpolation-constraints} and the endpoint
identifications,
\[\mathfrak X_{s,T}^{\lambda}=\mathbb X_{s,T}^{\lambda}.\]
For $0\leq s<1/2$, choose $a=0$ and $h_0=g_0$. For $1/2<s\leq3$, choose $a=u_1^0(0)$ and $h_0=g_0-\tfrac N3a^2\rho$. These choices give an admissible representation of $d$. The trace theorem, the smoothness of $\rho$, and the
equivalence of norms in the noncritical constrained spaces
give
\[
\|U\|_{\widehat{\mathbb X}_{s,T}^{\lambda}}
\leq
C_{T,s}\bigl(1+\|d\|_{s,\lambda}\bigr)\|d\|_{s,\lambda}.
\]
Consequently,
\[
\|u\|_{\mathcal X_{s,T}}
\leq
\widetilde F_{T,s}\bigl(\|d\|_{s,\lambda}\bigr)
\|d\|_{s,\lambda}.
\]
At $s=1/2,3/2,5/2$, the admissible data and their size retain
the interpolation definitions
\eqref{eq:physical-interpolated-data} and
\eqref{eq:chart-size}.
This completes the proof.
\end{proof}

\section{Concluding Remarks}\label{sec7}

Motivated by \cite{BSZ 2003,CapistranoParadaDaSilva25}, we have established global well-posedness for the KdV equation on a finite equilateral star graph in the full range $0\leq s\leq3$. The analysis provides a functional framework for the linear and nonlinear problems, based on linear lifting procedures, energy estimates in Sobolev spaces, and a careful treatment of the compatibility conditions at the central vertex.

For $s\notin\left\{\frac12,\frac32,\frac52\right\},$ the admissible data spaces are described by the usual Sobolev
regularity and the corresponding pointwise compatibility conditions. At the three critical indices, however, the admissible spaces are defined through constrained real interpolation and are of Lions--Magenes type. An interesting open problem is to obtain an explicit intrinsic characterization of these critical data spaces, possibly in terms of weighted or nonlocal compatibility conditions. 

Another important open problem is the extension of the present results to non-equilateral star graphs, namely, to the case in which the edge lengths $l_j$ are not all equal. The arguments used here rely on the fact that all edge components are defined on the same interval $(0,L)$. In particular, this allows sums and differences of the components to be treated as solutions of scalar initial-boundary value problems. When the lengths are different, restricting a solution defined on a larger interval to an individual edge does not, in general, preserve the boundary conditions at its right endpoint. Consequently, this case requires a genuinely different linear analysis and is not covered by the present approach.

Further directions include the study of more general network topologies, such as graphs containing loops
\cite{PavaMunoz24}, and the extension of the present methods to other dispersive models, including the Kawahara equation
\cite{CavalcanteKwakMarques25}.

Finally, as mentioned in the introduction, several control-theoretic
questions arise naturally from this analysis. These include
controllability and stabilization problems associated with boundary
controls acting at the external vertices or at the central junction.
Such questions are not addressed in the present article and will be
the subject of future work.

\subsection*{Acknowledgments}This paper was finalized during the first author’s visit to several universities in France in 2025, thanks to the support of the CAPES–COFECUB program. The first author warmly thanks the Université de Lorraine, the Université de Tours, and the Université Sorbonne Paris Nord for their kind hospitality. The second author gratefully acknowledges the support of the Department of Mathematics at the Universidade Federal de Pernambuco, where this work was initiated, as well as the financial support provided by the Réseau Mathématique Franco-Brésilien during this visit.

\appendix
\section{Interpolation of the nonlinear compatibility conditions}
\label{app:compatibility-interpolation}

We continue under the standing assumption that the star graph is
equilateral. For $0\leq\vartheta\leq1$ and $\lambda\geq0$, set
\begin{align*}
\mathcal E^\lambda_\vartheta:={}&
[H^{3\vartheta}(0,L)]^N
\times H^{\vartheta+(\lambda-1)/3}(0,T)
\times[H^\vartheta(0,T)]^N\times[H^{\vartheta+(\lambda+1)/3}(0,T)]^N
\times\mathbb R,
\end{align*}
endowed with its natural product norm. An element of this space will
be denoted by $U=(u^0,h_0,g,p,a)$. Throughout the construction, we fix $\rho(t)\equiv1$ on $[0,T]$. Define the nonlinear parametrization
\begin{equation}\label{eq:nonlinear-chart}
\mathcal P(u^0,h_0,g,p,a)
:=
\left(
u^0,h_0+\frac N3a^2\rho,g,p
\right).
\end{equation}
For every $0\leq\vartheta\leq1$, this defines a locally Lipschitz map
\[
\mathcal P:
\mathcal E^\lambda_\vartheta
\longrightarrow
\mathcal D^\lambda_{3\vartheta,T}.
\]
On $\mathcal E^\lambda_1$, consider the following bounded linear
functionals:
\begin{align*}
\Lambda_{V,j}(U) &=u_j^0(0)-a, &&j=1,\ldots,N,\\
\Lambda_{D,j}(U) &=u_j^0(L)-p_j(0), &&j=1,\ldots,N,\\
\Lambda_{N,j}(U) &=\partial_xu_j^0(L)-g_j(0), &&j=1,\ldots,N,\\
\Lambda_K(U) &=\sum_{j=1}^N\partial_x^2u_j^0(0)+\alpha a-h_0(0).
\end{align*}
Their boundedness follows from the trace theorem, since
\[u^0\in[H^3(0,L)]^N,\qquad h_0\in H^{(\lambda+2)/3}(0,T),\]
\[g\in[H^1(0,T)]^N,\qquad p\in[H^{(\lambda+4)/3}(0,T)]^N,\]
and $\lambda\geq0$. Let $\Lambda:\mathcal E^\lambda_1\longrightarrow\mathbb R^{3N+1}$
denote the vector formed by these functionals, and put
\begin{equation*}
\widehat{\mathbb X}^\lambda_{0,T}=\mathcal E^\lambda_0,\qquad\widehat{\mathbb X}^\lambda_{3,T}=\ker\Lambda,
\end{equation*}
where $\ker\Lambda$ is endowed with the norm induced by $\mathcal E^\lambda_1$. Since $\Lambda$ is bounded, $\ker\Lambda$ is a closed subspace of $\mathcal E^\lambda_1$.
Moreover,
\[\mathcal E^\lambda_1\hookrightarrow\mathcal E^\lambda_0.\]
Consequently, $\left(\widehat{\mathbb X}^\lambda_{0,T},\widehat{\mathbb X}^\lambda_{3,T}
\right)$ is a compatible couple of Hilbert spaces. Moreover, for $0<s<3$, we define
\begin{equation*}
\widehat{\mathbb X}^\lambda_{s,T}:=\left(\widehat{\mathbb X}^\lambda_{0,T},\widehat{\mathbb X}^\lambda_{3,T}\right)_{s/3,2}.
\end{equation*}
The purpose of the extra scalar $a$ and the modified datum $h_0$ is to linearize all the compatibility conditions at the upper endpoint. Indeed, if $U\in\widehat{\mathbb X}^\lambda_{3,T}$ and $g_0=h_0+\frac N3a^2\rho$, then
\[u_j^0(0)=a,\qquad u_j^0(L)=p_j(0),\qquad \partial_xu_j^0(L)=g_j(0),\qquad j=1,\ldots,N,\]
and
\begin{align*}
\sum_{j=1}^N\partial_x^2u_j^0(0)&=-\alpha a+h_0(0)=-\alpha a-\frac N3a^2+g_0(0).
\end{align*}
Thus, $\mathcal P(U)$ satisfies all the $3$-compatibility conditions, including the nonlinear compatibility condition at the central vertex. Conversely, if $d=(u^0,g_0,g,p)\in\mathbb X^\lambda_{3,T}$, then, setting $a=u_1^0(0)$ and $h_0=g_0-\frac N3a^2\rho$, we obtain
\[U=(u^0,h_0,g,p,a)\in\widehat{\mathbb X}^\lambda_{3,T},\qquad\mathcal P(U)=d.\]
Therefore,
\begin{equation*}
\mathcal P\bigl(\widehat{\mathbb X}^\lambda_{3,T}\bigr)=\mathbb X^\lambda_{3,T}.
\end{equation*}
At the lower endpoint, no compatibility condition is imposed, and by taking $a=0$ and $h_0=g_0$, we similarly have
\begin{equation*}
\mathcal P\bigl(\widehat{\mathbb X}^\lambda_{0,T}\bigr)=\mathcal D^\lambda_{0,T}.
\end{equation*}
We are now in a position to state the interpolation result for the
graph constraints.
\begin{lemma}[Identification of the constrained interpolation spaces]
\label{lem:interpolation-constraints}
Let $\lambda\geq0$ and $0<\vartheta<1$ with
\[
\vartheta\notin
\left\{\frac16,\frac12,\frac56\right\}.
\]
Then, with equivalence of norms,
\begin{equation}\label{eq:constraint-identification}
\begin{aligned}
\left(
\widehat{\mathbb X}^\lambda_{0,T},
\widehat{\mathbb X}^\lambda_{3,T}
\right)_{\vartheta,2}=
\left\{
U\in\mathcal E^\lambda_\vartheta:
\begin{array}{ll}
\Lambda_{V,j}(U)=\Lambda_{D,j}(U)=0,
&\vartheta>\frac16,\quad j=1,\ldots,N,\\[1mm]
\Lambda_{N,j}(U)=0,
&\vartheta>\frac12,\quad j=1,\ldots,N,\\[1mm]
\Lambda_K(U)=0,
&\vartheta>\frac56
\end{array}
\right\}.
\end{aligned}
\end{equation}
\end{lemma}
\begin{proof}
The usual real interpolation identities for Sobolev spaces, together with the interpolation of finite products, give
\begin{equation}\label{eq:ambient-interpolation}
\left(\mathcal E^\lambda_0,\mathcal E^\lambda_1\right)_{\vartheta,2}=\mathcal E^\lambda_\vartheta
\end{equation}
with equivalence of norms. Moreover, $\mathcal E^\lambda_1$ is continuously and densely embedded into $\mathcal E^\lambda_0$. We apply the constrained real interpolation theorem of
L\"ofstr\"om \cite[Theorem~2, p.~35]{LofstromConstraints}.
We verify the required strong independence using
\cite[Definitions~1--2 and Lemma~1, pp.~33--35]{LofstromConstraints}. For
$0<\tau\leq1$ and $U\in\mathcal E^\lambda_1$, set
\begin{align*}
\|U\|_\tau
:=
\left(
\|U\|_{\mathcal E^\lambda_0}^2
+
\tau^2\|U\|_{\mathcal E^\lambda_1}^2
\right)^{1/2},
\end{align*}
and, for a linear functional $\ell$ on
$\mathcal E^\lambda_1$, define the functional\footnote{L\"ofstr\"om uses the functional $N(\tau,\ell)$ associated with $J(\tau,U):=\max\left\{\|U\|_{\mathcal E^\lambda_0},\tau\|U\|_{\mathcal E^\lambda_1}
\right\}$. Since
\[J(\tau,U)\leq\|U\|_\tau\leq\sqrt{2}\,J(\tau,U),\]
our functional $I(\tau,\ell)$ is uniformly equivalent to $N(\tau,\ell)$. Thus the orders of the functionals and the strong independence criterion are unchanged.}
\begin{align*}
I(\tau,\ell)
:=
\sup_{U\in\mathcal E^\lambda_1\setminus\{0\}}
\frac{|\ell(U)|}{\|U\|_\tau}.
\end{align*}
For $m=0,1,2$, put
\[
\sigma_m:=\frac{m+\frac12}{3}.
\]
The one-dimensional multiplicative trace inequality gives, for
$w\in H^3(0,L)$ and $m=0,1,2$,
\begin{align*}
\left|\partial_x^m w(0)\right|
+\left|\partial_x^m w(L)\right|
\leq
C\left(
\|w\|_{L^2(0,L)}^{1-\sigma_m}
\|w\|_{H^3(0,L)}^{\sigma_m}
+
\|w\|_{L^2(0,L)}
\right).
\end{align*}
Consequently,
\begin{align*}
\left|\partial_x^m w(0)\right|
+\left|\partial_x^m w(L)\right|
\leq
C\tau^{-\sigma_m}
\left(
\|w\|_{L^2(0,L)}^2
+
\tau^2\|w\|_{H^3(0,L)}^2
\right)^{1/2}.
\end{align*}
We also use the time-trace estimate
\begin{align*}
|z(0)|\leq C\left(\|z\|_{L^2(0,T)}^{1/2}\|z\|_{H^1(0,T)}^{1/2}+\|z\|_{L^2(0,T)}\right).
\end{align*}
To include the local data scale $\lambda=0$, we also use the following weighted trace estimate. For $b\in\{-1/3,0,1/3\}$ and $0<\tau\leq1$,
\begin{equation}\label{eq:weighted-time-trace-general}
|z(0)|\leq C_b\tau^{-(1/2-b)}\left(\|z\|_{H^b(0,T)}^2+\tau^2\|z\|_{H^{b+1}(0,T)}^2\right)^{1/2},\qquad z\in H^{b+1}(0,T).
\end{equation}
Here $\tau$ is the interpolation parameter, not the length of the fixed time interval $(0,T)$. Indeed, choose a common bounded extension operator for the two Sobolev orders on this fixed interval and apply
the Fourier inversion formula and Cauchy--Schwarz. With $\langle\xi\rangle=(1+|\xi|^2)^{1/2}$, the relevant integral satisfies
\[\int_{\mathbb R}\frac{d\xi}{\langle\xi\rangle^{2b}(1+\tau^2\langle\xi\rangle^2)}\leq C_b\tau^{2b-1}.\]
This follows by splitting at $|\xi|=\tau^{-1}$, since $-1/2<b<1/2$. For every $\lambda\geq0$, the endpoint norms of $\mathcal E^\lambda_0$ and $\mathcal E^\lambda_1$ control the Sobolev pairs
\[(H^{1/3},H^{4/3})\quad\text{for }p_j,\qquad (L^2,H^1)\quad\text{for }g_j,\qquad
(H^{-1/3},H^{2/3})\quad\text{for }h_0,\]
respectively. Thus, with the weighted endpoint norm $\|U\|_\tau$ already defined, \eqref{eq:weighted-time-trace-general} gives
\[|p_j(0)|\leq C\tau^{-1/6}\|U\|_\tau,\qquad |g_j(0)|\leq C\tau^{-1/2}\|U\|_\tau,\qquad
|h_0(0)|\leq C\tau^{-5/6}\|U\|_\tau.\]
The scalar component $a$ is controlled directly by $\|U\|_{\mathcal E^\lambda_0}$. Consequently, the upper bounds for $I(\tau,\Lambda_{V,j})$, $I(\tau,\Lambda_{D,j})$, $I(\tau,\Lambda_{N,j})$, and $I(\tau,\Lambda_K)$ remain exactly the same. The spatial biorthogonal liftings used below are unchanged, so the constrained interpolation identification holds also for
$\lambda=0$. It follows from these estimates that
\begin{equation}\label{eq:I-upper-constraints}
\begin{cases}
I(\tau,\Lambda_{V,j}) +I(\tau,\Lambda_{D,j}) &\leq C\tau^{-1/6},\\
I(\tau,\Lambda_{N,j}) &\leq C\tau^{-1/2},\\
I(\tau,\Lambda_K) &\leq C\tau^{-5/6}.
\end{cases}
\end{equation}
Thus the possible orders of the functionals are respectively
\[\sigma_{V,j}=\sigma_{D,j}=\frac16,\qquad \sigma_{N,j}=\frac12,
\qquad\sigma_K=\frac56.\]
We now prove the corresponding lower estimates and the strong
independence of the family. Choose
$\chi\in C^\infty_c([0,\infty))$ such that $\chi=1$ in a
neighbourhood of $0$. For all sufficiently small $\varepsilon>0$,
define
\begin{align*}
\psi_{V,\varepsilon}(x) &= \chi\left(\frac{x}{\varepsilon}\right), & \psi_{K,\varepsilon}(x)
&=\frac{x^2}{2} \chi\left(\frac{x}{\varepsilon}\right),\\
\psi_{D,\varepsilon}(x) &= \chi\left(\frac{L-x}{\varepsilon}\right), & \psi_{N,\varepsilon}(x)
&=(x-L) \chi\left(\frac{L-x}{\varepsilon}\right).
\end{align*}
They satisfy
\begin{equation*}
\begin{cases}
\psi_{V,\varepsilon}(0)&=1,\qquad \partial_x^2\psi_{V,\varepsilon}(0)=0,\\
\psi_{K,\varepsilon}(0)&=0,\qquad \partial_x^2\psi_{K,\varepsilon}(0)=1,\\
\psi_{D,\varepsilon}(L)&=1,\qquad \partial_x\psi_{D,\varepsilon}(L)=0,\\
\psi_{N,\varepsilon}(L)&=0,\qquad \partial_x\psi_{N,\varepsilon}(L)=1.
\end{cases}
\end{equation*}
Moreover, the functions supported near $0$ vanish in a neighborhood of $L$, while those supported near $L$ vanish in a neighborhood of $0$. For each $j=1,\ldots,N$, insert $\psi_{V,\varepsilon}$, $\psi_{D,\varepsilon}$, and $\psi_{N,\varepsilon}$, respectively, into the $j$-th component of $u^0$, with all the other components of $U$ equal to zero. For $\Lambda_K$, insert $\psi_{K,\varepsilon}$ into the first component of $u^0$. In this way we obtain elements $R_{V,j,\varepsilon}$, $R_{D,j,\varepsilon}$, $R_{N,j,\varepsilon}$, and $ R_{K,\varepsilon}\in\mathcal E^\lambda_1$ such that
\begin{equation}\label{eq:biorthogonal-liftings}
\Lambda_\ell(R_{\ell',\varepsilon})=\delta_{\ell,\ell'}.
\end{equation}
In particular, each $R_{\ell,\varepsilon}$ belongs to the kernels of all the other functionals. If
\[m_\ell=
\begin{cases}
0,&\ell=(V,j)\ \text{or}\ \ell=(D,j),\\
1,&\ell=(N,j),\\
2,&\ell=K,
\end{cases}\]
then the usual scaling estimates give
\begin{equation}\label{eq:jet-scaling-endpoints}
\begin{split}
\|R_{\ell,\varepsilon}\|_{\mathcal E^\lambda_0}
&\leq C\varepsilon^{m_\ell+\frac12},\\
\|R_{\ell,\varepsilon}\|_{\mathcal E^\lambda_1}
&\leq C\varepsilon^{m_\ell-\frac52}.
\end{split}
\end{equation}
Fix $\varepsilon_0\in(0,1)$ sufficiently small so that
\eqref{eq:biorthogonal-liftings} holds whenever
$0<\varepsilon\leq\varepsilon_0$, and set
\[\varepsilon(\tau):=\min\{\varepsilon_0,\tau^{1/3}\},\qquad 0<\tau\leq1.\]
Using \eqref{eq:jet-scaling-endpoints}, we obtain
\begin{align}\label{eq:jet-weighted-estimate}
\|R_{\ell,\varepsilon(\tau)}\|_\tau \leq C\tau^{\sigma_\ell}, \qquad 0<\tau\leq1.
\end{align}
Indeed, for $\tau\leq\varepsilon_0^3$ this follows from the scaling estimates, while for
$\tau\geq\varepsilon_0^3$ it follows by using the fixed lifting $R_{\ell,\varepsilon_0}$ and adjusting the constant. Let $\mathcal I$ denote the finite index set of all the functionals in $\Lambda$. If $c=(c_\ell)_{\ell\in\mathcal I}$, then \eqref{eq:biorthogonal-liftings} and \eqref{eq:jet-weighted-estimate} imply
\begin{align*}
I\left(
\tau,\sum_{\ell\in\mathcal I}c_\ell\Lambda_\ell
\right)
&\geq
\frac{
\left|
\sum_{\ell'\in\mathcal I}
c_{\ell'}\Lambda_{\ell'}
(R_{\ell,\varepsilon(\tau)})
\right|
}{
\|R_{\ell,\varepsilon(\tau)}\|_\tau
}\geq
C^{-1}|c_\ell|\tau^{-\sigma_\ell}.
\end{align*}
Taking the maximum over $\ell$ and using
\eqref{eq:I-upper-constraints}, we conclude that
\begin{align}\label{eq:strong-independence-estimate}
C^{-1}
\max_{\ell\in\mathcal I}
|c_\ell|\tau^{-\sigma_\ell}
&\leq
I\left(
\tau,\sum_{\ell\in\mathcal I}c_\ell\Lambda_\ell
\right)\leq
C
\sum_{\ell\in\mathcal I}
|c_\ell|\tau^{-\sigma_\ell}.
\end{align}
Since the family is finite, the maximum and sum in
\eqref{eq:strong-independence-estimate} are equivalent uniformly in
the coefficients. Moreover, because each supporting element in
\eqref{eq:biorthogonal-liftings} belongs to the kernels of all the
other functionals, the same estimates remain valid after restricting
successively to the kernels of the preceding functionals. Hence the functionals form a strongly independent basis in the
sense of \cite[Definition~1 and Lemma~1, pp.~33--35]
{LofstromConstraints}, with orders
$\frac16$, $\frac12$, and $\frac56$.

We may therefore apply the constrained interpolation theorem, ordering
the functionals according to their nondecreasing orders. For
$\vartheta\notin\{1/6,1/2,5/6\}$, it gives
\begin{align*}
\left(
\mathcal E^\lambda_0,
\mathcal E^\lambda_1\cap\ker\Lambda
\right)_{\vartheta,2}
=
\left\{
U\in\mathcal E^\lambda_\vartheta:
\Lambda_\ell(U)=0
\ \text{whenever}\ 
\vartheta>\sigma_\ell
\right\},
\end{align*}
with equivalence of norms. In view of
\eqref{eq:ambient-interpolation} and the definitions of
$\widehat{\mathbb X}^\lambda_{0,T}$ and
$\widehat{\mathbb X}^\lambda_{3,T}$, this is precisely
\eqref{eq:constraint-identification}.
\end{proof}

\begin{remark}[The three critical indices]
At the critical values $\vartheta\in\left\{\frac16,\frac12,\frac56\right\}$,
the constrained interpolation theorem does not yield the
ambient-space characterization
\eqref{eq:constraint-identification}. We therefore retain the
interpolation definition 
$\widehat{\mathbb X}^{\lambda}_{3\vartheta,T}
=
\bigl(
\widehat{\mathbb X}^{\lambda}_{0,T},
\widehat{\mathbb X}^{\lambda}_{3,T}
\bigr)_{\vartheta,2}
$
at these three values. These critical constrained interpolation spaces are of Lions--Magenes type. More precisely, they play, for the present finite family of endpoint constraints, the same role as the classical critical space
\[H_{00}^{1/2}(0,L)=\bigl(L^2(0,L),H_0^1(0,L)\bigr)_{1/2,2};\]
see \cite[Chapter~1, Theorem~11.7, p.~66]{LionsMagenes}. We do not need, and do not assert here, an explicit weighted
characterization of the corresponding critical spaces. Consequently, the interpolation space at $s=1/2,3/2,5/2$ must be understood through the preceding interpolation definition. In particular, it cannot be replaced by the full ambient Sobolev space using the rule that a trace condition is imposed only when the corresponding classical trace exists.
\end{remark}



\end{document}